\documentclass[a4paper, 11pt, twoside]{article}

\usepackage{amsmath, amscd, amsfonts, amssymb, amsthm, latexsym, url, color, todonotes, bm, framed}
\usepackage{graphicx}
\usepackage[left=1.1in, right=1.1in, top=1.3in, bottom=0.8in, includefoot, headheight=13.6pt]{geometry}
\input{xy}
\xyoption{all}

\usepackage[ps2pdf=true,colorlinks,breaklinks=true]{hyperref}
\usepackage[figure,table]{hypcap}
\hypersetup{
	bookmarksnumbered,
	pdfstartview={FitH},
	citecolor={black},
	linkcolor={black},
	urlcolor={black},
	pdfpagemode={UseOutlines}
}
\makeatletter
\newcommand\org@hypertarget{}
\let\org@hypertarget\hypertarget
\renewcommand\hypertarget[2]{%
  \Hy@raisedlink{\org@hypertarget{#1}{}}#2%
} 
\makeatother 

\newtheorem{theorem}{Theorem}[section]
\newtheorem{lemma}[theorem]{Lemma}
\newtheorem{corollary}[theorem]{Corollary}
\newtheorem{proposition}[theorem]{Proposition}

\theoremstyle{definition}
\newtheorem{definition}[theorem]{Definition}
\newtheorem{remark}[theorem]{Remark}
\newtheorem{example}[theorem]{Example}

\newcommand{\xysquare}[8]{
\[\xymatrix{
#1 \ar@{#5}[r] \ar@{#6}[d] & #2 \ar@{#7}[d]\\
#3 \ar@{#8}[r] & #4
}\]
}

\newcommand{\al}{\alpha}
\newcommand{\bb}{\mathbb}

\newcommand{\blob}{\bullet}

\newcommand{\comment}[1]{}

\newcommand{\ep}{\varepsilon}

\newcommand{\into}{\hookrightarrow}
\newcommand{\isoto}{\stackrel{\simeq}{\to}}
\newcommand{\Isoto}{\stackrel{\simeq}{\longrightarrow}}

\newcommand{\onto}{\twoheadrightarrow}
\newcommand{\op}{\operatorname}
\newcommand{\pid}[1]{\langle #1 \rangle}
\newcommand{\quis}{\stackrel{\sim}{\to}}
\newcommand{\res}{\overline}
\newcommand{\roi}{\mathcal{O}}

\newcommand{\sub}[1]{{\mbox{\scriptsize #1}}}

\newcommand{\To}{\longrightarrow}
\newcommand{\ul}[1]{\underline{#1}}

\newcommand{\xto}{\xrightarrow}

\newcommand{\TR}{T\!R}
\newcommand{\TC}{T\!C}
\newcommand{\CH}{C\!H}

\renewcommand{\cal}{\mathcal}
\renewcommand{\hat}{\widehat}
\renewcommand{\frak}{\mathfrak}

\renewcommand{\tilde}{\widetilde}
\renewcommand{\Im}{\operatorname{Im}}
\renewcommand{\ker}{\operatorname{Ker}}
\renewcommand{\projlim}{\varprojlim}

\DeclareMathOperator{\Pic}{Pic}

\DeclareMathOperator{\Spec}{Spec}
\DeclareMathOperator{\Spf}{Spf}

\DeclareMathOperator{\Tor}{Tor}

\newcommand{\dlog}{\op{dlog}}
\newcommand{\sinfty}{{\scalebox{0.5}{$\infty$}}}

\newcommand{\dotimes}{\otimes^{\bb L}}
\newcommand{\xTo}[1]{\stackrel{#1}{\To}}

\usepackage{fancyhdr}

\usepackage{sectsty}
\sectionfont{\Large\sc\centering}
\subsectionfont{\large}
\partfont{\LARGE\sc\centering}
\chapterfont{\large\sc\centering}
\chaptertitlefont{\LARGE\centering}

\begin{document}

\title{$K$-theory and logarithmic Hodge--Witt sheaves of formal schemes in characteristic $p$}

\author{Matthew Morrow}

\date{}

\maketitle

\begin{abstract}
We describe the mod $p^r$ pro $K$-groups $\{K_n(A/I^s)/p^r\}_s$ of a regular local $\bb F_p$-algebra $A$ modulo powers of a suitable ideal $I$, in terms of logarithmic Hodge--Witt groups, by proving pro analogues of the theorems of Geisser--Levine and Bloch--Kato--Gabber. This is achieved by combining the pro Hochschild--Kostant--Rosenberg theorem in topological cyclic homology with the development of the theory of de Rham--Witt complexes and logarithmic Hodge--Witt sheaves on formal schemes in characteristic $p$.

Applications include the following: the infinitesimal part of the weak Lefschetz conjecture for Chow groups; a $p$-adic version of Kato--Saito's conjecture that their Zariski and Nisnevich higher dimensional class groups are isomorphic; continuity results in $K$-theory; and criteria, in terms of integral or torsion \'etale-motivic cycle classes, for algebraic cycles on formal schemes to admit infinitesimal deformations.

Moreover, in the case $n=1$, we compare the \'etale cohomology of $W_r\Omega^1_\text{\scriptsize{log}}$ and the fppf cohomology of $\bm\mu_{p^r}$ on a formal scheme, and thus present equivalent conditions for line bundles to deform in terms of their classes in either of these cohomologies.
\end{abstract}

\tableofcontents

\setcounter{section}{-1}
\section{Introduction}
\subsection{$K$-theory}
The primary goal of this article is to extend results concerning the $K$-theory and motivic cohomology of smooth varieties in characteristic $p$ to the case of regular formal schemes. If $A$ is an $\bb F_p$-algebra, then we consider the natural homomorphisms \begin{equation}K_n(A)/p^r\longleftarrow K_n^M(A)/p^r\xTo{\dlog[\cdot]}W_r\Omega_{A,\sub{log}}^n,\label{eqn:1}\end{equation} where $W_r\Omega_{A,\sub{log}}^n$ (also denoted by $\nu_r^n(A)$ in the literature) is the subgroup of the Hodge--Witt group $W_r\Omega_A^n$ consisting of elements which can be written \'etale locally as sums of dlog forms, and the map $\dlog[\cdot]$ is given by $\{a_1,\dots,a_n\}\mapsto\dlog[a_1]\cdots\dlog[a_n]$ as usual. If $A$ is regular and local then both of these homomorphisms are known to be isomorphisms: this reduces, via Gersten sequences, to the case that $A$ is a field, in which case the leftwards isomorphism is due to Geisser and Levine \cite{GeisserLevine2000}, who also proved that $K_n(A)$ is $p$-torsion-free, and the rightwards isomorphism is the Bloch--Kato--Gabber theorem (see Theorem \ref{theorem_GL_BKG} for more details and references; also, to avoid issues caused by finite residue fields, we use Kerz--Gabber's improved Milnor $K$-theory throughout).

We extend these results to the pro abelian groups $\{K_n(A/I^s)\}_s$, where $I\subseteq A$ is an ideal. We must first describe two hypotheses: the first of these is that $A$ is $F$-finite (i.e., a finitely generated module over its subring of $p^\sub{th}$-powers); the second is that our closed subschemes $Y$ are often required to be {\em generalised normal crossing}, or {\em gnc}, meaning that $Y$ admits a closed cover by subschemes such that the reduced subscheme of any possible multiple intersection is regular (e.g., it suffices for $Y$ to be regular, or for it to be a normal crossing divisor on a regular scheme, which we believe cover all cases of interest for the applications; see Section \ref{subsection_gnc_schemes}); a ring is said to be gnc if and only if its spectrum is.

The following is our pro version of the isomorphisms recalled above:

\begin{theorem}[See Thm.~\ref{theorem_pro_GL} \& Corol.~\ref{corollary_thm_from_intro}]\label{theorem_main_1}
Let $A$ be a regular, F-finite $\bb F_p$-algebra, and $I\subseteq A$ an ideal such that $A/I$ is gnc and local; fix $n,r\ge 0$. Then the natural homomorphisms of pro abelian groups \[\{K_n(A/I^s)/p^r\}_s\longleftarrow \{K_n^M(A/I^s)/p^r\}_s\xTo{\dlog[\cdot]} \{W_r\Omega_{A/I^s,\sub{log}}^n\}_s\] are surjective and have the same kernel, thereby inducing an isomorphism \[\{K_n(A/I^s)/p^r\}_s\Isoto \{W_r\Omega_{A/I^s,\sub{log}}^n\}_s.\] Moreover, the pro abelian group $\{K_n(A/I^s)\}_s$ is $p$-torsion-free.
\end{theorem}

The kernel of the surjection $\{K_n^M(A/I^s)/p^r\}_s\to \{K_n(A/I^s)/p^r\}_s$ appearing in the statement of the theorem is reasonably well controlled; see Section \ref{subsection_milnor_vs_quillen} for some precise results, where we show in particular that it vanishes if $I$ is principal and $A/I$ is regular. This covers the traditional case of curves on $K$-theory, namely when $A=R[[t]]$ and $I=(t)$; a consequence of this is a curious, and
seemingly new, log/exp isomorphism between the relative part of $W_r\Omega_{R[t]/t^s,\sub{log}}^n$ and the big Hodge--Witt groups of $R$ itself:

\begin{corollary}[See Corol.~\ref{corollary_log_exp}]
Let $R$ be a regular, local, F-finite $\bb F_p$-algebra; fix $n\ge0$, $r\ge 1$. Then there exists a short exact sequence of pro abelian groups \[0\To\{\bb W_s\Omega_R^{n-1}/p^r\}_s\xTo{\dlog[\cdot]\circ\gamma_n}\{W_r\Omega_{R[t]/t^s,\sub{log}}^n\}_s\To W_r\Omega_{R,\sub{log}}^n\To 0,\] where $\gamma_n:\{\bb W_{s-1}\Omega^{n-1}_R\}_s\isoto \{K_n^\sub{sym}(R[t]/t^s,(t))\}_s$ is the original comparison map of Bloch--Deligne--Illusie between the de Rham--Witt complex and curves on $K$-theory (see Section \ref{subsection_Milnor} for more details).
\end{corollary}

Applying $\projlim_s$ to Theorem \ref{theorem_main_1}, together with a continuity result for logarithmic Hodge--Witt groups, we establish the following continuity result for $K$-theory; this is already known if $A/I$ is regular, thanks to Geisser and Hesselholt \cite{GeisserHesselholt2006b}:

\begin{theorem}[See Thm.~\ref{theorem_continuity}]
With notation as in Theorem \ref{theorem_main_1}, assume moreover that $A$ is $I$-adically complete. Then the canonical maps \[K_n(A;\bb Z/p^r)\To \pi_n\op{holim}_sK_n(A/I^s;\bb Z/p^r)\To \projlim_sK_n(A/I^s;\bb Z/p^r)\] are isomorphisms for all $n\ge0$, $r\ge1$. 
\end{theorem}

We present some similar continuity results for Milnor $K$-theory in Corollary \ref{corollary_continuity_Milnor}.

\subsection{Infinitesimal deformations of Chow groups}
We prove variations on Theorem \ref{theorem_main_1} for relative $K$-groups and in the context of sheaves in the Zariski, Nisnevich, and \'etale topologies; combining these with our development of the theory of logarithmic Hodge--Witt groups on formal $\bb F_p$-schemes, we prove a number of theorems concerning Chow groups and infinitesimal thickenings, including the following:

\begin{theorem}[See Thm.~\ref{theorem_class_groups}]\label{theorem_intro_class}
Let $X$ be a regular, F-finite $\bb F_p$-scheme, and $Y\into X$ a gnc closed subscheme. Then the canonical map of pro abelian groups \[\{H^i_\sub{Zar}(X,\cal K_{n,(X,Y_s)}/p^r)\}_s\To \{H^i_\sub{Nis}(X,\cal K_{n,(X,Y_s),\sub{Nis}}/p^r)\}_s\] is an isomorphism for all $n,i,r\ge0$, where $Y_s$ denotes the $s^\sub{th}$ infinitesimal thickening of $Y$ inside $X$.
\end{theorem}

In particular, if $X$ is a smooth variety over a perfect field of characteristic $p$ and $Y\into X$ is a normal crossing divisor, then Theorem \ref{theorem_intro_class} implies that \[\projlim_sH^i_\sub{Zar}(X,\cal K_{n,(X,Y_s)}/p^r)\Isoto \projlim_sH^i_\sub{Nis}(X,\cal K_{n,(X,Y_s),\sub{Nis}}/p^r).\] Replacing Quillen by Milnor $K$-theory and removing the mod $p^r$, this was conjectured to be true by Kato and Saito \cite[pg.~256]{KatoSaito1986} when $i=n=\dim X$, as part of their higher dimensional class field theory, in which the left and right sides play the role of certain Zariski/Nisnevich class groups in their theory.

To state our applications to the deformation of algebraic cycles, we consider for any $\bb F_p$-scheme $Y$ its ``cohomological Chow groups'' and ``\'etale-motivic cohomology groups'' \[\CH^n(Y):=H^n_\sub{Zar}(Y,\cal K_{n,Y}),\qquad H^*_\sub{\'et}(Y,\bb Z_p(n)):=H^{*-n}_\sub{\'et}(Y,\{W_r\Omega_{Y,\sub{log}}^n\}_r),\] where the left denotes cohomology of a Zariski sheafified $K$-group and the right denotes Jannsen's continuous \'etale cohomology of the pro \'etale sheaf $\{W_r\Omega_{Y,\sub{log}}^n\}_r$. If $Y$ is a smooth variety over a perfect field of characteristic $p$, then $\CH^n(Y)$ is the usual Chow group, by the Bloch--Quillen formula, and $H^*_\sub{\'et}(Y,\bb Z_p(n))$ is the \'etale-/Lichtenbaum- motivic cohomology with $\bb Z_p(n)$ coefficients, by Geisser--Levine \cite{GeisserLevine2000}; in this case of a smooth variety, or more generally a regular $\bb F_p$-scheme $Y$, the isomorphisms (\ref{eqn:1}) recalled at the start of the introduction induce the \'etale-motivic cycle class map $c_n:\CH_n(Y)\to H^{2n}_\sub{\'et}(Y,\bb Z_p(n))$.

Here in the introduction we state our general deformation result only in the case of schemes for simplicity, but it holds also for non-algebrisable formal schemes: 

\begin{theorem}[See Thm.~\ref{theorem_higher_codim}]\label{theorem_main_3}
Let $X$ be a regular, F-finite $\bb F_p$-scheme, and $Y\into X$ a regular closed subscheme. Let $z\in\CH^n(Y)$. Then:
\begin{enumerate}
\item Given $r\ge 1$ there exist $t\ge p^r$ (depending only on $X$ and $Y$, not $z$) such that, if the image of $c_n(z)$ in $H^{2n}_\sub{\'et}(Y,\bb Z/p^r\bb Z(n))$ lifts to $H^{2n}_\sub{\'et}(Y_t,\bb Z/p^r\bb Z(n))$ then $L$ lifts to $\CH^n(Y_{p^r})$.
\item $z$ lifts to $\projlim_s\CH^n(Y_s)$ if and only if $c_n(z)$ lifts to $\projlim_sH^{2n}_\sub{\'et}(Y_s,\bb Z_p(n))$.
\end{enumerate}
\end{theorem}

By proving an algebrisation lemma for \'etale-motivic cohomology, this has the following consequence for deformation in families:

\begin{corollary}[See Thm.~\ref{corollary_variational_in_families}]\label{corol:deforma_intro}
Let $A$ be a Noetherian, F-finite $\bb F_p$-algebra which is complete with respect to an ideal $I\subseteq X$, and let $X$ be a proper scheme over $A$; assume that $X$ and the special fibre $Y:=X\times_AA/I$ are regular\footnote{e.g., we could suppose $A$ is regular and $X$ is smooth over $A$; but the case in which $X$ is a desingularisation, with regular exceptional fibre, of $\Spec A$ is also interesting.}. For any $z\in\CH^n(Y)$, the following are equivalent:
\begin{enumerate}
\item $z$ lifts to $\projlim_s\CH^n(Y_s)$;
\item $c_n(z)$ lifts to $\projlim_sH^{2n}_\sub{\'et}(Y_s,\bb Z_p(n))$;
\item $c_n(z)$ lifts to $H^{2n}_\sub{\'et}(X,\bb Z_p(n))$.
\end{enumerate}
\end{corollary}

Similarly, taking advantage of the weak Lefschetz theorem in crystalline cohomology, the arguments used to prove Theorem \ref{theorem_main_3} establish the infinitesimal part of the weak Lefschetz conjecture for Chow groups in characteristic $p$; the analogous result over an algebraically closed field of characteristic $0$ is due to Patel and Ravindra \cite{PatelRavindra2014}:

\begin{theorem}[Infinitesimal weak Lefschetz for Chow groups; see Thm.~\ref{theorem_lefschetz}]
Let $X$ be a smooth, projective, $d$-dimensional variety over a perfect field $k$ of characteristic $p$, and $Y\into X$ a smooth ample divisor. Then the canonical map \[\projlim_sH^i_\sub{Zar}(Y_s,\cal K_{n,Y_s})\To H^i_\sub{Zar}(Y,\cal K_{n,Y})\] has kernel and cokernel killed by a power of $p$ if $i+n<d-1$. In particular, if $2n<d-1$ then \[(\projlim_s\CH^n(Y_s))\otimes_{\bb Z}\bb Z[\tfrac1p]\Isoto\CH^n(Y)\otimes_{\bb Z}\bb Z[\tfrac1p].\]
\end{theorem}

\subsection{Logarithmic Hodge--Witt sheaves on regular formal $\bb F_p$-schemes}
The technical heart of the article is the development of the theory of logarithmic Hodge--Witt sheaves on regular formal $\bb F_p$-schemes. To understand why, we very briefly sketch the proof of Theorem \ref{theorem_main_1}. The trace map from $K$-theory to topological cyclic homology, together with the pro Hochschild--Kostant--Rosenberg theorem for the latter \cite{Morrow_Dundas}, allows us to circumvent Milnor $K$-theory and directly construct a homomorphism of pro abelian groups \[\dlog^n_{r,A/I^\infty}:\{K_n(A/I^s)\}_s\To \{W_r\Omega_{A/I^s,\sub{log}}^n\}_s\] which is given by $\dlog[\cdot]$ on symbols. (This even exists without any conditions on $A/I$.) From McCarthy's theorem and the isomorphisms (\ref{eqn:1}) for $A/I$, it is then not difficult to obtain an isomorphism \[\{K_n(A/I^s)/p^s\}_s\Isoto \{W_s\Omega_{A/I^s,\sub{log}}^n\}_s,\] in which the diagonal indexing of the pro abelian groups should be noted. Modding out by multiples of $p^r$, the key to proving Theorem \ref{theorem_main_1} becomes the following:

\begin{theorem}[See \S\ref{subsection_dlog_n>1}]\label{theorem_main_2a}
Let $A$ be a regular, local, F-finite $\bb F_p$-algebra, and $I\subseteq A$ any ideal; fix $n,r\ge0$. Then the canonical reduction map \[\{W_{r'}\Omega_{A/I^s,\sub{log}}^n/p^r\}_s\To\{W_r\Omega_{A/I^s,\sub{log}}^n\}_s\] is an isomorphism of pro abelian groups for any $r'\ge r$.
\end{theorem}

If $I=0$ and $A$ is strictly Henselian then Theorem \ref{theorem_main_2a} reduces to the well-known result of Illusie that if $Y$ is a smooth variety over a perfect field of characteristic $p$, then the sequence of \'etale sheaves 
\[W_{r'}\Omega_{Y,\sub{log}}^n\xTo{p^r}W_{r'}\Omega_{Y,\sub{log}}^n\To W_r\Omega_{Y,\sub{log}}^n\To0\]
is exact \cite[\S I.5.7]{Illusie1979}. It is exactly this type of result which we are obliged to extend to regular formal $\bb F_p$-schemes, while at the same time analysing the logarithmic Hodge--Witt sheaves $W_r\Omega_{Y,\sub{log}}^n$ in the Zariski topology. Apart from Theorem~\ref{theorem_main_2a}, our most interesting results in this direction are perhaps the following:

\begin{theorem}[See \S\ref{subsection_dlog_n=1}]\label{theorem_main_5}
Let $X$ be a regular, F-finite $\bb F_p$-scheme, and $Y\into X$ a closed subscheme; fix $r\ge0$. Then:
\begin{enumerate}
\item The sequence of pro \'etale sheaves
\[0\To\{\bb G_{m,Y_s}\}_s\xTo{p^r}\{\bb G_{m,Y_s}\}_s\xTo{\dlog[\cdot]}\{W_r\Omega_{Y_s,\sub{log}}^1\}_s\To 0\] is exact.
\item The canonical map of pro abelian groups \[\{H^i_\sub{fppf}(Y_s,\bm\mu_{p^r,Y_s})\}_s\To \{H^{i-1}_\sub{\'et}(Y_s,W_r\Omega_{Y_s,\sub{log}}^1)\}_s\] is an isomorphism for all $i\ge0$.
\end{enumerate}
\end{theorem}

\begin{theorem}[See Corol.~\ref{corollary_log_vs_R-F} \& \ref{corollary_Zar_vs_etale}]\label{theorem_main_6}
Let $Y$ be any $\bb F_p$-scheme. Then:
\begin{enumerate}
\item The sequence of pro \'etale sheaves \[0\To\{W_r\Omega_{Y,\sub{log}}^n\}_r\To\{W_r\Omega_Y^n\}_r\xto{1-F} \{W_r\Omega_Y^n\}_r\To 0\] is exact.
\item $W_r\Omega_{Y,\sub{log}}^n$ is generated Zariski locally by dlog forms.
\end{enumerate}
\end{theorem}

Theorem \ref{theorem_main_5} gives analogues, for the formal completion of $X$ along $Y$, of the exact ``Cartier sequence'' \cite[Prop.~3.23.2]{Illusie1979} \[0\To\bb G_{m,X}\xTo{p^r}\bb G_{m,X}\xTo{\dlog[\cdot]}W_r\Omega_{X,\sub{log}}^1\To 0\]  and well-known isomorphism $H^i_\sub{fppf}(X,\bm\mu_{p^r,X})\isoto H^{i-1}_\sub{\'et}(X,W_r\Omega_{X,\sub{log}}^1)$. The key to proving the formal analogues is to mimic Illusie's proof in the smooth case by first establishing a Cartier isomorphism for regular formal $\bb F_p$-schemes (see \S\ref{subsection_cartier}) and then analysing the various filtrations on the de Rham--Witt complex (see \S\ref{subsection_filtrations}).

Theorem \ref{theorem_main_6}(i) is also due to Illusie in the smooth case, when  part (ii) is a consequence of the isomorphisms (\ref{eqn:1}). We eliminate the smoothness hypothesis by using the following type of argument (see \S\ref{subsection_F_fixed}): we may first assume that $Y=\Spec A$ is affine, and that we have a representation of $A$ as $B/I$, where $B$ is regular and $I$-adically complete; using infinitesimal deformations we lift the assertions to $W_r\Omega_{B,\sub{log}}^n$; by N\'eron--Popescu desingularision we then reduce the assertions the smooth case.

\subsection{$n=1$ and line bundles}
In Section \ref{section_line_bundles}, after developing the necessary foundations for the de Rham--Witt complex on a formal $\bb F_p$-scheme, but before turning to higher algebraic $K$-theory, we focus on line bundles and $W_r\Omega_\sub{log}^1$. There are two main reasons for this interlude. The first is to reprove and strengthen the deformation results for line bundles from \cite{Morrow_Variational_Tate} without using any $K$-theory or topological cyclic homology; indeed, we prove Theorem \ref{theorem_main_3} in the case $n=1$ using only arguments with logarithmic Hodge--Witt sheaves:

\begin{theorem}\label{theorem_intro_9}
Let $\cal Y$ be a regular F-finite, formal $\bb F_p$-scheme whose reduced subscheme of definition $Y=Y_1$ is regular. Let $L\in\Pic(Y)$. Then:
\begin{enumerate}
\item Given $r\ge 1$ there exists $t\ge p^r$ (depending only on $\cal Y$, not $L$) such that, if $c_1(L)\in H_\sub{\'et}^2(Y,\bb Z/p^r\bb Z(1))$ lifts to $H_\sub{\'et}^2(Y_t,\bb Z/p^r\bb Z(1))$ then $L$ lifts to $\Pic(Y_{p^r})$.
\item If $c_1(L) \in H_\sub{\'et}^2(Y,\bb Z_p(1))$ lifts to some $c\in\projlim_sH_\sub{\'et}^2(Y_s,\bb Z_p(1))$, then there exists $\tilde L\in\projlim_s\Pic(Y_s)$ which lifts $L$ and satisfies $c_1(\tilde L)=c$. In other words, the sequence \[\projlim_s\Pic(Y_s)\To \Pic(Y)\oplus \projlim_sH_\sub{\'et}^2(Y_s,\bb Z_p(1))\To H_\sub{\'et}^2(Y,\bb Z_p(1)) \] is exact.
\end{enumerate}
\end{theorem}

The second reason for the interlude is to relate this deformation result (in terms of $W_r\Omega^1_\sub{log}$) to earlier results of de Jong (in terms of the fppf cohomology of $\bm\mu_{p^r}$). This is achieved through Theorem \ref{theorem_main_5}, which allows us to rewrite Theorem \ref{theorem_intro_9} in terms of the fppf cohomolgies of $\bm\mu_{p^r,Y_s}$: see Remarks \ref{remark_flat_approach} and \ref{remark_relation_to_Tate_paper} for details.

\subsection{Guide}
Section \ref{section_prelim}, which presents various preliminary, known results, should be ignored by most readers initially and referred back to when necessary.

The applications to line bundles, namely the main results of Section \ref{section_line_bundles}, depend only on Sections \ref{subsection_cartier} -- \ref{subsection_filtrations} (minus Lemmas \ref{lemma_completion_of_dRW_complex} and Lemma \ref{lemma_Cartier_drw}).

The applications to $K$-theory, namely Section \ref{section_K_theory}, depend on all of Section \ref{section_drw} and \ref{section_hw_n>1}, but not Section\ref{section_line_bundles}. The five topics in Section \ref{section_applications} are largely independent of one another, but all require the material on $K$-theory in Section \ref{section_K_theory} and hence also the dependences of the previous sentence.

\subsection{Acknowledgements}
Part of this research was carried out during visits to the IHES during September--October 2014 and September 2015, and I thank the institute for their support. I am particularly grateful to A.~Abbes, whose suggestion during the first visit of establishing a suitable Cartier isomorphism had a profound effect on the paper. I also thank A.~J.~de Jong for discussions about deforming line bundles using fppf cohomology, which provided the major motivation for linking the fppf and log Hodge--Witt methods. Finally, I am grateful to T.~Geisser, K.~R\"ulling, and S.~Saito for helpful relevant conversations. I was funded by the Hausdorff Center for Mathematics during the project.

\section{Preliminary notation, hypotheses, and results}\label{section_prelim}
In this section we present various notations, conventions, and folklore results which will be used throughout; most readers should skip it and refer back when necessary. All rings in the paper are commutative.

\subsection{Regularity, F-finitentess, and N\'eron--Popescu desingularisation}\label{subsection_F-finite_etc}
It is too restrictive to work only with smooth algebras over perfect fields of characteristic $p$, mainly since this does not allow completions of such algebras to be uniformly treated. Therefore we work with the much wider class of regular $\bb F_p$-algebras. Since any regular $\bb F_p$-algebra is a filtered colimit of smooth (which includes the hypothesis of finite type) $\bb F_p$-algebras by N\'eron--Popescu desingularision \cite{Popescu1985, Popescu1986}, any result which commutes with filtered colimits automatically extends from the smooth setting to the regular case.

If $A$ is an algebra over a perfect field $k$ of characteristic $p$, then the $A$-modules \[\Omega_A^n:=\Omega_{A/\bb Z}^n,\quad \Omega_{A/\bb F_p}^n,\quad\Omega_{A/k}^n\] are identical, since $d(a)=pa^{(p-1)/p}d(a^{1/p})=0$ for all $a\in k$; we henceforth identify them without mention. If $A$ is Noetherian and {\em F-finite}, i.e., finitely generated over its subring of $p^\sub{th}$-powers, then it is easy to check that $\Omega_A^n$ is a finitely generated $A$-module. (More generally, a $\bb Z_{(p)}$-algebra $A$ is called $F$-finite if and only if $A/pA$ is F-finite in the previous sense.)

Due to this consequence, as well as many others, of F-finiteness (for example, although we will not use this fact, it is a remarkable theorem of Kunz \cite{Kunz1976} that Noetherian, F-finite, $\bb F_p$-algebras are always excellent and have finite Krull dimension), it will play an important role throughout and so we now explicitly mention its ``preservation properties''. If $A$ is a Noetherian, F-finite $\bb F_p$-algebra, then so are (see, e.g., \cite{Morrow_Dundas} for the proofs, which are not difficult): any finite type $A$-algebra; any localisation of $A$; the completion or Henselisation of $A$ along any ideal; and the strict Henselisation of $A$ along any prime ideal. Moreover, if $A$ is a Noetherian $\bb F_p$-algebra which is complete with respect to some ideal $I\subseteq A$, then Nakayama's lemma shows that $A$ is F-finite if and only if $A/I$ is F-finite.
Finally, we need the following:

\begin{lemma}\label{lemma_F-finite_formal_smooth}
Let $A$ be a regular, F-finite $\bb F_p$-algebra. Then $A$ is formally smooth over $\bb F_p$, i.e., it satisfies the usual infinitesimal lifting property.
\end{lemma}
\begin{proof}
Since $A$ is regular and $\bb F_p$ is a perfect field, the morphism $\bb F_p\to A$ is geometrically regular, and hence $A$ is a filtered colimit of smooth, finite type $\bb F_p$-algebras $A_i$ by N\'eron--Popescu desingularisation. Therefore $\Omega_A^1$ is the filtered colimit of the projective $A$-modules $\Omega_{A_i}^1\otimes_{A_i}A$, and hence it is flat. But, as remarked above, $\Omega_A^1$ is finitely generated; hence it is a projective $A$-module.

Since $A$ is a filtered colimit of smooth, finite type $k$-algebras, it is moreover true that the cotangent complex $\bb L_{A/\bb F_p}$ is supported in degree one. So we have proved that $\bb L_{A/\bb F_p}$ is quasi-isomorphic to the projective module $\Omega_A^1$. This is well-known to imply that $A$ is formally smooth over $\bb F_p$, e.g., \cite[Prop.~III.3.1.2]{Illusie1971}.
\end{proof}

\subsection{Logarithmic Hodge--Witt sheaves in various topologies}\label{subsection_HW_sheaves}
Let $X$ be an $\bb F_p$-scheme, and let $\tau$ denote the Zariski, Nisinevich, or \'etale topology. Viewing the the Hodge--Witt sheaf $W_r\Omega_X^n$ (which will be recalled in more detail in Section \ref{subsection_de_Rham_Witt}) as a sheaf in the $\tau$-topology, we denote by \[W_r\Omega^n_{X,\sub{log},\tau}\subseteq W_r\Omega_X^n\] the subsheaf which is generated $\tau$-locally by dlog forms, i.e., the $\tau$-sheafification of the image of the map of presheaves \[\dlog[\cdot]:\bb G_{m,X}^{\otimes n}\to W_r\Omega_X^n,\quad \al_1\otimes\cdots\otimes\al_n\mapsto\dlog[\al_1]\cdots\dlog[\al_n].\] When $X=\Spec A$ is affine, we write $W_r\Omega_{A,\sub{log},\tau}^n$ for the global sections of $W_r\Omega_{X,\sub{log},\tau}^n$. If $Y$ is a closed subscheme of $X$ then, as with many other sheaves in the paper, we write \[W_r\Omega_{(X,Y),\sub{log},\tau}^n:=\ker(W_r\Omega_{X,\sub{log},\tau}^n\to W_r\Omega_{Y,\sub{log},\tau}^n).\] When $\tau$ is the \'etale topology we will tend to omit it from the notation, in which case $W_r\Omega_{X,\sub{log}}^n$ is also denoted by $\nu^n_{r,X}$ or $\nu_r(n)_X$ in the literature. 

We let $\ep$ (resp.~$\ep_\sub{Nis}$) denote the projection from the Zariski (resp.~Nisnevich) topos to the \'etale topos. Then are then inclusions of Zariski sheaves \[W_r\Omega_{X,\sub{log},\sub{Zar}}^n\subseteq \ep_{\sub{Nis}*}W_r\Omega_{X,\sub{log},\sub{Nis}}^n\subseteq \ep_*W_r\Omega_{X,\sub{log}}^n,\] which are known to experts to be equalities is $X$ is regular, though the statement seems not to be in the literature:

\begin{theorem}\label{theorem_Zar=et}
If $X$ is a regular $\bb F_p$-scheme, then the inclusion of Zariski sheaves $W_r\Omega_{X,\sub{log},\sub{Zar}}^n\subseteq \ep_*W_r\Omega_{X,\sub{log}}^n$ is an equality.
\end{theorem}
\begin{proof}
In Theorem \ref{theorem_GL_BKG} we will recall the result that $\dlog[\cdot]:K_n^M(A)/p^r\to W_r\Omega_{A,\sub{log}}^n$ is surjective for any regular, local $\bb F_p$-algebra $A$, which is exactly the desired assertion.
\end{proof}

We will extend this result to arbitrary $\bb F_p$-schemes in Corollary \ref{corollary_Zar_vs_etale} using an infinitesimal deformation argument to reduce to the regular case.

\subsection{\'Etale-motivic cohomology}\label{subsection_etale_motivic}
For any $\bb F_p$-scheme $Y$, we adopt the notations \[H_\sub{\'et}^*(Y,\bb Z/p^r\bb Z(n)):=H^{*-n}_\sub{\'et}(Y,W_r\Omega_{Y,\sub{log}}^n),\qquad H_\sub{\'et}^*(Y,\bb Z_p(n)):=H^{*-n}_\sub{\'et}(Y,\{W_r\Omega_{Y,\sub{log}}^n\}_r),\] where the right-most group denotes Jannsen's continuous \'etale cohomology \cite{Jannsen1988} of the pro \'etale sheaf $\{W_r\Omega_{Y,\sub{log}}^n\}_r$ (the more common notation would be $H^{*-n}_\sub{cont}(Y_\sub{\'et},W_.\Omega_{Y,\sub{log}}^n)$).

Suppose now that $Y$ is smooth over a perfect field of characteristic $p$. Then it follows from Geisser--Levine \cite[Thm.~8.5]{GeisserLevine2000} that $W_r\Omega_{Y,\sub{log}}^n[n]\simeq z^n(-,\blob)_\sub{\'et}\dotimes_{\bb Z}\bb Z/p^r\bb Z$, where $z^n(-,\blob)_\sub{\'et}$ denotes Bloch's cycle complex of \'etale presheaves on $Y$, and so $H_\sub{\'et}^*(Y,\bb Z/p^r\bb Z(n))$ is the \'etale-/Lichenbaum-motivic cohomology of $Y$ with $\bb Z/p^r\bb Z$ coefficients (e.g., \cite[Def.~10.1]{MazzaWeibelVoevodsky2006}). Regardless of this identification, it had been understood much earlier that the cohomology groups $H_\sub{\'et}^*(Y,\bb Z/p^r\bb Z(n))$ and $H_\sub{\'et}^*(Y,\bb Z_p(n))$, as we have defined them, are the correct $p$-adic replacements for the $\ell$-adic \'etale cohomology of $Y$, particularly for the study of the Tate conjecture (e.g., \cite{Milne1988}). In the particular case $n=1$, there is a canonical identification $H_\sub{\'et}^*(Y,\bb Z/p^r\bb Z(1))=H^*_\sub{fppf}(Y,\bm\mu_{p^r,Y})$; this will essentially be recalled in the proof of Corollary \ref{corollary_fppf_mu_n}.

\subsection{Formal schemes}\label{subsection_formal_schemes}
If $\cal Y$ is a Noetherian formal $\bb F_p$-scheme, then we denote by $Y_1\into \cal Y$ a subscheme of definition and $Y_s\into\cal Y$ its $s^\sub{th}$-infinitesimal thickening; thus $\cal Y$ can be identified with the ind scheme $\{Y_s\}_s$, and the particular choice of $Y_1$ is irrelevant. If $Y\mapsto\cal F_Y\in Y_\sub{\'et}^\sim$ is a functorial collection of \'etale sheaves on $\bb F_p$-schemes (e.g., $\roi_Y$, $W_r\Omega_Y^n$, $\roi_Y^\times$, $W_r\Omega^n_{Y,\sub{log}}$, etc.), then by identifying the \'etale sites of $Y_1,Y_2,\dots$, we will view $\{\cal F_{Y_s}\}_s$ as a pro \'etale sheaf on $Y_1$. We will do the same in the Zariski and Nisnevich topologies.

We say that $\cal Y$ is {\em regular} (resp.~{\em F-finite}) if and only if it admits an open affine cover by the formal spectra $\Spf A$ of regular (resp.~F-finite) $\bb F_p$-algebras $A$.

Let $A$ be a Noetherian $\bb F_p$-algebra complete with respect to an ideal $I$. Then an $A$-algebra $A'$ is {\em $I$-formally \'etale} over $A$ if and only if it has the usual unique lifting property for diagrams of $A$-algebras
\[\xymatrix{
&C\ar[d]\\
A'\ar@{-->}[ur]^{\exists\,!}\ar[r] & C/J
}\]
in which $J$ is a nilpotent ideal of the $A$-algebra $C$ and the horizontal morphism is required to kill a power of $IA'$. If moreover $A'$ is $IA'$-adically complete and $A'/IA'$ is of finite type over $A/IA$, then we will say that $A'$ is a {\em topologically finite type (tft)}, $I$-formally \'etale $A$-algebra. Standard arguments show that the functor $A'\mapsto A'/IA'$ defines an equivalence of categories from tft $I$-formally \'etale $A$-algebras to \'etale $A/I$-algebras. Moreover, if $A$ is regular (resp.~F-finite), then so is any tft $I$-formally \'etale $A$-algebra.

\subsection{Generalised normal crossing schemes}\label{subsection_gnc_schemes}
We will use induction, via pro excision theorems, to reduce some assertions to the case of a regular scheme; this approach works for normal crossing divisors, or more generally for the class of schemes in the following definition, for which we know no standard terminology:

\begin{definition}
A Noetherian scheme $Y$ will be said to be {\em generalised normal crossing} (or simply {\em gnc}) if and only if it admits a cover by closed subschemes $Y^1,\dots,Y^c$ such that $(\bigcap_{i\in S}Y^i)_\sub{red}$ is regular for any subset $S\subseteq \{1,\dots,c\}$. The smallest such $c$ will be called the {\em complexity} of $Y$.

If $Y=\Spec A$ is affine, then we also say that $A$ is gnc.
\end{definition}

\begin{example}
The following examples are all obvious, but important enough to state:
\begin{enumerate}
\item A Noetherian scheme is gnc of complexity $\le 1$ if and only if its underlying reduced closed subscheme is regular.
\item If $X$ is a regular scheme and $Y\into X$ is a normal crossing divisor, then $Y$ is gnc. 
\item The union of the $x,y$-plane and the $z$-axis in $\bb A^3$ is a gnc scheme of complexity $2$ which does not fall under either of the previous examples.
\end{enumerate}
\end{example}

\begin{lemma}\label{lemma_complexity}
Let $Y$ be a Noetherian gnc scheme of complexity $\le c$. Then $Y$ admits a closed cover $Z,Z'$ such that $Z$ is regular and such that $Z'$ and $Z\cap Z'$ are gnc schemes of complexity $<c$.
\end{lemma}
\begin{proof}
Let $Y^1,\dots,Y^c$ be a closed cover of $Y$ with the property of the definition, and put $Z=(Y^c)_\sub{red}$ and $Z'=\bigcup_{i=i}^{c-1}Y^i$.
\end{proof}

It is plausible that, assuming embedded resolution of singularities in characteristic $p$, pro cdh descent for algebraic $K$-theory \cite{Morrow_pro_cdh_descent} would allow us to extend some of our results beyond the class of gnc schemes, but we have not seriously considered this.

\subsection{Artin--Rees properties in characteristic $p$}
The absolute Frobenius $a\mapsto a^p$ on an $\bb F_p$-algebra $A$ is denoted by $\phi$, or by $\phi^A$ when the ring must be made explicit. Given an $A$-module $M$, its restriction $\phi_*M$ along the Frobenius is the new $A$-module with underlying group equal to $M$ and action $a\cdot m:=a^pm$. Assuming $A$ is F-finite, then $M$ is finitely generated over $A$ if and only if $\phi_*M$ is finitely generated over $A$.

The following Artin--Rees properties will be used often in Section \ref{section_drw}; part (i) states that the functor $M\mapsto M/I^\infty M:=\{M/I^sM\}_s$ (we will occasionally use such $I^\infty$ notation when it is unlikely to cause confusion), from finitely generated $A$-modules to pro $A$-modules, is exact:

\begin{proposition}\label{proposition_AR_p}
Let $A$ be a ring, $I\subseteq A$ an ideal, and $M$ an $A$-module.
\begin{enumerate}
\item If $A$ is Noetherian and $M$ is finitely generated, then $\{\Tor^A_i(M,A/I^s)\}_s=0$ for all $i>0$.
\item If $A$ is an $\bb F_p$-algebra and $I$ is finitely generated then the canonical map \[\{\phi^A_*M\otimes_AA/I^s\}_s\To \{\phi^{A/I^s}_*(M\otimes_A A/I^s)\}_s\] is an isomorphism of pro $A$-modules.
\end{enumerate}
\end{proposition}
\begin{proof}
(i) is due to M.~Andr\'e \cite[Prop.~10 \& Lem.~11]{Andre1974} and D.~Quillen \cite[Lem.~9.9]{Quillen1968}. (ii) is simply the statement that the chains of ideals $I^s$ and $\phi(I^s)A$, for $s\ge1$, are intertwined, which is an easy consequence of $I$ being finitely generated.
\end{proof}

\section{The de Rham--Witt complex of a formal $\bb F_p$-scheme}\label{section_drw}
In this section we develop the theory of the de Rham--Witt complex on a regular, F-finite, formal $\bb F_p$-scheme. Since the calculations are usually of a local nature, this reduces to studying the pro $W_r(A)$-modules $\{W_r\Omega_{A/I^s}^n\}_s$ when $A$ is a regular, F-finite $\bb F_p$-algebra, and $I\subseteq A$ is an ideal, and we will state many of our results only in such an affine case in order to simplify notation.
 
\subsection{The Cartier isomorphism}\label{subsection_cartier}
For any $\bb F_p$-algebra $A$, the {\em inverse Cartier maps} $C^{-1}:\Omega_A^n\To H^n(\Omega_A^\blob)$, for $n\ge0$, are the additive maps characterised by the properties \[C^{-1}(a)=a^p,\quad C^{-1}(da)=a^{p-1}da,\quad C^{-1}(\omega\wedge\omega')=C^{-1}(\omega)\wedge C^{-1}(\omega').\] Replacing $\Omega_A^\blob$ by $\phi_*\Omega_A^\blob$ so that the de Rham differentials $d$ become $A$-linear, the inverse Cartier maps become morphisms of $A$-modules $C^{-1}:\Omega_A^n\to H^n(\phi_*\Omega_A^\blob)$.

The following celebrated theorem was proved by P.~Cartier in the case that $A$ is a smooth algebra over a perfect field of characteristic $p$; the more general case of a regular algebra follows immediately from N\'eron--Popescu desingularistion, as explained in Section \ref{subsection_F-finite_etc}:

\begin{theorem}[Cartier isomorphism]
Let $A$ is a regular $\bb F_p$-algebra and $n\ge0$. Then the inverse Cartier map $C^{-1}:\Omega_A^n\to H^n(\phi_*\Omega_A^\blob)$ is an isomorphism of $A$-modules.
\end{theorem}

Combining this with an Artin--Rees argument we obtain the following analogue of Cartier's theorem for formal schemes:

\begin{theorem}[Formal Cartier isomorphism]\label{theorem_pro_Cartier}
Let $A$ be a regular, F-finite $\bb F_p$-algebra, $I\subseteq A$ an ideal, and $n\ge0$. Then the inverse Cartier maps $C^{-1}:\Omega_{A/I^s}^n\To H^n(\phi_*\Omega_{A/I^s}^\blob)$ induce an isomorphism of pro $A$-modules \[C^{-1}:\{\Omega_{A/I^s}^n\}_s\Isoto\{H^n(\phi_*\Omega_{A/I^s}^\blob)\}_s.\]
\end{theorem}
\begin{proof}
For any fixed value of $s\ge1$, there is a natural commutative diagram of $A/I^s$-modules, in which the left vertical arrow is an isomorphism by the previous theorem
\[\xymatrix@C=0.5cm{
H^n(\phi_*\Omega_A^\blob)\otimes_AA/I^s\ar[r]^-{(1)} & 
H^n(\phi_*\Omega_A^\blob\otimes_AA/I^s)\ar[r]^-{(2)} & 
H^n(\phi_*(\Omega_A^\blob\otimes_AA/I^s))\ar[r]^-{(3)} & 
H^n(\phi_*(\Omega_{A/I^s}^\blob))\\
\Omega^n_A\otimes_AA/I^s\ar[u]_{C^{-1}\otimes_AA/I^s}\ar[rrr] &&&
\Omega_{A/I^s}^n\ar[u]^{C^{-1}}
}\]
By varying $s$ this becomes a diagram of pro $A$-modules. Then the long horizontal arrow, hence also arrow (3), become isomorphisms by the usual argument using the Leibnitz rule. Arrow (2) becomes an isomorphism by Proposition \ref{proposition_AR_p}(ii). Arrow (1) becomes an isomorphism by Proposition \ref{proposition_AR_p}(i), noting that the complex $\phi_*\Omega_A^\blob$ consists of finitely generated $A$-modules by Section \ref{subsection_F-finite_etc}.
\end{proof}

\begin{corollary}\label{corollary_pro_Cartier}
Let $A$ be a regular, F-finite $\bb F_p$-algebra, and $I\subseteq A$ an ideal. Then the sequences of pro abelian groups \[0\To\{A/I^s\}_s\stackrel\phi\To\{A/I^s\}_s\stackrel d\To\{\Omega_{A/I^s}^1\}_s,\quad
0\To\{A/I^s\,^\times\}_s\stackrel{p^r}\To\{A/I^s\,^\times\}_s\stackrel\dlog\To\{\Omega_{A/I^s}^1\}_s\]
are exact.
\end{corollary}
\begin{proof}
The first sequence is exactly Theorem \ref{theorem_pro_Cartier} for $n=0$; the second sequence follows by restricting to units.
\end{proof}

Now we consider the Cartier filtration. Let $A$ be an $\bb F_p$-algebra, and recall that the subgroups \[0=:B_0\Omega_A^n\subseteq B_1\Omega_A^n\subseteq\cdots \subseteq Z_1\Omega_A^n\subseteq Z_0\Omega_A^n:=\Omega_A^n\] are defined inductively as follows, for $i\ge1$:
\begin{itemize}
\item $Z_i\Omega_A^n$ is the $A$-submodule of $\phi^i_*\Omega_A^n$ satisfying $Z_i\Omega_A^n/d\Omega_A^{n-1}=C^{-1}(Z_{i-1}\Omega_A^n)$.
\item $B_i\Omega_A^n$ is the $A$-submodule of $\phi^i_*\Omega_A^n$ satisfying $B_i\Omega_A^n/d\Omega_A^{n-1}=C^{-1}(B_{i-1}\Omega_A^n)$.
\end{itemize}
The above equalities are merely ones of abelian groups as we have written $d\Omega_A^{n-1}$ rather than $d\phi_*^i\Omega_A^{n-1}$; also, since the notation is potentially misleading we mention that the inclusion $Z_1\Omega_A^n\subseteq\ker d$ need not be an equality.

Iterating the inverse Cartier map defines morphisms of $A$-modules \[C^{-i}:\phi^i_*\Omega_A^n\To Z_i\Omega_A^n/B_i\Omega_A^n.\] If $A$ is regular then it follows from the Cartier isomorphism, via a purely formal argument, that $C^{-i}$ is an isomorphism for all $i\ge0$. In the same way, it follows from Theorem \ref{theorem_pro_Cartier} (or from the classical case and Lemma \ref{lemma_Cartier_filtration} below) that:

\begin{theorem}
Let $A$ be a regular, F-finite $\bb F_p$-algebra, and $I\subseteq A$ an ideal. Then \[C^{-i}:\{\phi_*^i\Omega_{A/I^s}^n\}_s\To \{Z_i\Omega_{A/I^s}^n/B_i\Omega_{A/I^s}^n\}_s\] is an isomorphism of pro $A$-modules for each $i\ge1$ and $n\ge0$.
\end{theorem}

In Section \ref{subsection_filtrations} we will need the relationship between the Cartier filtrations on $\Omega_A^n$ and $\Omega_{A/I^s}^n$:

\begin{lemma}\label{lemma_Cartier_filtration}
Let $A$ be a Noetherian, F-finite $\bb F_p$-algebra, and $I\subseteq A$ an ideal. Then the canonical maps of pro $A$-modules 
\[\{Z_i\Omega_A^n\otimes_AA/I^s\}_s\To\{Z_i\Omega_{A/I^s}^n\}_s,\quad \{B_i\Omega_A^n\otimes_AA/I^s\}_s\To\{B_i\Omega_{A/I^s}^n\}_s\]
are isomorphisms for all $i\ge1$ and $n\ge 0$.
\end{lemma}
\begin{proof}
The claim is true for $Z_0$ since $\Omega_A^n\otimes_AA/I^\infty\isoto\Omega_{A/I^\infty}^n$, and we now proceed by induction on $i\ge1$. There is an obvious commutative diagram of pro $A$-modules
\[\xymatrix{
\phi_*^i\Omega_A^{n-1}\otimes_AA/I^\infty \ar[r]^d\ar[d] & Z_i\Omega_A^n\otimes_AA/I^\infty\ar[r]\ar[d]&(Z_i\Omega_A^n/d(\phi_*^i\Omega_A^{n-1}))\otimes_AA/I^\infty\ar[r]\ar[d] & 0\\
\phi_*^i\Omega_{A/I^\infty}^{n-1} \ar[r]^d & Z_i\Omega_{A/I^\infty}^n\ar[r]&\{Z_i\Omega_{A/I^s}^n/d(\phi_*^i\Omega_{A/I^s}^{n-1})\}_s\ar[r] & 0
}\]
The top row is exact since it results from applying $-\otimes_AA/I^\infty$ to an exact sequence of finitely generated $A$-modules (Proposition \ref{proposition_AR_p}(i)); the bottom row is exact by definition. The right vertical arrow is a quotient, via the inverse Cartier map, of $Z_{i-1}\Omega_A^n\otimes_AA/I^\infty\to Z_{i-1}\Omega_{A/I^\infty}^n$, which is an isomorphism by the inductive hypothesis; hence the right vertical arrow is surjective. The left vertical arrow is an isomorphism as usual. Hence the central vertical arrow is surjective; but it is also injective since it is a restriction of the isomorphism $\phi^i_*\Omega_A^n\otimes_AA/I^\infty\isoto\phi^i_*\Omega_{A/I^\infty}^n$. This completes the inductive step.

The proof for $B_i$ is entirely similar and hence omitted.
\end{proof}

\subsection{Preliminaries on Witt rings and Hodge--Witt groups}\label{subsection_de_Rham_Witt}
In this section we recall various basic results on de Rham--Witt complexes, especially regarding completions, from \cite{LangerZink2004, GeisserHesselholt2006b, Morrow_Dundas}; we work with more general rings than $\bb F_p$-algebras, since it causes no additional difficulty. We begin with a reminder on Witt rings of a ring $A$ and associated notation, restricting our attention to $p$-typical Witt rings $W_r(A)$ since they are sufficient for the main results. The Restriction, Frobenius and Verschiebung maps are denoted as usual by \[R,\; F:W_r(A)\To W_{r-1}(A),\quad V:W_{r-1}(A)\To W_r(A),\] and the Teichm\"uller map by $[\cdot]=[\cdot]_r:A\to W_r(A)$. The Restriction $R$ and Frobenius $F$ are ring homomorphisms, while $V$ is merely additive and $[\cdot]$ multiplicative.

Each element of $W_r(A)$ may be written uniquely as a Witt vector $(a_0,\dots,a_{r-1})=\sum_{i=0}^{r-1}V^i[a_i]_{r-i}$ for some $a_i\in A$; we will often use this to reduce questions to the study of terms of the form $V^i[a]_{r-i}$, which we will abbreviate by $V^i[a]$ when $r$ is clear from the context.

If $I\subseteq A$ is an ideal then $W_r(I)$ denotes the ideal of $W_r(A)$ defined as the kernel of the quotient map $W_r(A)\onto W_r(A/I)$. Alternatively, $W_r(I)$ is the Witt vectors of the non-unital ring $I$. An element $\al\in W_r(A)$ lies in $W_r(I)$ if and only if, in its expansion $\al=\sum_{i=0}^{r-1}V^i[a_i]$, the coefficients $a_i\in A$ all belong to $I$. 

Witt rings of $\bb Z_{(p)}$-algebras behave well in the presence of F-finiteness thanks to the following results of A.~Langer and T.~Zink:

\begin{theorem}[{Langer--Zink \cite[App.]{LangerZink2004}}]\label{theorem_Langer_Zink2}
Let $A$ be an F-finite $\bb Z_{(p)}$-algebra and $r\ge1$. Then:
\begin{enumerate}
\item The Frobenius $F:W_{r+1}(A)\to W_r(A)$ is a finite ring homomorphism.
\item If $B$ is a finitely generated $A$-algebra, then $W_r(B)$ is a finitely generated $W_r(A)$-algebra.
\item If $A$ is Noetherian then $W_r(A)$ is also Noetherian.
\end{enumerate}
\end{theorem}

We explicitly state the following standard lemma on chains of ideals in Witt rings, as it will be used repeatedly:

\begin{lemma}\label{lemma_witt_intertwined}
Let $A$ be a ring, $I\subseteq A$ an ideal, and $N,r\ge1$. Then $W_r(I)^N\subseteq W_r(I^N)$. Moreover, if $I$ is generated by finitely many elements $t_1,\dots,t_m\in I$, then there exists $M\ge N$ such that $W_r(I^M)\subseteq\pid{[t_1]^N,\dots,[t_m]^N}\subseteq W_r(I)^N$.
\end{lemma}
\begin{proof}
See, e.g., \cite[Lem.~2.1]{Morrow_Dundas} and its proof.
\end{proof}

We also explicitly state the following generalisations of Proposition \ref{proposition_AR_p} to the case of Witt vectors:

\begin{proposition}\label{proposition_AR_Witt}
Let $A$ be a Noetherian, F-finite $\bb Z_{(p)}$-algebra, $I\subseteq A$ an ideal, and $r\ge1$. Then:
\begin{enumerate}
\item If $M$ is a finitely generated $W_r(A)$-module then $\{\Tor^{W_r(A)}_i(M,W_r(A/I^s))\}_s=0$ for all $i>0$.
\item If $M$ is a $W_{r-1}(A)$-module, then the canonical map
\[\{F^A_*M\otimes_{W_r(A)}W_r(A/I^s)\}_s\To \{F^{A/I^s}_*(M\otimes_{W_{r-1}(A)}W_{r-1}(A/I^s))\}_s\] is an isomorphism of $W_r(A)$-modules.
\end{enumerate}
\end{proposition}
\begin{proof}
(i) is a special case of Proposition \ref{proposition_AR_p}(i), whose hypotheses are verified thanks to Theorem \ref{theorem_Langer_Zink2}(iii) and Lemma \ref{lemma_witt_intertwined}. (ii) is simply the statement that the chains of ideals $W_r(I^s)$ and $F(W_r(I^s))W_{r-1}(A)$, for $s\ge1$, are intertwined. By  Lemma \ref{lemma_witt_intertwined} it is sufficient to show instead that the chains $\pid{[t_1]^s,\cdots,[t_m]^s}$ and $\pid{F([t_1]^s),\cdots,F([t_m]^s)}$ are intertwined for some set of generators $t_1,\dots,t_m\in I$; this follows from the identity $F([t_i])=[t_i]^p$.
\end{proof}

We now review de Rham--Witt complexes (which, following common nomenclature, are composed of {\em Hodge--Witt} groups). Given an $\bb F_p$-algebra $A$, the existence and theory of the $p$-typical de Rham--Witt complex $W_r\Omega_A^\bullet$, which is a pro differential graded $W(A)$-algebra, is due to Bloch, Deligne, and Illusie; see especially \cite[Def.~I.1.4]{Illusie1979}. It was later extended by Hesselholt and Madsen to $\bb Z_{(p)}$-algebras with $p$ odd, and by Costeanu \cite{Costeanu2008} to $\bb Z_{(2)}$-algebras; see the introduction to \cite{Hesselholt2010} for further discussion. Again, we will focus on the $p$-typical case as it is sufficient for our main results. Recall that there are Restriction, Frobenius and Verschiebung maps \[R,\,F:W_r\Omega_A^n\To W_{r-1}\Omega_A^n,\quad V:W_{r-1}\Omega_A^n\To W_r\Omega_A^n\] which are compatible with those on the Witt ring of $A$.

We need conditions under which the Hodge--Witt groups are finitely generated:

\begin{lemma}\label{lemma_finite_generation_conditions}
Let $A$ be a Noetherian, F-finite $\bb Z_{(p)}$-algebra in which $p$ is nilpotent, and $r\ge1$. Then $W_r(A)$ is Noetherian and $W_r\Omega_A^n$ is a finitely generated module over it.
\end{lemma}
\begin{proof}

Langer--Zink's Theorem \ref{theorem_Langer_Zink2}(iii) states that $W_r(A)$ is Noetherian; it is also an F-finite ring in which $p$ is nilpotent (e.g., \cite[Lem.~2.9]{Morrow_Dundas}). Hence $\Omega_{W_r(A)}^n$ is a finitely-generated $W_r(A)$-module, as mentioned in Section \ref{subsection_F-finite_etc}. Since there is a natural surjection $\Omega_{W_r(A)}^n\to W_r\Omega_A^n$, the proof is now complete.
\end{proof}

Next we recall some basic properties of the Hodge--Witt groups in the presence of an ideal $I\subseteq A$; in this situation we write \[W_r\Omega_{(A,I)}^n:=\ker(W_r\Omega_A^n\to W_r\Omega_{A/I}^n).\] The following lemma recalls a standard result about this kernel:

\begin{lemma}\label{lemma_Witt_dg_ideal_gen_by_I}
Let $A$ be a ring, $I\subseteq A$ a finitely generated ideal, and $r\ge 1$. Then:
\begin{enumerate}
\item $W_r\Omega_{(A,I)}^\blob$ is the dg ideal of the dg algebra $W_r\Omega_A^\blob$ generated by $W_r(I)\subseteq W_r(A)$
\item For any $n\ge0$, the canonical maps $W_r(A/I^s)\otimes_{W_r(A)}W_r\Omega_A^n\to W_r\Omega_{A/I^s}$ induce an isomorphism of pro $W_r(A)$-modules \[\{W_r\Omega_A^n\otimes_{W_r(A)}W_r(A/I^s)\}_s\Isoto\{W_r\Omega_{A/I^s}^n\}_s.\]
\end{enumerate}
\end{lemma}
\begin{proof}
Claim (i), which does not require finite generation of $I$, is proved by directly checking that the Restriction, Frobenius, and Verschiebung maps on $W_r\Omega_A^\blob$ descend to the quotient by the dg ideal generated by $W_r(I)$; see \cite[Lem.~2.4]{GeisserHesselholt2006b}.

For (ii), note that the canonical maps are surjective, so one needs only to show that the pro abelian group arising from the kernels is zero. This is an easy consequence of Lemma \ref{lemma_witt_intertwined} and the Leibnitz rule; see, e.g.,~\cite[Prop.~2.5]{GeisserHesselholt2006b}, where the same result is proved.
\end{proof}

\begin{lemma}\label{lemma_completion_of_dRW_complex}
Let $A$ be a Noetherian, F-finite $\bb Z_{(p)}$-algebra in which $p$ is nilpotent, $I\subseteq A$ an ideal, $r\ge1$, and $n\ge0$. Then the canonical maps \[W_r\Omega_A^n\otimes_{W_r(A)}W_r(\hat A)\To W_r\Omega_{\hat A}^n\To \projlim_s  W_r\Omega_{A/I^s}^n\] are isomorphisms, where $\hat A$ denotes the $I$-adic completion of $A$.
\end{lemma}
\begin{proof}
Since $W_r\Omega_A^n$ is finitely generated over the Noetherian ring $W_r(A)$ by Lemma \ref{lemma_finite_generation_conditions}, and since $W_r(\hat A)$ coincides with the $W_r(I)$-adic completion of $W_r(A)$ by \cite[Lem.~2.3]{Morrow_Dundas}, the canonical map \[W_r\Omega_A^n\otimes_{W_r(A)}W_r(\hat A)\To\projlim_s\,W_r\Omega_A^n\otimes_{W_r(A)}W_r(A)/W_r(I)^s\] is an isomorphism by standard commutative algebra \cite[Thm.~8.7]{Matsumura1989}. By Lemma \ref{lemma_witt_intertwined}, the chains of ideals $W_r(I^s)$ and $W_r(I)^s$ are intertwined, so we may next replace $W_r(A)/W_r(I)^s$ on the right by $W_r(A)/W_r(I^s)=W_r(A/I^s)$. Then applying Lemma \ref{lemma_Witt_dg_ideal_gen_by_I}(ii) reveals that \[W_r\Omega_A^n\otimes_{W_r(A)}W_r(\hat A)\isoto \projlim_sW_r\Omega_{A/I^s}^n\]
But this isomorphism remains valid after replacing $A$ by $\hat A$ (which is still a Noetherian, F-finite -- by Section \ref{subsection_F-finite_etc} --, $\bb Z_{(p)}$-algebra in which $p$ is nilpotent) and $I$ by $I\hat A$; since $\hat A/I^s\hat A=A/I^s$ this means that \[W_r\Omega_{\hat A}^n\isoto \projlim_sW_r\Omega_{A/I^s}^n.\] Combining the two isomorphisms completes the proof.
\end{proof}

\subsection{Filtrations on the de Rham--Witt complex of $\bb F_p$-algebras}\label{subsection_filtrations}
In this section we study the usual filtrations on the de Rham--Witt complex.

\begin{definition}\label{definition_filtrations}
Let $A$ be an $\bb F_p$-algebra. The descending {\em canonical}, {\em $p$-}, and {\em $V$}-{\em filtrations} on $W_r\Omega_A^n$ are defined respectively by
\begin{align*}
\op{Fil}^iW_r\Omega_A^n&:=\ker(W_r\Omega_A^n\xto{R^{r-i}}W_i\Omega_A^n).\\
\op{Fil}_p^iW_r\Omega_A^n&:=\ker(W_r\Omega_A^n\xto{p^{r-i}}W_r\Omega_A^n)\\
\op{Fil}^i_VW_r\Omega_A^n&:=V^iW_{r-i}\Omega_A^n+dV^iW_{r-i}\Omega_A^{n-1}
\end{align*}
\end{definition}

It is an easy exercise, using standard de Rham--Witt identities which will occur for example in Remark \ref{remark_drw_etale_base_change}, to see that $\op{Fil}^i_VW_r\Omega_A^*$ is the dg ideal of the dg algebra $W_r\Omega_A^\blob$ generated by $V^iW_{r-i}(A)\subseteq W_r(A)$, and that there are inclusions of $W_r(A)$-submodules of $W_r\Omega_A^n$: \[\op{Fil}^iW_r\Omega_A^n\supseteq \op{Fil}^i_VW_r\Omega_A^n\subseteq \op{Fil}^i_pW_r\Omega_A^n.\] It was proved by Illusie \cite[Prop.~I.3.2 \& I.3.4]{Illusie1979} that these three filtrations coincide if $A$ is regular (to be precise, Illusie proved the equality whenever $A$ is smooth over a perfect field of characteristic $p$, which extends to the regular case N\'eron--Popescu desingularisation). Hesselholt then observed that the canonical and $V$-filtrations in fact coincide for any $\bb F_p$-algebra (the key observation is simply that  for any fixed $i\ge1$, the groups $W_{r+i}\Omega_A^n/\op{Fil}_V^iW_{r+i}\Omega_A^n$, for $n\ge0$ and $r\ge1$, have an induced structure of a Witt complex over $A$; a detailed proof in the generality of log structures may be found in \cite[Lem.~3.2.4]{Hesselholt2003}). In general the canonical and $p$-filtrations do not coincide (e.g., if $a^p=0$ then the vector $(a,0)\in W_2(A)$ is $p$-torsion but not killed by $R$), and the goal of this section is note that they do coincide for formal schemes under our usual hypotheses.

We first show that the filtrations behave well under base change along $A\to A/I^\infty$, similarly to Lemma \ref{lemma_Cartier_filtration}:

\begin{lemma}\label{lemma_pro_filtrations}
Let $A$ be a Noetherian, F-finite $\bb F_p$-algebra, and $I\subseteq A$ an ideal. Then the canonical maps of pro $W_r(A)$-modules \[\{(\op{Fil}^iW_r\Omega_A^n)\otimes_{W_r(A)}W_r(A/I^s)\}_s\To\{\op{Fil}^iW_r\Omega_{A/I^s}^n\}_s\] \[\{(\op{Fil}_p^iW_r\Omega_A^n)\otimes_{W_r(A)}W_r(A/I^s)\}_s\To\{\op{Fil}_p^iW_r\Omega_{A/I^s}^n\}_s\] \[\{(\op{Fil}^i_VW_r\Omega_A^n)\otimes_{W_r(A)}W_r(A/I^s)\}_s\To\{\op{Fil}^i_VW_r\Omega_{A/I^s}^n\}_s\] are isomorphisms for all $i,n\ge0$ and $r\ge 1$.
\end{lemma}
\begin{proof}
For each $s\ge1$ there is a natural commutative diagram of $W_r(A/I^s)$-modules
\[\xymatrix@C=5mm{
0\ar[r] & \op{Fil}^iW_r\Omega_A^n\otimes_{W_r(A)}W_r(A/I^s) \ar[d]\ar[r] & W_r\Omega_A^n\otimes_{W_r(A)}W_r(A/I^s) \ar[r]^{R^{r-i}}\ar[d] & W_r\Omega_A^n\otimes_{W_r(A)}W_r(A/I^s)\ar[d]\\
0\ar[r] & \op{Fil}^iW_r\Omega_{A/I^s}^n \ar[r] & W_r\Omega_{A/I^s}^n\ar[r] ^{R^{r-i}}&W_r\Omega_{A/I^s}^n
}\]
The bottom row is exact by definition. As pro abelian groups indexed over $s\ge 1$, the top row is exact by Proposition \ref{proposition_AR_Witt}(i) and Lemma \ref{lemma_finite_generation_conditions}, and the central and right vertical arrows are isomorphisms by Lemma \ref{lemma_Witt_dg_ideal_gen_by_I}(ii). Hence the left vertical arrow becomes an isomorphism of pro abelian groups, proving the desired result for the canonical filtration, and hence also the $V$-filtration with which it coincides. The same argument works for the $p$-filtration.
\end{proof}

We will now show that the canonical and $p$-filtrations on $W_r\Omega_{A/I^\infty}^n$ coincide. In particular, this implies that the $p$-torsion in $W_r\Omega_{A/I^\infty}^n$ vanishes in $W_{r-1}\Omega_{A/I^\infty}^n$, whence the pro abelian group $\{W_s\Omega_{A/I^s}^n\}_s$ has no $p$-torsion, which will be used at several key moments later in the paper:

\begin{proposition}\label{proposition_filtrations_are_equal}
Let $A$ be a regular, F-finite $\bb F_p$-algebra, and $n\ge0$, $r\ge 1$. Then the inclusions of pro $W_r(A)$-modules \[\{\op{Fil}^iW_r\Omega_{A/I^s}^n\}_s\supseteq \{\op{Fil}^i_VW_r\Omega_{A/I^s}^n\}_s\subseteq \{\op{Fil}_p^iW_r\Omega_{A/I^s}^n\}_s\] are equalities, for all $i\ge0$.
\end{proposition}
\begin{proof}
As recalled immediately after Definition \ref{definition_filtrations}, the left inclusion is even an equality for each fixed $s\ge 1$, while the right inclusion would be an equality with $A$ in place of $A/I^s$; hence the equality of the right inclusion follows from Lemma \ref{lemma_pro_filtrations}.
\end{proof}

\begin{remark}
The reader only interesting in Section \ref{section_line_bundles} can skip the remainder of Section~\ref{section_drw}.
\end{remark}

If $B$ is any $\bb F_p$-algebra, then multiplication $p^i:W_r\Omega_B^n\to W_r\Omega_B^n$ has image in $\op{Fil}^iW_r\Omega_A^n$ and sends $\op{Fil}^1W_r\Omega_B^n$ to $\op{Fil}^{i+1}W_r\Omega_B^n$ (use $\op{Fil}=\op{Fil}_V$ and the identity $p=FV$), thus inducing \[\ul{p}^i:\Omega_A^n=W_r\Omega_B^n/\op{Fil}^1W_r\Omega_B^n\to\op{Fil}^iW_r\Omega_B^n/\op{Fil}^{i+1}W_r\Omega_B^n.\] Varying $n$ yields a map of complexes $\ul{p}^i:\Omega_B^\blob\to \op{Fil}^iW_r\Omega_B^\blob/\op{Fil}^{i+1}W_r\Omega_B^\blob$, which Illusie proved was a quasi-isomorphism if $B$ is a smooth algebra over a perfect field \cite[Corol.~I.3.14]{Illusie1979}; we will need the following formal version of his result in the proof of Theorem \ref{theorem_lefschetz}:

\begin{corollary}\label{corollary_graded_pieces_are_de_rham}
Let $A$ be a regular, F-finite $\bb F_p$-algebra, $I\subseteq A$ an ideal, and $i\ge0$. Then \[\ul{p}^i:\{\Omega_{A/I^s}^\blob\}_s\to\{\op{Fil}^iW_r\Omega_{A/I^s}^\blob/\op{Fil}^{i+1}W_r\Omega_{A/I^s}^\blob\}_s\] is a quasi-isomorphism of pro complexes (i.e., induces an isomorphism on all the pro cohomology groups).
\end{corollary}
\begin{proof}
As usual Illusie's result remains valid for $A$, by N\'eron--Popescu desingularisation; the resulting quasi-isomorphism $\ul{p}^i:\Omega_A^\blob\quis \op{Fil}^iW_r\Omega_A^\blob/\op{Fil}^{i+1}W_r\Omega_A^\blob$ of abelian groups may be seen as one of $W_{2r}(A)$-modules $\ul{p}^i:F^r_*R^{r-1}_*\Omega_A^\blob\quis F^r_*\op{Fil}^iW_r\Omega_A^\blob/F^r_*\op{Fil}^{i+1}W_r\Omega_A^\blob$.

Base changing along $W_{2r}(A)\to W_{2r}(A/I^\infty)$, with Proposition \ref{proposition_AR_Witt} and Lemma \ref{lemma_pro_filtrations} in mind, yields the desired quasi-isomorphism of pro complexes.
\end{proof}

We finish this section by relating the Cartier filtration to the de Rham--Witt complex; the exact sequences in the following lemma will be used in the proof of Theorem~\ref{theorem_de_Rham--Witt_log}:

\begin{lemma}\label{lemma_Cartier_drw}
Let $A$ be a regular, F-finite $\bb F_p$-algebra, $I\subseteq A$ an ideal, and $n\ge0$, $i\ge1$. Then the following sequences become exact when assembled into pro abelian groups over $s\ge1$:
\begin{align*}
&0\To B_i\Omega_{A/I^s}^n\To \Omega_{A/I^s}^n\xto{V^r}W_{i+1}\Omega_{A/I^s}^n\tag{\ref{lemma_Cartier_drw}i}\\
&0\To B_{i+1}\Omega_{A/I^s}^n\To \Omega_{A/I^s}^n\xto{V^r} W_{i+1}\Omega_{A/I^s}^n/dV^r\Omega_{A/I^s}^{i-1}\tag{\ref{lemma_Cartier_drw}ii}\\
&0\To Z_{i+1}\Omega_{A/I^s}^{n-1}\To\Omega_{A/I^s}^{n-1}\xto{dV^r}W_{i+1}\Omega_{A/I^s}^n\tag{\ref{lemma_Cartier_drw}iii}\\
&W_{i+1}\Omega_{A/I^s}^{n-1}\xto{F^n}\Omega_{A/I^s}^{n-1}\xto{dV^r}W_{i+1}\Omega_{A/I^s}^n/V^r\Omega_{A/I^s}^n\tag{\ref{lemma_Cartier_drw}iv}
\end{align*}
\end{lemma}
\begin{proof}
By equations (3.8.1), (3.8.2), and (3.11.3) of \cite{Illusie1979}, there are exact sequences of abelian groups:
\begin{align*}
&0\To B_i\Omega_A^n\To \Omega_A^n\xto{V^r}W_{i+1}\Omega_A^n\\
&0\To B_{i+1}\Omega_A^n\To \Omega_A^n\xto{V^r} W_{i+1}\Omega_A^n/dV^r\Omega_A^{i-1}\\
&0\To Z_{i+1}\Omega_A^{n-1}\To\Omega_A^{n-1}\xto{dV^r}W_{i+1}\Omega_A^n\\
&0\To Z_i\Omega_A^{n-1}\To\Omega_A^{n-1}\xto{dV^r}W_{i+1}\Omega_A^n/V^r\Omega_A^n\\
&W_i\Omega_A^n\xto{V}W_{i+1}\Omega_A^n\xto{F^r}Z_i\Omega_A^n\To 0
\end{align*}
(Again, we are extending Illusie's result in the smooth case to $A$ using N\'eron--Popescu desingularisation.) The final two sequences may be assembled into an exact sequence $W_{i+1}\Omega_A^{n-1}\xto{F^r}\Omega_A^{n-1}\xto{dV^r}W_{i+1}\Omega_A^n/V^r\Omega_A^n$.

By pulling back along appropriate maps as indicated, these may be viewed as exact sequences of modules over the indicated rings:
\begin{align*}
&W_{i+1}(A)\qquad 0\To R^r_*B_i\Omega_A^n\To F_*^r\Omega_A^n\xto{V^r}W_{i+1}\Omega_A^n\\
&W_{i+1}(A)\qquad 0\To R_*^{r+1}B_{i+1}\Omega_A^n\To F_*^{r+1}\Omega_A^n\xto{V^r}F_*W_{i+1}\Omega_A^n/dV^r\Omega_A^{i-1}\\
&W_{i+2}(A)\qquad 0\To R_*^{r+1}Z_{i+1}\Omega_A^{n-1}\To F_*^{r+1}\Omega_A^{n-1}\xto{dV^r}F_*W_{i+1}\Omega_A^n\\
&W_{i+1}(A)\qquad W_{i+1}\Omega_A^{n-1}\xto{F^r}\Omega_A^{n-1}\xto{dV^r}W_{i+1}\Omega_A^n/V^r\Omega_A^n.
\end{align*}
By now base changing to $W_{i+1}(A/I^\infty)$ (resp.~$W_{i+2}(A/I^\infty)$ in the third case) using Proposition \ref{proposition_AR_Witt}(ii) and Lemma \ref{lemma_Cartier_filtration}, the proof is completed.
\end{proof}

\begin{remark}\label{remark_drw_etale_base_change}
A possible source of confusion, which we should already perhaps have mentioned in Section \ref{subsection_cartier}, is that the results in \cite{Illusie1979} are stated in terms of \'etale sheaves, whereas we prefer to take global sections in the affine case. For example, at the beginning of the previous proof when appealing to \cite{Illusie1979}, we implicitly used that the presheaves given by \[A\mapsto B_i\Omega_A^n,\;  Z_i\Omega_A^n,\; dV^r\Omega_A^{i-1}, V^r\Omega_A^n,\;\mbox{etc.} \] are in fact \'etale sheaves with vanishing higher cohomology on affines (otherwise, when taking global sections in Illusie's results, surjections might not be preserved and we might not get the claimed global sections). This of course is well-known but, since it will be implicitly used multiple times, we include a detailed explanation in one case, namely for $A\mapsto V^r\Omega_A$.

If $A\to A'$ is an \'etale morphism of $\bb F_p$-algebras then so is $W_r(A)\to W_r(A')$ \cite[Prop.~A.8]{LangerZink2004}, and the canonical base change map \begin{equation}W_r\Omega_A^n\otimes_{W_r(A)}W_r(A')\To W_r\Omega_{A'}^n\label{eqn_drw_etale}\end{equation} \cite[Prop.~1.7]{LangerZink2004} is an isomorphism. The sequence of $W_r(A)$-modules \[F_*^iW_{r-i}\Omega_A^n\xto{V^i} W_r\Omega_A^n\To W_r\Omega_A^n/V^iW_{r-i}\Omega_A^n\To 0\] is exact. After applying $-\otimes_{W_r(A)}W_r(A')$ this may be compared with the analogous sequence for $A'$, whence (\ref{eqn_drw_etale}) and the isomorphism $F_*^iW_{r-i}(A)\otimes_{W_r(A)}W_r(A')\isoto F^i_*W_{r-i}(A')$ of $W_r(A')$-modules \cite[Corol.~A.11]{LangerZink2004} show that $(W_r\Omega_A^n/V^iW_{r-i}\Omega_A^n)\otimes_{W_r(A)}W_r(A')\isoto W_r\Omega_{A'}^n/V^iW_{r-i}\Omega_{A'}^n$. It follows that \[V^iW_{r-i}\Omega_A^n\otimes_{W_r(A)}W_r(A')\isoto V^iW_{r-i}\Omega_{A'}^n,\] which suffices.
\end{remark}

\subsection{Further preliminaries on Hodge--Witt groups: Frobenius-fixed points}\label{subsection_F_fixed}
Now we study the kernel and cokernel of the operator $R-F$ on the Hodge--Witt groups, particular in the presence of nilpotent elements. As in Section \ref{subsection_de_Rham_Witt} we work with more general rings than $\bb F_p$-algebras when it causes no additional difficulty.

Given a ring $A$, we write \[W_r\Omega_A^{n,F=1}:=\ker (W_r\Omega_A^n\xto{R-F}W_{r-1}\Omega_A^n)\] and similarly, if $J\subseteq A$ is an ideal, \[W_r\Omega_{(A,J)}^{n,F=1}:=\ker (W_r\Omega_{(A,J)}^n\xto{R-F}W_{r-1}\Omega_{(A,J)}^n).\]

\begin{lemma}\label{lemma_surj_of_1-F_in_nilp_case}
Let $B$ be a ring, and $I\subseteq J\subseteq B$ ideals such that $I$ is nilpotent; fix $n\ge0$ and $r\ge 1$. Then:
\begin{enumerate}
\item If $x\in W_r\Omega_B^n$, then $dx=(R-F)dx'$, where $x':=-\sum_{j=1}^rV^iR^{i-1}x\in W_{r+1}\Omega_B^n$.
\item The map $R-F:W_{r+1}\Omega_{(B,I)}^n\to W_r\Omega_{(B,I)}^n$ is surjective.
\item The canonical maps $W_r\Omega_B^{n,F=1}\to W_r\Omega_{B/I}^{n,F=1}$ and $W_r\Omega_{(B,J)}^{n,F=1}\to W_r\Omega_{(B/I,J/I)}^{n,F=1}$ are surjective.
\end{enumerate}
\end{lemma}
\begin{proof}
(i): Applying the identity $d=FdV$ to $x,VRx,\dots, V^{r-1}R^{r-1}x\in W_r\Omega_B^n$, we see that
\begin{align*}
dx=&FdV(x)\\
&-dV(Rx)+FdV^2(Rx)\\
&-\cdots\\
&-dV^{r-1}(R^{r-1}x)+FdV^r(R^{r-1}x)\\
=&(F-R)d(Vx+V^2Rx+\cdots+V^rR^{r-1}x)+V^rR^rx\\
=&(F-R)d(Vx+V^2Rx+\cdots+V^rR^{r-1}x)
\end{align*}
since $R^rx=0$, as required.

(ii): Let $N\ge0$ be large enough so that $I^{p^N-1}=0$. We claim that any $\omega\in W_r\Omega_{(B,I)}^n$ may be lifted to an element $\tilde \omega\in W_{2r+N}\Omega_{(B,I)}^n$ satisfying $F^{r+N}(\tilde\omega)=0$. Once this claim has been proved we can set \[z:=\sum_{j=0}^{r+N-1}F^{r+N-1-j}R^j(\tilde\omega)\in W_{r+1}\Omega_{(B,I)}^n,\] which satisfies
\[(R-F)z=R^{r+N}(\tilde \omega)-F^{r+N}(\tilde\omega)=R^{r+N}(\tilde\omega)=\omega\] and so completes the proof.

We will prove the claim in two steps. The first step is to show that each $x\in W_r(I)$ may be lifted to some $\tilde x\in W_{2r+N}(I)$ such that both $F^{r+N}(\tilde x)$ and $F^{r+N}(d\tilde x)$ are zero. It is enough to suppose $x=V^i[b]_i$ for some $b\in I$ and $i\in\{0,\dots,r-1\}$, where we must clarify the Teichm\"uller lift $[\cdot]_i:A\to W_i(B)$ with the correct subscript. We will show that $\tilde x:=V_i[b]_{2r+N-i}\in W_{2r+N}(I)$, which is certainly a lift of $x$, has the desired property. Indeed, \[F^{N+r}(\tilde x)=F^{N+r}V^i[b]_{2r+N-i}=p^iF^{N+r-i}[b]_{2r+N-i}=p^i[b^{p^{N+r-i}}]_r=0\] by standard Witt vector identities and choice of $N$; similarly, \[F^{N+r}(d\tilde x)=F^{N+r}dV^i[b]_{2r+N-i}=F^{n+r-i}d[b]_{2r+N-i}=[a^{p^N-1}]d[b]_r=0,\] completing the first step of the claim.

Second step of the claim: By Lemma \ref{lemma_Witt_dg_ideal_gen_by_I}(i), $W_r\Omega_{(B,I)}^*$ is the differential graded ideal of $W_r\Omega_B^*$ generated by $W_r(I)\subseteq W_r\Omega_B^0$. In other words, each element of $W_r\Omega_{(B,I)}^n$ is a finite sum of terms of the form $\omega=x\,dx_1\cdots dx_n$ where $x,x_1,\dots,x_n$ are elements of $W_r(B)$, at least one of which belongs to $W_r(I)$. Let $\tilde x,\tilde x_1,\dots\tilde x_n\in W_{2r+N}(B)$ be lifts of these elements, chosen with the convention that if the element belongs to $W_r(I)$ then we choose a lift in $W_{2r+N}(I)$ with the property of the first step. Put $\tilde \omega:=\tilde x\,d\tilde x_1\cdots d\tilde x_n\in W_{2r+N}\Omega_{(B,I)}^n$, which is a lift of $\omega$. Then $F^{N+r}(\tilde\omega)=F^{N+r}(\tilde x)F^{N+r}(d\tilde x_1)\cdots F^{N+r}(d\tilde x_n)=0$ by choice of our lifts.

(iii): There is a commutative diagram with surjective vertical arrows:
\[\xymatrix{
W_r\Omega_B^n\ar[d]\ar[r]^{R-F}&W_{r-1}\Omega_B^n\ar[d]\\
W_r\Omega_{B/I}^n\ar[r]_{R-F}&W_{r-1}\Omega_{B/I}^n
}\]
Since the induced map on the kernels of the vertical arrows is surjective by (ii), it follows that the induced map on the kernels of the horizontal arrows is also surjective, as required. The more general case for an ideal $J$ is the same argument, using the obvious identity $\ker(W_r\Omega_{(B,J)}^n\to W_r\Omega_{(A/I,J/I)}^n)= W_r\Omega_{(B,I)}^n$.
\end{proof}

Under mild finiteness hypotheses, the results of the previous lemma may be extended to $I$-adically complete rings; this will provide a useful technique to lift problems to regular rings:

\begin{proposition}\label{proposition_limit_of_F-fixed_points}
Let $A$ be a Noetherian, F-finite $\bb Z_{(p)}$-algebra in which $p$ is nilpotent, and $I\subseteq J\subseteq A$ ideals such that $A$ is $I$-adically complete; fix $n\ge0$ and $r\ge1$. Then:
\begin{enumerate}
\item The canonical maps $W_r\Omega_A^{n,F=1}\to\projlim_sW_r\Omega_{A/I^s}^{n,F=1}$ and $W_r\Omega_{(A,J)}^{n,F=1}\to\projlim_sW_r\Omega_{(A/I^s,J/I^s)}^{n,F=1}$ are isomorphisms.
\item The canonical maps $W_r\Omega_{A}^{n,F=1}\to W_r\Omega_{A/I}^{n,F=1}$ and $W_r\Omega_{(A,J)}^{n,F=1}\to W_r\Omega_{(A/I,J/I)}^{n,F=1}$ are surjective.
\item The map $R-F:W_r\Omega_{(A,I)}^n\to W_{r-1}\Omega_{(A,I)}^n$ is surjective.
\end{enumerate}
\end{proposition}
\begin{proof}
(i): Taking the inverse limit over $s\ge1$ of the short exact sequences \[0\To W_r\Omega_{A/I^s}^{n,F=1}\To W_r\Omega_{A/I^s}^n\xto{R-F}W_{r-1}\Omega_{A/I^s}^n\] and using Lemma \ref{lemma_completion_of_dRW_complex} yields a short exact sequence \[0\To \projlim_sW_r\Omega_{A/I^s}^{n,F=1}\To W_r\Omega_A^n\xto{R-F}W_{r-1}\Omega_A^n, \] whence $W_r\Omega_A^{n,F=1}\isoto\projlim_sW_r\Omega_{A/I^s}^{n,F=1}$. Taking the kernel of the map to $W_r\Omega_{A/J}^{n,F=1}$ proves $W_r\Omega_{(A,J)}^{n,F=1}\isoto\projlim_sW_r\Omega_{(A/I^s,J/I^s)}^{n,F=1}$.

(ii): Since the transition maps in the systems $\{W_r\Omega_{A/I^s}^{n,F=1}\}_s$ and $\{W_r\Omega_{(A/I^s,J/I^s)}^{n,F=1}\}_s$ are surjective, by Lemma \ref{lemma_surj_of_1-F_in_nilp_case}(iii), the assertion follows from (i).

(iii): Lemmas \ref{lemma_surj_of_1-F_in_nilp_case} yields short exact sequences \[0\To W_r\Omega_{(A/I^s,I/I^s)}^{n,F=1}\To W_r\Omega_{(A/I^s,I/I^s)}^n\xto{R-F}W_{r-1}\Omega_{(A/I^s,I/I^s)}^n\To0,\] where the transition maps over $s\ge1$ are surjective on the left group. Taking the limit and using Lemma \ref{lemma_completion_of_dRW_complex} yields \[0\To \projlim_sW_r\Omega_{(A/I^s,I/I^s)}^{n,F=1}\To W_r\Omega_{A,I}^n\xto{R-F}W_{r-1}\Omega_{A,I}^n\To 0,\] as desired.
\end{proof}

\begin{remark}[The case of big de Rham--Witt complexes]
Some of our results remain true for big Hodge--Witt groups $\bb W_S\Omega_A^n$ associated to finite truncation sets $S$, although we do not need them. For general terminology surrounding big Witt vectors and truncation sets, see \cite{Hesselholt2010}. Given an inclusion of truncation sets  $S\supseteq T$, there are associated Restriction, Frobenius and Verschiebung maps \[R_T,\;F_T:\bb W_S\Omega_A^n\to\bb W_T\Omega_A^n,\quad V_T:\bb W_T\Omega_A^n\to\bb W_S\Omega_A^n.\] If $m\ge1$ is an integer then one defines a new truncation set by $S/m:=\{s\in S:sm\in S\}$ and writes $R_m$, $F_m$, and $V_m$ instead of $R_{S/m}$, $F_{S/m}$, and $V_{S/m}$ respectively. The $p$-typical case is recovered with the truncation set $S=\{1,p,\dots,p^{r-1}\}$.

If $S$ is any finite truncation set in place of $\{1,p,\dots,p^{r-1}\}$ then Lemmas \ref{lemma_witt_intertwined}, \ref{lemma_Witt_dg_ideal_gen_by_I}, and \ref{lemma_completion_of_dRW_complex} remain true. If we also fix $m\ge 1$, then the obvious analogues of Lemma \ref{lemma_surj_of_1-F_in_nilp_case} and Proposition \ref{proposition_limit_of_F-fixed_points} remains true for the morphism $R_m-F_m:\bb W_S\Omega_A^n\to\bb W_{S/m}\Omega_A^n$
\end{remark}

\section{Interlude: logarithmic Hodge--Witt sheaves ($n=1$) and line bundles}\label{section_line_bundles}
In this section we study $W_r\Omega_\sub{log}^1$, its relation to the fppf sheaf $\bm\mu_{p^r}$, and and its role in deforming line bundles on formal $\bb F_p$-schemes. Section \ref{section_hw_n>1} and onwards do not depend on this material, and so readers interested mainly in higher algebraic $K$-theory may skip this section.

\subsection{A formal Cartier sequence and relation to $\pmb\mu_{p^r}$}\label{subsection_dlog_n=1}
The following Cartier sequence extends to regular formal $\bb F_p$-schemes a well-known result for smooth varieties, namely \cite[Prop.~3.23.2]{Illusie1979}, and it will underly the deformation results for line bundles in Section \ref{subsection_deform_line}:

\begin{theorem}\label{theorem_pro_higher_Cartier}
Let $\cal Y$ be a regular, F-finite formal $\bb F_p$-scheme, $Y_1\into \cal Y$ a subscheme of definition, and $r\ge1$. Then the following sequence of pro \'etale sheaves on $Y_1$ is exact:
\[0\To\{\bb G_{m,Y_s}\}_s\xTo{p^r}\{\bb G_{m,Y_s}\}_s\xTo{\dlog[\cdot]}\{W_r\Omega_{Y_s,\sub{log}}^1\}_s\To 0.\]
\end{theorem}
\begin{proof}
The necessary right exactness is tautological from the definition of $W_r\Omega_{Y_s,\sub{log}}^1$, and the left exactness follows from Corollary \ref{corollary_pro_Cartier} (which was stated in the affine case but remains true for sheaves). Now we suppose that $\cal Y=\Spf A$ is the formal spectrum of a regular, F-finite $\bb F_p$-algebra $A$ which is complete with respect to an ideal $I\subseteq A$.

For any $\bb F_p$-algebra $B$, there is a natural commutative diagram of abelian groups:
\begin{equation}\xymatrix{
& B^\times/B^{\times p^{r-1}}\ar[r]^p \ar[d]^{\dlog[\cdot]}& B^\times/B^{\times p^r}\ar[r] \ar[d]^{\dlog[\cdot]}& B^\times/B^{\times p}\ar[r] \ar[d]^{\dlog}& 0\\
0\ar[r] & \op{Fil}_p^1W_r\Omega_B^1\ar[r] & W_r\Omega_B^1\ar[r]^{R^{r-1}}&\Omega_B^1\ar[r]&0
}\label{eqn_dlog}\end{equation}
The top row is exact, and the bottom row is exact except possibly at the middle.

Now let $B=A/I^s$ and assemble diagram (\ref{eqn_dlog}) into one of pro abelian groups indexed over $s\ge1$. Then the bottom row becomes exact thanks to the equality $\{\op{Fil}^1W_r\Omega_{A/I^s}^n\}_s=\{\op{Fil}_p^1W_r\Omega_{A/I^s}^n\}_s$ of Proposition \ref{proposition_filtrations_are_equal}; the right vertical arrow becomes injective by Corollary \ref{corollary_pro_Cartier}; and the left vertical arrow becomes injective because the composition
\[\xymatrix{
\{(A/I^s\,^{\times})/(A/I^s\,^{\times p^{r-1}})\}_s\ar@/_1cm/[rr]_{\dlog[\cdot]_{r-1}}
\ar[r]^-{\dlog[\cdot]_r}&\{\op{Fil}_p^1W_r\Omega_{A/I^s}\}_s\ar[r]^R&\{W_{r-1}\Omega_{A/I^s}\}_s
}\]
may be assumed to be injective by induction on $r\ge1$. Hence the central vertical arrow of diagram (\ref{eqn_dlog}) becomes an injection of pro abelian groups, proving the necessary central exactness of the theorem.
\end{proof}

If $Y$ is any $\bb F_p$-scheme then $\dlog[\cdot]$ induces a natural homomorphism $H^*_\sub{fppf}(Y,\bm\mu_{p^r,Y})\to H^{*-1}_\sub{\'et}(Y,W_r\Omega_{Y,\sub{log}}^1)$ which is known to be an isomorphism is $Y$ is smooth over a perfect field (further details will be recalled in the following proof); the previous theorem allows us to prove an analogous result for formal schemes, which we include especially to unify the two existing approaches to deforming line bundles in characteristic $p$ (see Remark \ref{remark_relation_to_Tate_paper}); we denote by $\rho$ the projection map from the fttp topos to the \'etale topos:

\begin{corollary}\label{corollary_fppf_mu_n}
Let $\cal Y$ be a regular, F-finite formal $\bb F_p$-scheme, $Y_1\into \cal Y$ a subscheme of definition, and $r\ge1$. Then the canonical map $\{R\rho_*\bm\mu_{p^r,Y_s}[1]\}_s\to\{W_r\Omega^1_{Y_s,\sub{log}}\}_s$ of pro complexes of Zariski sheaves is a quasi-isomorphism; i.e., \[\{H^i_\sub{fppf}(Y_s,\bm\mu_{p^r,Y_s})\}_s\Isoto \{H^{i-1}_\sub{\'et}(Y_s,W_r\Omega_{Y_s,\sub{log}}^1)\}_s\] for all $i\ge0$.
\end{corollary}
\begin{proof}
On any $\bb F_p$-scheme $Y$, the sequence of fppf sheaves \[0\To\bm \mu_{p^r,Y}\To\bb G_{m,Y,\sub{fppf}}\xTo{p^r}\bb G_{m,Y,\sub{fppf}}\To 0\] is exact; applying the projection map $\rho$ obtains vanishings $R^i\rho_*\bm\mu_{p^r,Y}=0$ for $i>1$ and an exact sequence of \'etale sheaves \[0\to\rho_*\bm \mu_{p^r,Y}\to \bb G_{m,Y,\sub{\'et}}\xto{p^r}\bb G_{m,Y,\sub{\'et}}\to R^1\rho_*\bm\mu_{p^r,Y}\to 0,\] were we use the fact that $R\rho_*\bb G_{m,Y,\sub{fppf}}=\bb G_{m,Y,\sub{\'et}}$ \cite[Thm.~III.11.7]{Grothendieck1968}. Since the map $\dlog[\cdot]:\bb G_{m,Y,\sub{\'et}}\to W_r\Omega_{Y,\sub{log}}^1$ kills $p^{r\sub{-th}}$-powers, it therefore induces $\dlog[\cdot]:R^1\rho_*\bm\mu_{p^r,Y}\to W_r\Omega_{Y,\sub{log}}^1$ and we arrive at a diagram
\[\xymatrix{
0\ar[r]& H^n_\sub{\'et}(Y,\rho_*\bm\mu_{p^r,Y})\ar[r]& H^n_\sub{fppf}(Y,\bm\mu_{p^r,Y})\ar[r]^-\delta& H^{n-1}_\sub{\'et}(Y,R^1\rho_*\bm\mu_{p^r,Y})\ar[d]^{\dlog[\cdot]}\ar[r]& 0\\
&&&H^{n-1}_\sub{\'et}(Y,W_r\Omega_{Y,\sub{log}}^1)&
}\]
where the sequence associated to the change of topology is short exact.

Now adopt the hypotheses of the statement of the corollary and apply the previous paragraph to $Y=Y_s$ for all $s\ge 1$; in particular, there are exact sequences of pro \'etale sheaves on $Y_1$ \[0\to\{\rho_*\bm \mu_{p^r,Y_s}\}_s\to \{\bb G_{m,Y_s,\sub{\'et}}\}\xto{p^r}\{\bb G_{m,Y_s,\sub{\'et}}\}\to \{R^1\rho_*\bm\mu_{p^r,Y_s}\}\To 0.\] Theorem \ref{theorem_pro_higher_Cartier} now implies that $\{\rho_*\bm \mu_{p^r,Y_s}\}_s=0$ (which in any case is easy to see by a direct argument) and $\dlog[\cdot]:\{R^1\rho_*\bm\mu_{p^r,Y_s}\}\isoto \{W_r\Omega_{Y_s,\sub{log}}^1\}_s$ (on the other hand, this is essentially a reformation of the main content of Theorem \ref{theorem_pro_higher_Cartier}). Hence \[\dlog[\cdot]\circ\delta:\{H^i_\sub{fppf}(Y_s,\bm\mu_{p^r,Y_s})\}_s\To \{H^{i-1}_\sub{\'et}(Y_s,W_r\Omega_{Y_s,\sub{log}}^1)\}_s\] is an isomorphism for all $i\ge0$, as required.
\end{proof}

Although it is not required in the remainder of this section we now explicitly mention a consequence for the algebrisation of the fppf cohomology of $\bm\mu_{p^r}$. On an $\bb F_p$-scheme $Y$ we write \[H^*_\sub{fppf}(Y,\bm\mu_{p^\infty,Y}):=H^*_\sub{fppf}(Y,\{\bm\mu_{p^r,Y}\}_r)\]for the continuous fppf cohomology of the inverse system of sheaves $\cdots\xto{p}\bm\mu_{p^2,Y}\xto{p}\bm\mu_{p,Y}$:

\begin{corollary}
Let $A$ be a Noetherian, F-finite $\bb F_p$-algebra which is complete with respect to an ideal $I\subseteq A$, let $X$ be a proper scheme over $A$, and write $Y_s:=X\times_AA/I^s$ for $s\ge1$. Then the canonical map
\[H^i_\sub{fppf}(X,\bm\mu_{p^\infty,X})\To\projlim_sH^i_\sub{fppf}(Y_s,\bm\mu_{p^\infty,Y_s})\]
is surjective for all $i\ge0$.
\end{corollary}
\begin{proof}
In light of the previous corollary, this is exactly Corollary \ref{lemma_algebrisation_of_WOmega_log} (which depends only on Sections \ref{section_drw} and Corollary \ref{corollary_log_vs_R-F}) in the case $n=1$.
\end{proof}

\subsection{The deformation of line bundles on formal schemes}\label{subsection_deform_line}
In this section we apply Theorem \ref{theorem_pro_higher_Cartier} to the deformation and variational theory of line bundles in characteristic $p$. The goal is to characterise whether a line bundle on an $\bb F_p$-scheme $Y$ can be deformed in terms of its Chern class inside the cohomology groups \[H_\sub{\'et}^2(Y,\bb Z/p^r\bb Z(1)):=H^1_\sub{\'et}(Y,W_r\Omega_{Y,\sub{log}}^1),\qquad H_\sub{\'et}^2(Y,\bb Z_p(1)):=H^1_\sub{\'et}(Y,\{W_r\Omega_{Y,\sub{log}}^1\}_r),\] where we use the notation of Section \ref{subsection_etale_motivic}. The maps $\dlog[\cdot]:\bb G_{m,Y}\to W_r\Omega^1_{Y,\sub{log}}$ induce, on \'etale cohomology, \'etale-motivic Chern classes \[c_1:\Pic(Y)\To H^2(Y,\bb Z/p^r\bb Z(1))\mbox{ or }H^2(Y,\bb Z_p(1))\] (the intended codomain will be clear from the context).

Our general result on deforming line bundles on formal schemes is as follows:

\begin{theorem}\label{theorem_deforming_line_bundles}
Let $\cal Y$ be a regular F-finite, formal $\bb F_p$-scheme whose reduced subscheme of definition $Y=Y_1$ is regular. Let $L\in\Pic(Y)$. Then:
\begin{enumerate}
\item Given $r\ge 1$ there exists $t\ge p^r$ (depending only on $\cal Y$, not $L$) such that, if $c_1(L)\in H_\sub{\'et}^2(Y,\bb Z/p^r\bb Z(1))$ lifts to $H_\sub{\'et}^2(Y_t,\bb Z/p^r\bb Z(1))$ then $L$ lifts to $\Pic(Y_{p^r})$.
\item If $c_1(L) \in H_\sub{\'et}^2(Y,\bb Z_p(1))$ lifts to some $c\in\projlim_sH_\sub{\'et}^2(Y_s,\bb Z_p(1))$, then there exists $\tilde L\in\projlim_s\Pic(Y_s)$ which lifts $L$ and satisfies $c_1(\tilde L)=c$. In other words, the sequence \[\projlim_s\Pic(Y_s)\To \Pic(Y)\oplus \projlim_sH_\sub{\'et}^2(Y_s,\bb Z_p(1))\To H_\sub{\'et}^2(Y,\bb Z_p(1)) \] is exact.
\end{enumerate}
\end{theorem}
\begin{proof}
We compare the complex of \'etale sheaves $0\to \bb G_m\xto{p^r} \bb G_m\xto{\dlog[\cdot]} W_r\Omega^1_\sub{log}\to 0$ for $Y_s$ and $Y$, and take the kernel:
\[\xymatrix@R=5mm{
 & 0\ar[d] & 0\ar[d] & 0\ar[d] &\\
0\ar[r] & \bb G_{m,(Y_s,Y)} \ar[d]\ar[r]^{p^r} & \bb G_{m,(Y_s,Y)}\ar[r]^{\dlog[\cdot]}\ar[d]& W_r\Omega_{(Y_s,Y),\sub{log}}^1\ar[r]\ar[d] & 0\\
0\ar[r] & \bb G_{m,Y_s} \ar[d]\ar[r]^{p^r} & \bb G_{m,Y_s}\ar[r]^{\dlog[\cdot]}\ar[d]& W_r\Omega_{Y_s,\sub{log}}^1\ar[r]\ar[d] & 0\\
0\ar[r] & \bb G_{m,Y}\ar[d] \ar[r]^{p^r} & \bb G_{m,Y}\ar[d]\ar[r]^{\dlog[\cdot]}& W_r\Omega_{Y,\sub{log}}^1\ar[d]\ar[r] & 0\\
&0&0&0\\
}\]
Each vertical sequence is exact; the bottom row is exact by \cite[Prop.~3.32.2]{Illusie1979}; the middle row, hence also the top row, becomes exact when assembled into a sequence of pro \'etale sheaves over $s\ge1$ by Theorem \ref{theorem_pro_higher_Cartier}. Taking \'etale cohomology constructs a diagram
\[\xymatrix@R=5mm{
\Pic(Y_s)\ar[r]^{p^r}\ar[d]&\Pic(Y_s)\ar[r]^{c_1}\ar[d] & H_\sub{\'et}^2(Y_s,\bb Z/p^r\bb Z(1)) \ar[d]\\
\Pic(Y)\ar[r]^{p^r}\ar[d]&\Pic(Y)\ar[d]^\delta\ar[r]^{c_1}&H_\sub{\'et}^2(Y,\bb Z/p^r\bb Z(1))\ar[d]\\
H^2_\sub{\'et}(Y_s,\bb G_{m,(Y_s,Y)})\ar[r]^{p^r}&H^2_\sub{\'et}(Y_s,\bb G_{m,(Y_s,Y)})\ar[r]^{c_{1,(Y_s,Y)}}& H^1_\sub{\'et}(Y_s,W_r\Omega_{(Y_s,Y),\sub{log}}^1)
}\]
Since the bottom row becomes exact when assembled into a sequence of pro abelian groups over $s\ge1$, the following is true: for any $s\ge1$ there exists $t\ge s$ such that the image of $\ker c_{1,(Y_t,Y)}$ in $H^2_\sub{\'et}(Y_s,\bb G_{m,(Y_s,Y)})$ lands inside the multiples of $p^r$. However, if $p^r\ge s$ then $p^r$ kills $\bb G_{m,(Y_s,Y)}$, and so the image of $\ker c_{1,(Y_t,Y)}$ in $H^2_\sub{\'et}(Y_s,\bb G_{m,(Y_s,Y)})$ is zero.

The proof of (i) is now completed by an easy diagram chase: if $c_1(L)$ lifts to $H^2(Y_t,\bb Z/p^r\bb Z(1))$, then $c_{1,(Y_t,Y)}$ kills $\delta(L)$, so $\delta(L)$ vanishes at level $s$, and so $L$ lifts to $\Pic(Y_s)$.

(ii): We may assemble the first diagram into one of pro \'etale sheaves indexed over the diagonal $r=s$ with exact rows and columns; note that in order to do this the transition maps in the $r$-direction in the left column are scaled by a factor of $p$ compared to the central column. In particular, the top left entry of the diagram becomes the pro \'etale sheaf \[\cdots\xto{p} \bb G_{m,(Y_3,Y)}\xto{p}\bb G_{m,(Y_2,Y)}\xto{p}\bb G_{m,(Y_1,Y)},\] which is zero since $p^r$ kills $\bb G_{m,(Y_s,Y)}$ whenever $p^r\ge s$. In conclusion we obtain a bicartesian diagram of pro \'etale sheaves on $Y$,
\begin{equation}\xymatrix{
\{\bb G_{m,Y_s}\}_s\ar[r]^-{\dlog[\cdot]}\ar[d] & \{W_s\Omega_{Y_s,\sub{log}}^1\}_s\ar[d]\\
\bb G_{m,Y}\ar[r]_{\dlog[\cdot]} & \{W_s\Omega_{Y_s,\sub{log}}^1\}_s
}\label{equ_line_heart}\end{equation}
and so taking continuous cohomology yields the following diagram of abelian groups with exact columns:
\[\xymatrix@R=5mm{
\vdots\ar[d]&\vdots\ar[d]\\
H^1_\sub{\'et}(Y,\{\bb G_{m,(Y_s,Y)}\}_s)\ar[d]\ar[r]^-\cong&H^1_\sub{\'et}(Y,\{W_s\Omega^1_{(Y_s,Y),\sub{log}}\}_s)\ar[d]\\
 H^1_\sub{\'et}(Y,\{\bb G_{m,Y_s}\}_s)\ar[r]^-{c_1}\ar[d] & H^1_\sub{\'et}(Y,\{W_s\Omega_{Y_s,\sub{log}}^1\}_s)\ar[d]\\
\Pic(Y)\ar[r]^-{c_1}\ar[d] & H^1_\sub{\'et}(Y,\{W_s\Omega_{Y,\sub{log}}^1\}_s)\ar[d]\\
H^2_\sub{\'et}(Y,\bb G_{m,(Y_s,Y)})\ar[d]\ar[r]_-\cong&H^2_\sub{\'et}(Y,\{W_s\Omega^1_{(Y_s,Y),\sub{log}}\}_s)\ar[d]\\
\vdots&\vdots
}\]
The middle vertical arrows may be factored respectively as 
\[ H^1_\sub{\'et}(Y,\{\bb G_{m,Y_s}\}_s)\xto{(1)} \projlim_s\Pic(Y_s)\to \Pic(Y)\] and
\[H^1_\sub{\'et}(Y,\{W_s\Omega_{Y_s,\sub{log}}^1\}_s)\xto{(2)} \projlim_sH^1_\sub{\'et}(Y_s,\{W_r\Omega_{Y_s,\sub{log}}^1\}_r) \to H^1_\sub{\'et}(Y,\{W_r\Omega_{Y,\sub{log}}^1\}_r),\]
and an easy diagram chase will complete the proof as soon as it is shown that arrows (1) and (2) are surjective. But standard formalism of continuous cohomology and iterated inverse limits (see the next remark) indeed imply that there are short exact sequences
\[0\To{\projlim_s}^1H^0_\sub{\'et}(Y_s,\bb G_{m,Y_s})\To H^1_\sub{\'et}(Y,\{\bb G_{m,Y_s}\}_s)\xto{(1)} \projlim_s\Pic(Y_s)\To 0\] and
\[0\To{\projlim_s}^1\projlim_rH^0_\sub{\'et}(Y_s,W_r\Omega_{Y_s,\sub{log}}^1)\To H^1_\sub{\'et}(Y,\{W_s\Omega_{Y_s,\sub{log}}^1\}_s)\xto{(2)} \projlim_sH^1_\sub{\'et}(Y_s,\{W_r\Omega_{Y_s,\sub{log}}^1\}_r)\To 0.\]
\end{proof}

\begin{remark}\label{remark_ses_in_continuous_cohomology}
Since we will need to make a similar argument when treating higher codimension cycles, we now explain the second short exact sequence arising at the end of the previous proof. Let $\{\cal F_{r,s}\}_{r,s}$ be a $\bb N^2$-indexed inverse system of sheaves of abelian groups on a reasonable site, such as the \'etale site of a scheme. Then we claim that there is a short exact sequence in continuous cohomology for each $n\ge0$ \[\textstyle0\To \projlim_s^1\projlim_rH^{n-1}(\cal F_{r,s})\To H^n_\sub{cont}(\{\cal F_{s,s}\}_s)\To \projlim_sH^n_\sub{cont}(\{\cal F_{r,s}\}_r) \To 0.\]  This will follow from the commutative diagram of abelian groups
\[\xymatrix@=5mm{
&0\ar[d]&&&\\
&\projlim_s^1\projlim_rH^{n-1}(\cal F_{r,s})\ar[d]&&&\\
0\ar[r] &\projlim_s^1H^{n-1}(\cal F_{s,s})\ar[r]\ar[d]&H^n_\sub{cont}(\{\cal F_{s,s}\}_s)\ar[r]\ar[d]&\projlim_sH^n(\cal F_{s,s})\ar[r]\ar[d]&0\\
0\ar[r] &\projlim_s\projlim_r^1H^{n-1}(\cal F_{r,s})\ar[d]\ar[r]&\projlim_sH^n_\sub{cont}(\{\cal F_{r,s}\}_r)\ar[r]&\projlim_s\projlim_rH^n(\cal F_{r,s})\ar[r]&0\\
&0&&&
}\]
The left column is an exact sequence arising from a spectral sequence  $E_2^{ab}=\projlim_s^a\projlim_r^bH^{n-1}(\cal F_{r,s})\Rightarrow\projlim_s^{a+b}H^{n-1}(\cal F_{s,s})$ of Roos and the fact that the only non-zero derived inverse limit over a countable system of abelian groups is $\projlim^1$\cite{Roos1961}. The central row is the usual short exact sequence in continuous cohomology. The bottom row results from applying $\projlim_s$ to the short exact sequence in continuous cohomology associated to $\{\cal F_{r,s}\}_r$; in principle there is therefore an obstruction term $\projlim_s^1\projlim_r^1H^{n-1}(\cal F_{r,s})$ to right exactness, but this vanishes since it equals $\projlim_s^2H^{n-1}(\cal F_{s,s})=0$ by Roos. Since the right vertical arrow is an isomorphism, the desired short exact sequence now follows from a diagram chase.
\end{remark}

\begin{remark}[The fppf approach]\label{remark_flat_approach}
Under the hypotheses of Theorem \ref{theorem_deforming_line_bundles}, Corollary \ref{corollary_fppf_mu_n} says that \[\{H^*_\sub{fppf}(Y_s,\bm\mu_{p^r,Y_s})\}_s\isoto \{H^*_\sub{\'et}(Y_s,\bb Z/p^r(1))\}_s.\] Hence Theorem \ref{theorem_deforming_line_bundles} may be equivalently stated as follows:
\begin{enumerate}
\item Given $r\ge 1$ there exists $t\ge p^r$ (depending only on $\cal Y$, not $L$) such that, if the fppf class of $L$ inside $H^2_\sub{fppf}(Y,\bm\mu_{p^r,Y})$ lifts to $H^2_\sub{fppf}(Y_t,\bm\mu_{p^r,Y_t})$, then $L$ lifts to $\Pic(Y_{p^r})$.
\item If the fppf class of $L$ in $H^2_\sub{fppf}(Y,\bm\mu_{p^\infty,Y})$ lifts to some $c\in\projlim_sH^2_\sub{fppf}(Y_s,\bm\mu_{p^\infty,Y_s})$, then there exists $\tilde L\in\projlim_s\Pic(Y_s)$ which lifts $L$ and whose fppf class is $c$.
\end{enumerate}
\end{remark}

We finish the section with two standard consequences of Theorem \ref{theorem_deforming_line_bundles}:

\begin{corollary}\label{corollary_line_bundle_over_power_series}
Let $A$ be a Noetherian, F-finite $\bb F_p$-algebra which is complete with respect to an ideal $I\subseteq X$, and let $X$ be a proper scheme over $A$; assume that $X$ and the special fibre $Y:=X\times_AA/I$ are regular. Then a line bundle $L\in\Pic(Y)$ lifts to $\Pic(X)$ if and only if $c_1(L)\in H_\sub{\'et}^2(Y,\bb Z_p(1))$ lifts to $H_\sub{\'et}^2(X,\bb Z_p(1))$.
\end{corollary}
\begin{proof}
The implication ``only if'' is a consequence of functoriality. Conversely, if $c_1(L)$ lifts to $H_\sub{\'et}^2(X,\bb Z_p(1))$ then it certainly lifts to $\projlim_sH_\sub{\'et}^2(Y_s,\bb Z_p(1))$, whence Theorem \ref{theorem_deforming_line_bundles}(ii) (for the completion of $X$ along $Y$) implies that $L$ lifts to $\projlim_s\Pic(Y_s)$. But Grothendieck's algebrisation theorem for line bundles states that $\Pic(X)\to\projlim_s\Pic(Y_s)$ is an isomorphism.
\end{proof}

\begin{corollary}\label{corollary_line_bunde_on_hyperplane}
Let $k$ be a field of characteristic $p$ having a finite $p$-basis, $X$ a smooth projective $k$-variety, $Y\into X$ a smooth ample divisor, and $L\in\Pic(Y)$. Then $L$ lifts to $\Pic(X)$ if and only if $c_1(L)\in H_\sub{\'et}^2(Y,\bb Z_p(1))$ lifts to $H_\sub{\'et}^2(X,\bb Z_p(1))$.
\end{corollary}
\begin{proof}
As in the previous corollary, the implication ``only if'' is a consequence of functoriality. Conversely, if $c_1(L)$ lifts to $H_\sub{\'et}^2(X,\bb Z_p(1))$ then the same argument as the previous corollary shows that $L$ lifts to $\projlim_s\Pic(Y_s)$. But the Grothendieck--Lefschetz theorems \cite[Exp.~XI]{SGA_II} imply that any element of $\projlim_s\Pic(Y_s)$ lifts to $\Pic(U)$ for some open $U\subseteq X$ containing $Y$, and the normality of $X$ implies that $\Pic(X)\to\Pic(U)$ is surjective; hence $L$ lifts all the way to $\Pic(X)$.
\end{proof}

\begin{remark}[Relation to earlier results]\label{remark_relation_to_Tate_paper}
Weaker forms of the previous two corollaries were presented in \cite{Morrow_Variational_Tate}: in Corollary \ref{corollary_line_bundle_over_power_series} it was previously required that $A=k[[t_1,\dots,t_d]]$, with $k$ a finite or algebraically closed field, and that $X$ was smooth over $A$; moreover the Picard and cohomology groups were tensored by $\bb Z[\tfrac1p]$; and in Corollary \ref{corollary_line_bunde_on_hyperplane} it was previously required that $k$ be perfect, and the Picard and cohomology groups were tensored by $\bb Q$. Those proofs used higher algebraic $K$-theory and topological cyclic homology (essentially as the special case when $n=1$ of Section \ref{subsection_higher_codim}), which of course was very unsatisfactory, and a goal of this section has been to give self-contained and purely algebraic proofs.

Although it does seem that the previous two corollaries are known to some experts, the only reference we know containing similar results are two unpublished manuscripts of de Jong \cite{deJong20??, deJong2011}, both of which predate \cite{Morrow_Variational_Tate}. He essentially proved the fppf formulation of Theorem \ref{theorem_deforming_line_bundles} which we explained in Remark \ref{remark_flat_approach}, in the case that $Y$ is a smooth ample divisor on a smooth projective variety of dimension $\ge 3$ over a finite field. The major motivation for proving Corollary \ref{corollary_fppf_mu_n} was exactly to relate de Jong's fppf approach to deforming line bundles with the logarithmic Hodge--Witt approach implicit in \cite{Morrow_Variational_Tate}.
\end{remark}

\begin{remark}[Pro Gersten vanishing ($n=1$)]\label{remark_pro_Gersten_1}
Let $\cal Y$ be a regular, F-finite, formal $\bb F_p$-scheme whose reduced subscheme of definition $Y=Y_1$ is regular. Then we will show at the end of this remark that there is a long exact sequence of pro abelian groups
\[\hspace{-5mm}0\to \{H^2_\sub{Zar}(Y_s,\bb G_{m,Y_s,\sub{Zar}})\}_s
\to \{H^2_\sub{\'et}(Y_s,\bb G_{m,Y_s,\sub{\'et}})\}_s
\to H^2_\sub{\'et}(Y,\bb G_{m,Y,\sub{\'et}})
\to \{H^3_\sub{Zar}(Y_s,\bb G_{m,Y_s,\sub{Zar}})\}_s\to\cdots\] relating the Zariski and \'etale cohomologies of $\bb G_m$.

We explicitly state this consequence of our calculations because we have not seen it previously and for the sake of later reference concerning a pro Gersten conjecture (Remark \ref{remark_pro_Gersten}); this conjecture (or rather, question), in the case $n=1$, asks exactly whether the pro Zariski cohomologies appearing in the following sequence vanish, i.e., whether \[\{H^i_\sub{\'et}(Y_s,\bb G_{m,Y_s,\sub{\'et}})\}_s\Isoto H^i_\sub{\'et}(Y,\bb G_{m,Y,\sub{\'et}})\] for all $i\ge2$. We do not know whether this is true.

To prove the existence of the long exact sequence, we must show that the following square of pro complexes of Zariski sheaves on $Y$ is homotopy cartesian:
\[\xymatrix{
\{\bb G_{m,Y_s,\sub{Zar}}\}_s \ar[r]\ar[d] & \{R\ep_*\bb G_{m,Y_s,\sub{\'et}}\}_s\ar[d]\\
\bb G_{m,Y,\sub{Zar}}\ar[r] & R\ep_*\bb G_{m,Y,\sub{\'et}}
}\]
(The associated long exact sequence in Zariski cohomology is exactly what we seek, since $H^i_\sub{Zar}(Y,\bb G_{m,Y,\sub{Zar}})=0$ if $i>1$ and $H^1_\sub{Zar}(Y,\bb G_{m,Y,\sub{Zar}})\isoto H^1_\sub{\'et}(Y,\bb G_{m,Y,\sub{\'et}})$.) By repeating the proof of Theorem \ref{theorem_deforming_line_bundles} in the Zariski topology, one sees that the Zariski analogue of square (\ref{equ_line_heart}), i.e.,
\[\xymatrix{
\{\bb G_{m,Y_s,\sub{Zar}}\}_s\ar[r]^-{\dlog[\cdot]}\ar[d] & \{W_s\Omega_{Y_s,\sub{log},\sub{Zar}}^1\}_s\ar[d]\\
\bb G_{m,Y,\sub{Zar}}\ar[r]_{\dlog[\cdot]} & \{W_s\Omega_{Y_s,\sub{log},\sub{Zar}}^1\}_s
}\]
is a bicartesian square of pro Zariski sheaves on $Y$. Comparing with $R\ep_*$ of (\ref{equ_line_heart}), it is necessary and sufficient to show that the following square of pro complexes of Zariski sheaves on $Y$ is homotopy cartesian:
\[\xymatrix{
 \{W_s\Omega_{Y_s,\sub{log},\sub{Zar}}^1\}_s\ar[r]\ar[d] & \{R\ep_*W_s\Omega_{Y_s,\sub{log},\sub{\'et}}^1\}_s\ar[d]\\
\{W_s\Omega_{Y_s,\sub{log},\sub{Zar}}^1\}_s\ar[r] & \{R\ep_*W_s\Omega_{Y_s,\sub{log},\sub{\'et}}^1\}_s
}\]
This square is in fact homotopy cartesian for each fixed $s$; we will prove this in a moment in Corollary \ref{corollary_Zar_to_etale} in greater generality.
\end{remark}

\section{Logarithmic Hodge--Witt sheaves ($n\ge1$)}\label{section_hw_n>1}
This section is a systematic study of the logarithmic Hodge--Witt sheaves on general $\bb F_p$-schemes and regular formal $\bb F_p$-schemes. We refer the reader to Sections \ref{subsection_HW_sheaves} and \ref{subsection_formal_schemes} for relevant notation used in this section.

\subsection{Logarithmic Hodge--Witt sheaves on arbitrary $\bb F_p$-schemes}\label{subsection_logarithmic}
We now apply the theory developed in Section \ref{section_drw} to study the logarithmic Hodge--Witt groups of arbitrary $\bb F_p$-algebras, often by using Proposition \ref{proposition_limit_of_F-fixed_points} to lift to the regular case.

\begin{corollary}\label{corollary_log_vs_R-F}
Let $n\ge0$ and  $r\ge 1$.
\begin{enumerate}
\item If $A$ is an $\bb F_p$-algebra, then $R(W_r\Omega_A^{n,F=1})\subseteq W_{r-1}\Omega_{A,\sub{log},{Zar}}^n$.
\item If $Y$ is a scheme on which $p$ is locally nilpotent, then $R-F:W_r\Omega_Y^n\To W_{r-1}\Omega_Y^n$ is surjective in the \'etale topology.
\item If $Y$ is an $\bb F_p$-scheme, then the sequence of pro \'etale sheaves 
\[0\To\{W_r\Omega_{Y,\sub{log}}^n\}_r\To\{W_r\Omega_Y^n\}_r\xto{1-F} \{W_r\Omega_Y^n\}_r\To 0\]
on $Y$ is exact.
\end{enumerate}
\end{corollary}
\begin{proof}
(i): If $A$ is smooth over a perfect field of characteristic $p$, then this is due to Illusie \cite[I.5.7.4]{Illusie1979} as long as $W_r\Omega^n_{A,\sub{log},\sub{Zar}}$ is replaced by $W_r\Omega^n_{A,\sub{log}}$ (take $m=0$ and use the identity $\op{Fil}^nW_r\Omega_A^n=\ker R$ from Prop.~I.3.4 of op.~cit.); but these two groups of dlog forms are actually the same, by Theorem \ref{theorem_Zar=et}. Hence (i) is true whenever $A$ is a regular $\bb F_p$-algebra, since $A$ is then a filtered inductive limit of smooth $\bb F_p$-algebras, by N\'eron--Popescu desingularisation.

To prove the claim for an arbitrary $A$, it is enough, by taking a filtered inductive limit, to consider the case that $A$ is a finite type $\bb F_p$-algebra. Hence we may write $A$ as a quotient, $A=B/I$, of a smooth $\bb F_p$-algebra $B$ (we could even let $B$ be a polynomial algebra). Let $\omega\in W_r\Omega_A^{n,F=1}$. By Theorem \ref{proposition_limit_of_F-fixed_points}(ii), $\omega$ may be lifted to some $\tilde\omega\in W_r\Omega_{\hat B}^{n,F=1}$, where $\hat B$ is the $I$-adic completion of $B$, and our observations in the first paragraph imply that $R(\tilde \omega)\in W_{r-1}\Omega_{\hat B, \sub{log},\sub{Zar}}^n$. Reducing to $A$, we deduce that $R(\omega)\in W_{r-1}\Omega_{A,\sub{log},\sub{Zar}}^n$, as required.

(ii): By first reducing to the case that $Y$ is of finite type over $\bb Z$ and then passing to a point in the \'etale topology, we may suppose that $Y=\Spec A$, where $A$ is a Noetherian, F-finite, local, strictly Henselian $\bb Z_{(p)}$-algebra in which $p$ is nilpotent. Let $\hat A$ be the completion of $A$ at its maximal ideal $\frak m$.

We now apply a standard Artin Approximation trick using N\'eron--Popescu desingularisation. For any $A$-algebra $B$, set $X(B):=\op{Coker}(W_r\Omega_{(B,\frak mB)}^n\xto{R-F}W_{r-1}\Omega_{(B,\frak mB)}^n)$; then the functor $X:A\op{-algs}\to Ab$ commutes with filtered inductive limits, which implies that $X(A)\to X(\hat A)$ is injective. (Proof: By excellence the map $A\to\hat A$ is geometrically regular, so N\'eron--Popescu desingularisation implies we may write $\hat A$ as a filtered inductive limit of smooth $A$-algebras $B$. For each such $B$, the map $A/\frak m\to B/\frak m B$ has a section induced by $B/\frak mB\to \hat A/\frak m\hat A=A/\frak m$; since $A$ is Henselian, the map $A\to B$ therefore also has a section, and so $X(A)\to X(B)$ has a section.) Proposition \ref{proposition_limit_of_F-fixed_points}(iii) implies that $X(\hat A)=0$, whence $X(A)=0$, i.e., $R-F:W_r\Omega_{(A,\frak m)}^n\to W_{r-1}\Omega_{(A,\frak m)}^n$ is surjective.

To complete the proof, it is now enough to check that $R-F:W_r\Omega_k^n\to W_{r-1}\Omega_k^n$ is surjective, where $k:=A/\frak m$ is a separably closed field of characteristic $p$ with finite $p$-basis (since $A$, hence $k$, is F-finite). But this surjectivity follows from Illusie's aforementioned result in the smooth case.

(iii): This is an immediate consequence of (i) and (ii).
\end{proof}

Using similar arguments, we can also establish the coincidence of Zariski and \'etale logarithmic forms in general, eliminating the regularity hypothesis in Theorem \ref{theorem_Zar=et}:

\begin{corollary}\label{corollary_Zar_vs_etale}
Let $A$ be an $\bb F_p$-algebra. Then
\begin{enumerate}
\item The inclusions $W_r\Omega_{A,\sub{log},\sub{Zar}}^n\subseteq W_r\Omega_{A,\sub{log},\sub{Nis}}\subseteq W_r\Omega_{A,\sub{log}}^n$ are equalities.
\item There exists a unique map $\res F:W_r\Omega_A^n\to W_r\Omega_A^n/dV^{r-1}\Omega_A^{n-1}$ making the following diagram commute, in which $\pi$ denotes the canonical quotient map:
\[\xymatrix@C=1.5cm{
W_{r+1}\Omega_A^n\ar[d]_R\ar[r]^F & W_r\Omega_A^n\ar[d]^\pi\\
W_r\Omega_A^n\ar[r]^{\res F} & W_r\Omega_A^n/dV^{r-1}\Omega_A^{n-1}
}\]
\item The sequence $0\To W_r\Omega_{A,\sub{log}}^n\To W_r\Omega_A^n\xto{\pi-\res F}W_r\Omega_A^n/dV^{r-1}\Omega_A^{n-1}$ is exact.
\end{enumerate}
\end{corollary}
\begin{proof}
(ii): The existence of $\res F$ is an immediate consequence of the fact that $\ker R=V^r\Omega_A^n+dV^{r-1}\Omega_A^{r-1}$ (see the paragraph after Definition \ref{definition_filtrations}) and the standard identities $FV=p$, $FdV=d$.

(i) \& (iii): We will show that the kernel of $\pi-\res F$ (which contains $W_r\Omega_{A,\sub{log}}^n$) is contained in $W_r\Omega^n_{A,\sub{log},\sub{Zar}}$, which will simultaneously establish (i) and (iii). 
We consider the following commutative diagram
\[\xymatrix@C=2cm{
W_{r+1}\Omega_A^n\ar[r]^{R-F}\ar[d]_R & W_r\Omega_A^n\ar[d]^\pi \\
W_r\Omega_A^n\ar[r]^{\pi-\res F} & W_r\Omega_A^n/dV^{r-1}\Omega_A^{n-1}
}\]
in which the induced map on the kernels of the verticals arrows is obviously surjective: indeed, if $\omega\in\Omega_A^{n-1}$ then $dV^{r-1}\omega=FdV^r\omega=(R-F)dV^r(-\omega)$ where $dV^r(-\omega)\in\ker R$.


Hence the induced map on the kernels of the left horizontal arrows is surjective; i.e., $\ker(\pi-\res F)=R(W_{r+1}\Omega_A^{n,F=1})$, which is contained in $W_r\Omega_{A,\sub{log},\sub{Zar}}^n$ by Corollary \ref{corollary_log_vs_R-F}(i).
\end{proof}

\begin{remark}\label{remark_log_forms_cartier_filtration}
When $r=1$, the map $\res F:\Omega_A^n\to\Omega_A^n/d\Omega_A^{n-1}$ of the Corollary \ref{corollary_Zar_vs_etale} is course the inverse Cartier map $C^{-1}$ from Section \ref{subsection_cartier} (this can be directly checked on a typical element of $\Omega_A^n$; see \cite[Prop.~I.3.3]{Illusie1979} for the argument in the smooth case, which works in general). Part (iii) of the corollary then admits a generalisation to the Cartier filtration; to state it note that, by definition of the Cartier filtration, the inverse Cartier map induces a map $C^{-1}:\Omega_A^n/B_i\Omega_A^n\to\Omega_A^n/B_{i+1}\Omega_A^n$; also, for the sake of precise notation, let $\pi_i:\Omega_A^n/B_i\Omega_A^n\to\Omega_A^n/B_{i+1}\Omega_A^n$ denote the canonical quotient map. Then we claim that the sequence \[\Omega^n_{A,\sub{log},\sub{Zar}}\To\Omega_A^n/B_i\Omega_A^n\xto{\pi_i-C^{-1}}\Omega_A^n/B_{i+1}\Omega_A^n\] is exact.

Indeed, there is a commutative diagram of pro abelian groups with exact columns
\[\xymatrix@R=5mm{
0&0\\
\Omega_A^n/B_{i+1}\Omega_A^n\ar[u]\ar[r]^{\pi_{i+1}-C^{-1}}&\Omega_A^n/B_{i+2}\Omega_A^n\ar[u]\\
\Omega_A^n/B_i\Omega_A^n\ar[u]\ar[r]^{\pi_i-C^{-1}}&\Omega_A^n/B_{i+1}\Omega_A^n\ar[u]\\
B_{i+1}\Omega_A^n/B_i\Omega_A^n\ar[u]\ar[r]^{-C^{-1}}&B_{i+2}\Omega_A^n/B_{i+1}\Omega_A^n\ar[u]\\
0\ar[u]&0\ar[u]
}\]
The bottom horizontal arrow is surjective by definition of the Cartier filtration, whence the kernel of the middle horizontal arrow surjects onto the kernel of the top horizontal arrow. So a trivial induction reduces our claim to the case $i=0$, which is exactly Corollary \ref{corollary_Zar_vs_etale}(iii).
\end{remark}

The next corollary, in which $\ep$ (resp.~$\ep_\sub{Nis}$) denotes the projection map from the \'etale (resp.~Nisnevich) topos to the Zariski topos, will play an important role later when passing between the Zariski, Nisnevich, and \'etale topologies:

\begin{corollary}\label{corollary_Zar_to_etale}
Let $Y\into Y'$ be a nilpotent thickening of $\bb F_p$-schemes and $n\ge0$. Then the canonical maps of pro complexes of Zariski sheaves \[\{W_r\Omega_{(Y',Y),\sub{log},\sub{Zar}}^n\}_r\To \{R\ep_{\sub{Nis}*}W_r\Omega_{(Y',Y),\sub{log},\sub{Nis}}^n\}_r\To \{R\ep_*W_r\Omega_{(Y',Y),\sub{log}}^n\}_r\] are quasi-isomorphisms.
\end{corollary}
\begin{proof}
Firstly, $\ep_*W_r\Omega_{(Y',Y),\sub{log}}^n=W_r\Omega_{(Y',Y),\sub{log},\sub{Zar}}^n$ by Corollary \ref{corollary_Zar_vs_etale}(i), which forces the inclusions $W_r\Omega_{(Y',Y),\sub{log},\sub{Zar}}\subseteq \ep_{\sub{Nis}*}W_r\Omega_{(Y',Y),\sub{log},\sub{Nis}}\subseteq\ep_*W_r\Omega_{(Y',Y),\sub{log}}$ to all be equalities.

Next, Corollary \ref{corollary_log_vs_R-F}(iii) for both $Y$ and $Y'$ yields a short exact sequence of pro \'etale sheaves on $Y'$ \[0\To \{W_r\Omega_{(Y',Y),\sub{log}}^n\}_r\To \{W_r\Omega_{(Y',Y)}^n\}_r\xTo{1-F}\{W_r\Omega_{(Y',Y)}^n\}_r\To 0.\] Applying $R^i\ep_*$ (and recalling from Remark \ref{remark_drw_etale_base_change} the standard result that the Hodge--Witt sheaves have vanishing higher cohomology on affines), we see firstly that $\{R^i\ep_*W_r\Omega_{(Y',Y),\sub{log}}^n\}_r=0$ for $i>1$ and secondly that \[\{R^1\ep_*W_r\Omega_{(Y',Y),\sub{log}}^n\}_r=\{\op{Coker}(\ep_*W_r\Omega_{(Y',Y)}^n\xTo{R-F}\ep_*W_{r-1}\Omega_{(Y',Y)}^n)\}_r.\] But $R-F$ is Zariski locally surjective in the relative nilpotent case by Lemma \ref{lemma_surj_of_1-F_in_nilp_case}(ii), and so it follows that $\{R^1\ep_*W_r\Omega_{(Y',Y),\sub{log}}^n\}_r=0$. This proves that the composition of the two maps in the statement of the corollary is a quasi-isomorphism.

The previous paragraph works verbatim if $\ep$ is replace by the projection from the \'etale topos to the Nisnevich topology, which we momentarily denote by $\ep'$, thereby showing that $\{W_r\Omega_{(Y',Y),\sub{log},\sub{Nis}}^n\}_r\quis\{R\ep'_*W_r\Omega_{(Y',Y),\sub{log}}^n\}_r$. By now applying $R\ep_{\sub{Nis}*}$ and using the assertion about the composition which has already been proved, it follows that also $\{W_r\Omega_{(Y',Y),\sub{log}}^n\}_r\quis\{R\ep_{\sub{Nis}*}W_r\Omega_{(Y',Y),\sub{log},\sub{Nis}}^n\}_r$.
\end{proof}

The next result is needed in our later application to the weak Lefschetz conjecture for Chow groups, as it will allow us to related logarithmic Hodge--Witt cohomology to Frobenius-fixed points in crystalline cohomology:

\begin{corollary}\label{corollary_p_power_eigenspace}
Let $Y$ be an $\bb F_p$-scheme, and fix $n\ge0$; let $\phi:x\to x^p$ denote the absolute Frobenius. Then the sequence of complexes of \'etale sheaves \[0\To W_r\Omega_{Y,\sub{log}}^n[-n]\To W_r\Omega_Y^\blob\xTo{p^n-\phi}W_r\Omega_Y^\blob\To0\] is exact up to torsion killed by $p^{n+1}$ (i.e., the three cohomology sheaves of the sequence are killed by $p^{n+1}$).
\end{corollary}
\begin{proof}
We will show that, in the \'etale topology:
\begin{enumerate}
\item If $i\neq n$, the kernel and cokernel of $p^n-\phi:W_r\Omega_Y^i\to W_r\Omega_Y^i$ are killed by $p^{n+1}$.
\item The cokernel of $p^n-\phi:W_r\Omega_Y^n\to W_r\Omega_Y^n$ is $W_r\Omega_Y^n/p^nW_r\Omega_Y^n$.
\item The inclusion $W_r\Omega_{Y,\sub{log}}^n\subseteq \ker(p^n-\phi:W_r\Omega_Y^n\to W_r\Omega_Y^n)$ has cokernel killed by $p^{n+1}$.
\end{enumerate}
To prove (i), note that the composition (left or right) of $p^n-\phi$ with $\sum_{j=0}^{r-1}p^{(n-i-1)j}(RV)^{j+1}$ (resp.~$\sum_{j=0}^r(p^{i-n}\phi)^j$) is $p^{n+1}$ (resp.~$p^n$) if $i<n$ (resp.~$i>n$).

(ii) is an immediate consequence of the identity $(p^n-\phi)R=p^n(R-F)$ and the fact that both maps $R, R-F:W_{r+1}\Omega_Y^n\to W_r\Omega_Y^n$ are surjective ($R-F$ by Corollary \ref{corollary_log_vs_R-F}(ii)).

(iii): If $\omega$ is a section of $W_r\Omega_Y^n$ killed by $p^n-\phi$ then it is also killed by $R(p^n-\phi)=(R-F)p^n$, whence $R(p^n\omega)\in W_{r-1}\Omega_{Y,\sub{log}}^n$ by Corollary \ref{corollary_log_vs_R-F}(i); since $R:W_r\Omega_{Y,\sub{log}}^n\to W_{r-1}\Omega_{Y,\sub{log}}^n$ is surjective, there therefore exists $\omega'\in W_r\Omega_{Y,\sub{log}}^n$ such that $\omega'-p^n\omega\in\op{Fil}^{r-1}W_r\Omega_Y^n$. But $\op{Fil}^{r-1}W_r\Omega_Y^n=\op{Fil}^{r-1}_VW_r\Omega_Y^n\subseteq \op{Fil}^{r-1}_pW_r\Omega_Y^n$ by the properties of the filtrations recalled after Definition \ref{definition_filtrations}, whence $p^{n+1}\omega=p\omega'$, as required.
\end{proof}

\subsection{The structure of the $p$-filtration for a regular, formal $\bb F_p$-scheme}
\label{subsection_dlog_n>1}
The following theorem extends to regular formal schemes, and to the Zariski/Nisnevich topology, Illusie's analysis of the $p$-filtration on logarithmic Hodge--Witt forms in the \'etale topology on a smooth variety (see Remark \ref{remark_de_Rham--Witt_log}). It is a key technical step in the paper, without which many more of our results would only be stated in terms of double-indxed pro abelian groups such as $\{W_s\Omega_{A/I^s,\sub{log}}^n\}$.

\begin{theorem}\label{theorem_de_Rham--Witt_log}
Let $\cal Y$ be a regular F-finite formal $\bb F_p$-scheme, and $Y_1\into\cal Y$ a subscheme of definition; fix $n\ge 0$ and $1\le i\le r$. Letting $\tau$ denote the Zariski, Nisnevich, or \'etale topology, the following sequence of pro $\tau$-sheaves on $Y_1$ is exact:
\[\{W_r\Omega_{Y_s,\sub{log},\tau}^n\}_s\xTo{p^i}\{W_r\Omega^n_{Y_s,\sub{log},\tau}\}_s\xTo{R^{r-i}}\{W_i\Omega^n_{Y_s,\sub{log},\tau}\}_s\To0\]
\end{theorem}

\begin{remark}\label{remark_de_Rham--Witt_log}
If $Y$ is a smooth variety over a perfect field of characteristic $p$, then Illusie \cite[\S I.5.7]{Illusie1979} proved exactness of the sequence of \'etale sheaves \[W_r\Omega_{Y,\sub{log}}^n\xTo{p^i}W_r\Omega_{Y,\sub{log}}^n\xTo{R^{r-i}}W_i\Omega_{Y,\sub{log}}^n\To 0.\] To understand the proof of Theorem \ref{theorem_de_Rham--Witt_log}, which initially appears rather technical and unmotivated, we advise the reader to look at the proof of Illusie's result, which is largely a formal manipulation of the exact sequences which we quoted at the beginning of the proof of Lemma \ref{lemma_Cartier_drw}. Our proof is a modification of these manipulations, in which we must take care of certain pro systems, and which also works in the Zariski and Nisnevich topologies since we have Corollary \ref{corollary_Zar_vs_etale}(i) at our disposal.
\end{remark}

\begin{proof}[Proof of Theorem \ref{theorem_de_Rham--Witt_log}]
For any $\bb F_p$-scheme $Y$, the restriction map $R^{r-i}:W_r\Omega_{Y,\sub{log},\tau}^n\to W_i\Omega_{Y,\sub{log},\tau}^n$ clearly kills $p^iW_r\Omega_{Y,\sub{log},\tau}^n$ and is moreover a surjection of  $\tau$-sheaves since both sides are by definition quotients of $\bb G_{m,Y}^{\otimes n}$. This proves exactness except at the the middle, which is the following inclusion of pro  $\tau$-sheaves:
\begin{equation}\{\ker(W_r\Omega^n_{Y_s,\sub{log},\tau}\xto{R^{r-i}}W_i\Omega^n_{Y_s,\sub{log},\tau})\}_s\subseteq \{p^iW_r\Omega^n_{Y_s,\sub{log},\tau}\}_s.\label{theorem_middle_exactness}\end{equation} 
It is convenient to rewrite the left side as $\{R(\ker(W_{r+1}\Omega^n_{Y_s,\sub{log},\tau}\xto{R^{r+1-i}}W_i\Omega^n_{Y_s,\sub{log},\tau})\}_s$, by surjectivity of the restriction map.

To begin the proof of (\ref{theorem_middle_exactness}), let $\Spf A$ be an affine open of $\cal Y$, where $A$ is a regular F-finite $\bb F_p$-algebra which is complete with respect to some ideal $I\subseteq A$. Given integers $t\ge s$, we will denote by $\tau_{t,s}:A/I^t\to A/I^s$ the canonical map, as well as the canonical map it induces on the de Rham--Witt complex, etc. Then the inclusion (\ref{theorem_middle_exactness}), taking into account the rewriting of the left side and using the terminology of Section \ref{subsection_formal_schemes}, follows from (is even equivalent to, if $\tau=$ \'et) the next assertion:

\begin{quote}
For each $s\ge1$ there exists $t\ge s$ such that the following inclusion holds for any tft $I$-formally \'etale $A$-algebra $A'$:\end{quote} \[\tau_{t,s}(\Gamma(\Spec A'/I^tA',W_{r+1}\Omega_{Y_t,\sub{log},\tau}^n)\cap\op{Fil}^iW_{r+1}\Omega_{A'/I^tA'}^n))\subseteq \Gamma(\Spec(A'/I^sA'), p^iW_r\Omega^n_{Y_s,\sub{log},\tau})\]
Since dlog forms are contained in the kernel of $R-F$, and $\tau$ is Zariski, Nisnevich, or \'etale, the following inclusion is stronger:
\begin{equation}\tau_{t,s}(R(W_{r+1}\Omega_{A'/I^tA'}^{n,F=1}\cap\op{Fil}^iW_{r+1}\Omega_{A'/I^tA'}^n))\subseteq \Gamma(\Spec(A'/I^sA'), p^iW_r\Omega^n_{Y_s,\sub{log},Zar}),\label{equation_log_fil_rtp}\end{equation}
and it is this which we shall eventually show is true to complete the proof of the theorem. Because we closely follow Illusie's proof of \cite[(5.7.4)]{Illusie1979}, it is actually more convenient to prove (\ref{equation_log_fil_rtp}) with $r-1$ in place of $r$.

So we begin the proof properly now by picking integers $s_2,s_1,s_0=:t$ as follows:
\begin{itemize}\itemsep0pt
\item Pick $s_2\ge s$ such that $\tau_{s_2,s}(B_i\Omega^n_{A/I^{s_2}})$ is killed by $V^i$, by Lemma \ref{lemma_Cartier_drw}i.
\item Pick $s_1\ge s_2$ such that $\tau_{s_1,s_2}(\ker(\Omega_{A/I^{s_1}}^n\xto{V^r} W_{i+1}\Omega_{A/I^{s_1}}^n/dV^i\Omega_{A/I^{s_1}}^n))\subseteq B_{i+1}\Omega_{A/I^{s_2}}^n$, by Lemma \ref{lemma_Cartier_drw}ii.
\item Pick $t:=s_0\ge s_1$ such that $\tau_{s_0,s_1}(\ker(\Omega_{A/I^{s_0}}^n\xto{dV^{i-1}}W_i\Omega_{A/I^{s_0}}))\subseteq F^iW_{i+1}\Omega_{A/I^{s_1}}$, by Lemma \ref{lemma_Cartier_drw}iv.
\end{itemize}
We will first prove that \begin{equation}\tau_{t,s}R^{r-i-1}\left(W_r\Omega_{A/I^t}^{n,F=1}\cap\op{Fil}^iW_r\Omega_{A/I^t}^n\right)\subseteq V^i\Omega_{A/I^s,\sub{log},\sub{Zar}}^n\label{equ_log_fil}\end{equation} for $i<r-1$. In an attempt to keep the proof readable, we will omit explicitly writing the maps $\tau$ and simply ``pass (from level $t$) to level  $s$''.

Let $x\in W_r\Omega_{A/I^t}^{n,F=1}\cap\op{Fil}^iW_r\Omega_{A/I^t}^n$. Since the canonical and $V$-filtrations coincide (paragraph after Definition \ref{definition_filtrations}) we may write $x=V^iy+dV^iz$ for some $y\in W_{r-i}\Omega_{A/I^{s_0}}^n$, $z\in W_{r-i}\Omega_{A/I^{s_0}}^{n-1}$. Applying $R^{r-i-1}$ to the assumption that $Fx=Rx$ yields \[pV^{i-1}R^{r-i-1}y+dV^{i-1}R^{r-i-1}z=V^iR^{r-i}y+dV^iR^{r-i}z.\] Since $R^{r-i}$ kills $y$ and $z$, and $p$ kills $\Omega_{A/I^{s_0}}^n$, it follows that $dV^{i-1}R^{r-i-1}z=0$ in $\Omega_{A/I^{s_0}}^n$.

Passing to level $s_1$ and using surjectivity of the restriction map, it follows that there exists $u\in W_r\Omega_{A/I^{s_1}}^{n-1}$ such that $F^iR^{r-i}u=R^{r-i-1}z$; in other words, $z-F^iu\in\op{Fil}^1W_{r-i}\Omega_{A/I^{s_1}}^{n-1}$.

Again using the equality of the canonical and $V$-filtrations, we may write $z-F^iu=Vz'+dVv$ for some $z'\in W_{r-i-1}\Omega_{A/I^{s_1}}^{n-1}$ and $v\in W_{r-i-1}\Omega_{A/I^{s_1}}^{n-2}$; applying $V^i$ gives $V^iz-p^iu=V^{i+1}z'+p^idV^{i+1}v$. Returning to the element $x$, we have proved that
\begin{align}
x&=V^iy+d(V^{i+1}z'+p^idV^{i+1}v+p^iu)\notag\\
 &=V^iy'+dV^{i+1}z'\label{proposition_log_fil1},
\end{align}
where $y':=y+F^idu\in W_{r-i}\Omega_{A/I^{s_1}}^n$, and so the assumption that $Fx=Rx$ now reads \[V^iFy'+dV^iz'=V^iRy'+dV^{i+1}Rz'.\] Applying $R^{r-i-2}$ shows that $V^i(R-F)R^{r-i-2}y'\in dV^i\Omega_{A/I^{s_1}}^{n-1}$, whence passing to level $s_2$ implies that $(R-F)R^{r-i-2}y'\in B_{i+1}\Omega^n_{A/I^{s_2}}$.

Appealing to Remark \ref{remark_log_forms_cartier_filtration}, this means that $(1-C^{-1})R^{r-i-1}y'\in B_{i+1}\Omega^n_{A/I^{s_2}}/d\Omega_{A/I^{s_2}}^{n-1}$ and hence $R^{r-i-1}y'\in \Omega^n_{A/I^{s_2},\sub{log},\sub{Zar}}+B_i\Omega^n_{A/I^{s_4}}$. Finally passing to level $s$, equation (\ref{proposition_log_fil1}) now implies that $R^{r-i-1}x\in V^i\Omega^n_{A/I^s,\sub{log},\sub{Zar}}$.

This completes the proof of (\ref{equ_log_fil}). However, it is important to observe that we have actually proved a result which is preserved under \'etale base change: if $A'$ is any tft $I$-formally \'etale $A$-algebra, then the three properties which hold by choice of $s_2,s_1,s_0$ remain valid after replacing $A/I^{s_?}$ by $A'/I^{s_?}A'$ (and {\em without picking new values of $s_2,s_1,s_0$}), since the three properties behave well under \'etale base change (by the arguments of Remark \ref{remark_drw_etale_base_change}). Hence (\ref{equ_log_fil}) is also valid after replacing $A/I^t$ and $A/I^s$ by $A'/I^tA'$ and $A'/I^sA'$ respectively.

Fixing $r\ge 1$, we will now prove the following by descending induction on $i\le r$: for each $s\ge1$ there exists $t\ge s$ such that the following inclusion holds for any tft $I$-formally \'etale $A$-algebra $A'$:
\begin{equation}\tau_{t,s}(R(W_r\Omega_{A'/I^tA'}^{n,F=1}\cap\op{Fil}^iW_r\Omega_{A'/I^tA'}^n))\subseteq \Gamma(\Spec(A'/I^sA'), p^iW_{r-1}\Omega^n_{Y_s,\sub{log},\sub{Zar}}).\label{equation_log_fil_rtp_new}\end{equation} This is exactly assertion (\ref{equation_log_fil_rtp}) with $r-1$ in place of $r$, and so this proof by induction will complete the proof of the theorem.

The initial cases $i=r,r-1$ of (\ref{equation_log_fil_rtp_new}) are trivial since then the left side of (\ref{equation_log_fil_rtp_new}) is zero. So let $i<r-1$, and  by induction pick $s'\ge s$ such that (\ref{equation_log_fil_rtp_new}) holds for the index $i+1$, the integers $s'\ge s$ (instead of $t\ge s$), and any tft $I$-formally \'etale $A$-algebra $A'$; then finally pick $t\ge s'$ such that (\ref{equ_log_fil}) holds for the integers $t\ge s'$ (instead of $t\ge s$) and any tft $I$-formally \'etale $A$-algebra in place of $A$ (by the previous step of the proof). Now let $A'$ be any tft $I$-formally \'etale $A$-algebra; we must show that (\ref{equation_log_fil_rtp_new}) holds.

So let $x\in W_r\Omega_{A'/I^tA'}^{n,F=1}\cap\op{Fil}^iW_r\Omega_{A'/I^tA'}^n$. By choice of $t$, there exists an element $\omega\in\Omega_{A'/I^{s'}A',\sub{log},\sub{Zar}}^n$ such that $\tau_{t,s'}R^{r-i-1}x=V^i\omega$. Then there is a finite Zariski cover $\{\Spf A'_\lambda\}_{\lambda\in\Lambda}$ of $\Spf \hat{A'}$, and elements $f_\lambda\in (A_\lambda'/I^{s'}A_\lambda')^{\times\otimes n}$ such that $\omega|_{A'_\lambda/I^{s'}A'_\lambda}=\dlog f_\lambda$.

So, for any $\lambda\in\Lambda$, one has \[(\tau_{t,s'}R^{r-i-1}x)|_{A'_\lambda/I^{s'}A'_\lambda}=V^i\omega|_{A'_\lambda/I^{s'}A'_\lambda}=V^i\dlog f_\lambda= p^i\dlog[f_\lambda]_{i+1},\] where we have added a subscript to the Teichm\"uller lift to make the level clear. In other words, $(\tau_{t,s'}x)|_{A'_\lambda/I^{s'}A'_\lambda}-p^i\dlog[f_\lambda]_r\in\op{Fil}^{i+1}W_r\Omega_{A'_\lambda/I^{s'}A'_\lambda}^n$; but this element lies in $W_r\Omega_{A'_\lambda/I^{s'}A'_\lambda}^{n,F=1}$ (since each term does), whence the inductive hypothesis and choice of $s'$ imply that \[(\tau_{t,s}Rx)|_{A'_\lambda/I^sA'_\lambda}-p^i\dlog[\tau_{s',s}f_\lambda]_{r-1}\in \Gamma(\Spec A'_\lambda/I^sA'_\lambda, p^{i+1}W_{r-1}\Omega^n_{A'_\lambda/I^sA'_\lambda,\sub{log},\sub{Zar}}).\] Hence $\tau_{t',s}Rx\in \Gamma(\Spec(A'/I^sA'),p^iW_{r-1}\Omega^n_{Y_s,\sub{log},\sub{Zar}})$, as required to complete the inductive step proving (\ref{equation_log_fil_rtp_new}), and so finish the proof.
\end{proof}

\begin{corollary}\label{corollary_de_Rham--Witt_log_diagonal}
Let $\cal Y$ be a regular, F-finite formal $\bb F_p$-scheme, and $Y_1\into\cal Y$ a subscheme of definition; fix $n\ge 0$ and $i\ge1$. Letting $\tau$ denote the Zariski, Nisnevich, or \'etale topology, then the following sequence of pro $\tau$-sheaves on $Y_1$ is exact:
\[0\To\{W_s\Omega_{Y_s,\sub{log},\tau}^n\}_s\xTo{p^i}\{W_s\Omega_{Y_s,\sub{log},\tau}^n\}_s\xTo{\{R^{s-i}\}_s}\{W_i\Omega_{Y_s,\sub{log},\tau}^n\}_s\To 0.\]
That is, the canonical map $\{W_s\Omega_{Y_s,\sub{log},\tau}^n\dotimes_\bb Z\bb Z/p^i\}_s\to \{W_i\Omega_{Y_s,\sub{log},\tau}^n\}_s$ is an isomorphism.
\end{corollary}
\begin{proof}
Taking the diagonal over the indices $r$ and $s$ in Theorem \ref{theorem_de_Rham--Witt_log} proves exactness except at the left, which follows from the equality of the $p$- and canonical filtrations in Proposition \ref{proposition_filtrations_are_equal}.
\end{proof}

\begin{corollary}
Let $\cal Y$ be a regular, F-finite formal $\bb F_p$-scheme, and $Y_1\into\cal Y$ a subscheme of definition; fix $n\ge0$ and $0\le i\le r$. Then the inclusions \[\{p^iW_r\Omega_{Y_s,\sub{log},\sub{Zar}}^n\}_s\subseteq\{\ep_{\sub{Nis}*}(p^iW_r\Omega_{Y_s,\sub{log},\sub{Zar},\sub{Nis}}^n)\}_s\subseteq \{\ep_*(p^iW_r\Omega_{Y_s,\sub{log}}^n)\}_s\] (inside $\{W_r\Omega_{Y_s,\sub{log},\sub{Zar}}^n\}_s$) of pro Zariski sheaves on $Y_1$ are equalities.
\end{corollary}
\begin{proof}
If $i=0$ then the inclusions are equalities for any fixed $s\ge 1$ by Corollary \ref{corollary_Zar_vs_etale}(i). The case $i>0$ then follows by a straightforward induction from Corollary \ref{corollary_de_Rham--Witt_log_diagonal}.
\end{proof}

\subsection{Consequences of Theorem \ref{theorem_de_Rham--Witt_log}}
Now we present some applications of Theorem \ref{theorem_de_Rham--Witt_log} to logarithmic Hodge--Witt sheaves. First we prove that they satisfy a pro excision property (more general pro cdh descent properties also exist, but we do not discuss them in this paper); this will be used in Section \ref{section_K_theory}:

\begin{corollary}\label{corollary_pro_excision}
Let $X$ be a regular, F-finite $\bb F_p$-scheme, and $Z,Z'\into X$ closed subschemes; let $\tau$ be the Zariski, Nisnevich, or \'etale topology, and fix $r\ge1$, $n\ge0$. Then the square of pro $\tau$-sheaves on $X$
\[\xymatrix{
\{W_r\Omega_{(Z\cup Z')_s,\sub{log},\tau}^n\}_s\ar[r]\ar[d]&\{W_r\Omega_{Z_s,\sub{log},\tau}^n\}_s\ar[d]\\
\{W_r\Omega_{Z'_s,\sub{log},\tau}^n\}_s\ar[r]&\{W_r\Omega_{(Z\cap Z')_s,\sub{log},\tau}^n\}_s
}\]
is bicartesian.
\end{corollary}
\begin{proof}
Since all arrows in the diagram are surjective in the $\tau$-topology, it is sufficient to prove that the square is cartesian.

If $B$ is a Noetherian ring, $I_1,I_2\subseteq B$ are ideals, and $M$ is a finitely generated $B$-module, then elementary Artin--Rees arguments (using Proposition \ref{proposition_AR_p}(i)) show that the square of pro $B$-modules
\[\xymatrix{
\{B/(I_1\cap I_2)^s\}_s\ar[r]\ar[d] & \{B/I_2^s\}_s\ar[d]\\
\{B/I_1^s\}_s\ar[r] & \{B/(I_1+ I_2)^s\}_s
}\]
is bicartesian. Hence, if $A$ is a Noetherian, F-finite $\bb F_p$-algebra, and $I,J\subseteq A$ are ideals, then we may set $B=W_r(A)$, $I_1=W_r(I)$, $I_2=W_r(J)$, and $M=W_r\Omega_A^n$, to obtain a bicartesian square which, by Lemmas \ref{lemma_witt_intertwined}, \ref{lemma_finite_generation_conditions}, and \ref{lemma_Witt_dg_ideal_gen_by_I}(ii), will be exactly
\[\xymatrix{
\{W_r\Omega_{A/(I\cap J)^s}^n\}_s\ar[r]\ar[d]&\{W_r\Omega_{A/J^s}^n\}_s\ar[d]\\
\{W_r\Omega_{A/I^s}^n\}_s\ar[r]&\{W_r\Omega_{A/(I+ J)^s}^n\}_s
}\]
This argument works verbatim for coherent sheaves, yielding a bicartesian square of pro \'etale sheaves on $X$:
\[\xymatrix{
\{W_r\Omega_{(Z\cup Z')_s}^n\}_s\ar[r]\ar[d]&\{W_r\Omega_{Z_s}^n\}_s\ar[d]\\
\{W_r\Omega_{Z'_s}^n\}_s\ar[r]&\{W_r\Omega_{(Z\cap Z')_s}^n\}_s
}\]
Taking the diagonal over $r=s$ and then the kernel of $R-F$ using Corollary \ref{corollary_log_vs_R-F}, it follows that 
\[\xymatrix{
\{W_s\Omega_{(Z\cup Z')_s,\sub{log}}^n\}_s\ar[r]\ar[d]&\{W_s\Omega_{Z_s,\sub{log}}^n\}_s\ar[d]\\
\{W_s\Omega_{Z'_s,\sub{log}}^n\}_s\ar[r]&\{W_s\Omega_{(Z\cap Z')_s,\sub{log}}^n\}_s
}\]
is also bicartesian. Finally, applying $-\dotimes_\bb Z\bb Z/p^r\bb Z$ yields yet another homotopy cartesian square, which according to Corollary \ref{corollary_de_Rham--Witt_log_diagonal} is exactly the desired square in the \'etale topology:
\[\xymatrix{
\{W_r\Omega_{(Z\cup Z')_s,\sub{log}}^n\}_s\ar[r]\ar[d]&\{W_r\Omega_{Z_s,\sub{log}}^n\}_s\ar[d]\\
\{W_r\Omega_{Z'_s,\sub{log}}^n\}_s\ar[r]&\{W_r\Omega_{(Z\cap Z')_s,\sub{log}}^n\}_s
}\]
Pushing to the Zariski or Nisnevich topology yields the desired square in the $\tau$-topology, by Corollary \ref{corollary_Zar_vs_etale}(i), and hence proves it is cartesian; but we have already observed that cartnesianess is sufficient.
\end{proof}

Secondly we show that logarithmic Hodge--Witt groups are continuous; this will be used to prove continuity of $K$-theory in Section \ref{subsection_continuity_K}:

\begin{corollary}\label{lemma_continuity_of_log_HW}
Let $A$ be a regular, local, F-finite $\bb F_p$-algebra, and $I\subseteq A$ an ideal such that $A$ is $I$-adically complete. Then the canonical map
\[W_r\Omega_{A,\sub{log},\sub{Zar}}^n\To \projlim_sW_r\Omega_{A/I^s,\sub{log},\sub{Zar}}^n\] is an isomorphism for all $n\ge0$, $r\ge1$.
\end{corollary}
\begin{proof}
Injectivity is clear from the isomorphism $W_r\Omega_{A}^n\isoto\projlim_sW_r\Omega_{A/I^s}^n$ of Lemma \ref{lemma_completion_of_dRW_complex}. To prove surjectivity, we consider the commutative diagram
\[\xymatrix{
W_{r+1}\Omega_{A}^{n,F=1}\ar[r]^\cong\ar[d]^R & \projlim_sW_{r+1}\Omega_{A/I^s}^{n,F=1}\ar[d]^R\\
W_{r}\Omega_{A,\sub{log},\sub{Zar}}^{n}\ar[r] & \projlim_sW_{r}\Omega_{A/I^s,\sub{log},\sub{Zar}}^{n}\\
}\]
in which the horizontal arrow is Proposition \ref{proposition_limit_of_F-fixed_points}(i) and the vertical arrows are well-defined by Corollary \ref{corollary_log_vs_R-F}. If we can show that the right vertical arrow is surjective, the proof will be complete; we will prove the stronger result that $R:\projlim_sW_{r+1}\Omega_{A/I^s,\sub{log}}^n\to \projlim_sW_r\Omega_{A/I^s,\sub{log}}^n$ is surjective.

Indeed, taking $\projlim_s$ in Corollary \ref{corollary_de_Rham--Witt_log_diagonal} obtains an exact sequence \[0\To\projlim_sW_s\Omega_{A/I^s,\sub{log},\sub{Zar}}^n\xTo{p^r}\projlim_sW_s\Omega_{A/I^s,\sub{log},\sub{Zar}}^n\xTo{(\dag)} \projlim_sW_r\Omega_{A/I^s,\sub{log},\sub{Zar}}^n\To 0,\] where these is no $\projlim^1$ obstruction to right exactness since the transition maps in the left system are surjective. Since arrow (\dag) factors through $\projlim_sW_{r+1}\Omega_{A/I^s,\sub{log}}^n$, the proof is complete.
\end{proof}

Finally we prove a modification of Corollary \ref{corollary_Zar_to_etale} which we will need in Section \ref{subsection_class_group}:

\begin{corollary}\label{corollary_relative_nilpotent_hw}
Let $X$ be a regular, F-finite $\bb F_p$-scheme, $Y\into$ X a closed subscheme, fix $r\ge1$, $n\ge 0$. If $Y$ is regular, then the canonical maps of pro complexes of Zariski sheaves on $Y$ \[\{W_r\Omega_{(Y_s,Y),\sub{log},\sub{Zar}}^n\}_s\To \{R\ep_{\sub{Nis}*}W_r\Omega_{(Y_s,Y),\sub{log},\sub{Nis}}^n\}_s\To \{R\ep_*W_r\Omega_{(Y_s,Y),\sub{log}}^n\}_s\] are quasi-isomorphisms. If $Y$ is gnc, then the canonical maps of pro complexes of Zariski sheaves on $Y$ (resp.~X) \[\{W_r\Omega_{Y_s,\sub{log},\sub{Zar}}^n\}_s\To \{R\ep_{\sub{Nis}*}W_r\Omega_{Y_s,\sub{log},\sub{Nis}}^n\}_s\] \[\{W_r\Omega_{(X,Y_s),\sub{log},\sub{Zar}}^n\}_s\To \{R\ep_{\sub{Nis}*}W_r\Omega_{(X,Y_s),\sub{log},\sub{Nis}}^n\}_s\] are quasi-isomorphisms.
\end{corollary}
\begin{proof}
Suppose first that $Y$ is regular. Corollary \ref{corollary_de_Rham--Witt_log_diagonal} and its analogue for $Y$ due to Illusie yield a short exact sequence of pro $\tau$-sheaves
\[0\To\{W_s\Omega_{(Y_s,Y),\sub{log},\tau}^n\}_s\xTo{p^r}\{W_s\Omega_{(Y_s,Y),\sub{log},\tau}^n\}_s\xTo{\{R^{s-r}\}_s}\{W_r\Omega_{(Y_s,Y),\sub{log},\tau}^n\}_s\To 0,\] where $\tau$ is the Zariski, Nisnevich, or \'etale topology. Comparing the pushforwards of the Nisnevich and \'etale versions to the Zariski version, we see it is sufficient to show that that the canonical maps
\[\{W_s\Omega_{(Y_s,Y),\sub{log},\sub{Zar}}^n\}_r\To \{R\ep_{\sub{Nis}*}W_s\Omega_{(Y_s,Y),\sub{log},\sub{Nis}}^n\}_s\To \{R\ep_*W_s\Omega_{(Y_s,Y),\sub{log}}^n\}_s\] are quasi-isomorphisms; but this is a consequence of Corollary \ref{corollary_Zar_to_etale}.

Still assuming that $Y$ is regular, the Gersten resolution for logarithmic Hodge--Witt theory (Gros and Suwa for smooth varieties \cite{GrosSuwa1988}; Shiho in general \cite{Shiho2007}) implies that $W_r\Omega_{Y,\sub{Zar}}^n\quis R\ep_*W_r\Omega_{Y,\sub{Nis}}^n$. Combining this with the previous paragraph, we deduce that $\{W_r\Omega_{Y_s,\sub{log},\sub{Zar}}^n\}_s\quis \{R\ep_{\sub{Nis}*}W_r\Omega_{Y_s,\sub{log},\sub{Nis}}^n\}_s$.

Now suppose that $Y$ is merely gnc, and proceed by induction on its complexity; let $Z,Z'$ be a closed cover of $Y$ such that $Z$, $Z'$, and $Z'':=Z\cap Z'$ are gnc schemes of complexity less than that of $Y$ (these exists by Lemma \ref{lemma_complexity}). Comparing the Nisnevich pushforward of the bicartesian square from Corollary \ref{corollary_pro_excision} with the analogous Zariski square, the inductive hypothesis shows $\{W_r\Omega_{Y_s,\sub{log},\sub{Zar}}^n\}_s\quis \{R\ep_{\sub{Nis}*}W_r\Omega_{Y_s,\sub{log},\sub{Nis}}^n\}_s$.

Finally, still assuming that $Y$ is gnc, the Gersten resolution on $X$ implies that $W_r\Omega_{Y,\sub{Zar}}^n\quis R\ep_*W_r\Omega_{Y,\sub{Nis}}^n$, whence $\{W_r\Omega_{(X,Y_s),\sub{log},\sub{Zar}}^n\}_s\quis \{R\ep_{\sub{Nis}*}W_r\Omega_{(X,Y_s),\sub{log},\sub{Nis}}^n\}_s$.
\end{proof}

\section{Formal Geisser--Levine and Bloch--Kato--Gabber theorems}\label{section_K_theory}
In this section we present our main new construction in higher algebraic $K$-theory, namely a formal $\dlog$ map for regular formal $\bb F_p$-schemes.

\subsection{Preliminaries on $K$-theory}\label{subsection_prelim_K}
The algebraic $K$-groups $K_n$ of a ring or scheme are understood in the sense of Thomason--Trobaugh \cite{Thomason1990}.

To avoid unnecessarily restricting to rings with infinite residue fields, the notation $K_n^M(A)$ will be used to denote the {\em improved} Milnor $K$-theory of a ring $A$, as defined by Gabber and Kerz \cite{Kerz2009}. Assuming that $A$ is local, recall that this is a certain quotient of the usual Milnor $K$-group, and that the two coincide if $A$ is field or if the residue field of $A$ has $>M_n$ elements, where $M_n$ is some constant depending only on $n$ (e.g., $M_2=5$). If $I\subseteq A$ is an ideal, then $K_n^M(A,I):=\ker(K_n^M(A)\to K_n^M(A/I))$.

The following version of the results of Geisser--Levine, Bloch--Kato--Gabber, and Izhboldin, is in principle well-known to experts, though as with Theorem \ref{theorem_Zar=et} seems not to be in the literature\footnote{The history of the two theorems seems rather convoluted. Geisser--Levine proved the statement of Theorem \ref{theorem_GL_BKG}, but their definition of $K_n^M(A)$ was as $\ker(K_n^M(F)\to\oplus_{y\in Y^1}K_{n-1}^M(k(y))$, where we use the notation of the proof. It does not seem that this was known to be generated by symbols before the work of Elbaz-Vincent and M\"uller-Stach (and Kerz in the finite residue field case), and so it seems that Theorem \ref{theorem_Zar=et} cannot have been known before 2002, even in the case of a smooth variety over a perfect field.}:

\begin{theorem}[Geisser--Levine, Bloch--Kato--Gabber, Izhboldin, Kato, et al.]\label{theorem_GL_BKG}
Let $A$ be a regular, local $\bb F_p$-algebra, and $n\ge 0$. Then $K_n(A)$ and $K_n^M(A)$ are $p$-torsion-free, and the maps \[K_n(A)/p^r\longleftarrow K_n^M(A)/p^r\xTo{\dlog[\cdot]}W_r\Omega_{A,\sub{log}}^n\] are isomorphisms.
\end{theorem}
\begin{proof}
By N\'eron--Popescu desingularisation we may assume that $A$ is essentially of finite type over a perfect field $k$ of characteristic $p$ (although that is not necessary for many of the steps). Let $F$ denote the field of fractions of $A$, and put $Y=\Spec A$. Then there are Gersten sequences in logarithmic Hodge--Witt theory (Gros and Suwa \cite{GrosSuwa1988}) \begin{equation}0\To W_r\Omega^n_{A,\sub{log}}\To W_r\Omega^n_{F,\sub{log}}\To\bigoplus_{y\in Y^1}W_r\Omega^{n-1}_{k(y),\sub{log}}\To\cdots,\label{equ_de_R_W_Gersten}\end{equation} and in algebraic $K$-theory \[0\To K_n(A)\To K_n(F)\To\bigoplus_{y\in Y^1}K_{n-1}(k(y))\To\cdots\]
(Quillen \cite{Quillen1973a}). The latter also holds with $\bb Z/p^r$-coefficients (by examination of Quillen's proof) and hence, using the absence of $p$-torsion in the algebraic $K$-theory of characteristic $p$ fields (Geisser--Levine \cite{GeisserLevine2000}), we deduce that $K_n(A)$ is $p$-torsion-free and that \begin{equation}0\To K_n(A)/p^r\To K_n(F)/p^r\To\bigoplus_{y\in Y^1}K_{n-1}(k(y))/p^r\To \cdots\label{equ_K_Gersten}\end{equation} is exact.

By comparing (\ref{equ_de_R_W_Gersten}) and (\ref{equ_K_Gersten}) to the Gersten complex for Milnor $K$-theory
\begin{equation}0\To K_n^M(A)\To K_n^M(F)\To\bigoplus_{y\in Y^1}K_{n-1}^M(k(y))\To\cdots,\label{equ_K^M_Gersten}\end{equation}
(constructed by Kato \cite{Kato1986}), and using the validity of the isomorphisms of the theorem in the case of a characteristic $p$ field (Geisser--Levine again, and Bloch--Kato--Gabber \cite[Corol.~2.8]{Bloch1986}), it remains only to show that (\ref{equ_K^M_Gersten}) is universally exact (note also that $K_n^M(F)$ is $p$-torsion-free by Izhboldin \cite{Izhboldin1991}, whence we will deduce the same for $K_n^M(A)$).

Ignoring ``$0\to K_n^M(A)\to$'', the universal exactness of (\ref{equ_K^M_Gersten}) was proved by Rost \cite{Rost1996} (c.f.,~\cite[7.3(5)]{ColliotThelene-Hoobler-Kahn1997}). The universal exactness of \[K_n^M(A)\To K_n^M(F)\To \oplus_{y\in Y^1}K_{n-1}^M(k(y))\] is then due to Elbaz-Vincent and M\"uller-Stach \cite[Prop.~4.3]{Elbaz-Vincent-Muller-Stach2002} if $k$ is infinite, and in general it is due to Kerz \cite[Prop.~10(8) \& Thm.~13]{Kerz2010}. Finally, the universal injectivity of $K_n^M(A)\to K_n^M(F)$ is due to Kerz: this follows from his analogous result for usual Milnor $K$-theory \cite[Thm.~6.1]{Kerz2009}, together with his norm trick in the proof of \cite[Prop.~10(8)]{Kerz2010}.
\end{proof}

\begin{remark}
We used the surjectivity of $\dlog[\cdot]:K_n^M(A)\to W_r\Omega_{A,\sub{log}}^n$ from the previous theorem to prove Theorem \ref{theorem_Zar=et}, namely that $W_r\Omega_{A,\sub{log},\sub{Zar}}^n=W_r\Omega_{A,\sub{log}}^n$. However, it is more natural to view $\dlog[\cdot]$ as a map in the Zariski topology and so we henceforth write $W_r\Omega_{A,\sub{log},\sub{Zar}}^n$ in place of $W_r\Omega_{A,\sub{log}}^n$.
\end{remark}

For any regular $\bb F_p$-algebra $A$ (not necessary local), we define the homomorphism \[\dlog^n_{A,\bb Z/p^r}:K_n(A;\bb Z/p^r)\To W_r\Omega_{A,\sub{log},\sub{Zar}}^n,\] for each, $n\ge0$, $r\ge1$, as the composition
\begin{align*}\hspace{-5mm}
K_n(A;\bb Z/p^r)\to H_\sub{Zar}^0(Y,\cal K_{n,Y,\bb Z/p^r})&=H_\sub{Zar}^0(Y,\cal K_{n,Y}/p^r)\\
&\cong H_\sub{Zar}^0(Y,\cal K_{n,Y}^M/p^r)\stackrel{\dlog[\cdot]}{\isoto}H_\sub{Zar}^0(Y,W_r\Omega_{Y,\sub{log},\sub{Zar}}^n)= W_r\Omega_{A,\sub{log},\sub{Zar}}^n
\end{align*}
where $Y=\Spec A$. Here $\cal K_{n,Y}$ is a Zariski sheaf of Quillen $K$-groups on $X$ (and similarly for Milnor $K$-theory, or with $\bb Z/p^r$-coefficients), and the central equality and isomorphisms were recalled in the previous theorem. These maps are easily seen to be uniquely determined by the following properties:
\begin{enumerate}
\item (Naturality) $\dlog_{-,\bb Z/p^r}^n$ is natural for morphisms of regular $\bb F_p$-algebras.

\item (Symbols) If $\al_1,\dots,\al_n\in A^\times$, then $\dlog_{r,A}^n(\{\al_1,\dots,\al_n\})=\dlog[\al_1]\cdots\dlog[\al_n]$, where we use the composition \[\dlog_{r,A}^n:K_n(A)\To K_n(A;\bb Z/p^r)\xto{\dlog^n_{A,\bb Z/p^r}}W_r\Omega_{A,\sub{log},\sub{Zar}}^n\]

\item (Compatibility as $r\to\infty$) The following diagram commutes:
\[\xymatrix@C=2cm{
K_n(A;\bb Z/p^{r+1})\ar[d]\ar[r]^-{\dlog_{A,\bb Z/p^{r+1}}^n} & W_{r+1}\Omega_{A,\sub{log},\sub{Zar}}^n\ar[d]^R\\
K_n(A;\bb Z/p^r) \ar[r]_{\dlog_{A,\bb Z/p^r}^n}& W_r\Omega_{A,\sub{log},\sub{Zar}}^n
}\]
\item (Multiplicativity) $\dlog_{A\bb Z/p^r}^*:K_*(A;\bb Z/p^r)\to W_r\Omega_{A,\sub{log},\sub{Zar}}^*$ is a homomorphism of graded rings.
\end{enumerate}

\subsection{The formal dlog map: statements of main properties}\label{subsection_main_results}
The goal of this section is to state the existence of our pro/formal analogue of $\dlog^n_{A,\bb Z/p^r}$, whose basic properties are summarised in the following theorem, and then to present various analogues of Theorem \ref{theorem_GL_BKG}; proofs are contained in the next section.

\begin{theorem}\label{theorem_existence}
For any regular, F-finite $\bb F_p$-algebra $A$ and ideal $I\subseteq A$, there are homomorphisms of pro abelian groups \[\dlog_{A/I^\infty,\bb Z/p^r}^n:\{K_n(A/I^s;\bb Z/p^r)\}_s\To \{W_r\Omega_{A/I^s,\sub{log},\sub{Zar}}^n\}_s\] for all $n\ge0$, $r\ge1$ such that the following properties are satisfied:
\begin{enumerate}
\item (Naturality) $\dlog_{-,\bb Z/p^r}^n$ is natural for morphisms of pairs $A,I$ (in the obvious sense).
\item (Symbols) The composition \[\{K_n^M(A/I^s)\}_s\To \{K_n(A/I^s)\}_s\xto{\dlog_{r,A/I^\sinfty}^n} \{W_r\Omega_{A/I^s,\sub{log},\sub{Zar}}^n\}_s\] is induced by the homomorphisms $\dlog[\cdot]:K_n^M(A/I^s)\To W_r\Omega_{A/I^s,\sub{log},\sub{Zar}}^n$, where we use the composition \[\dlog_{r,A/I^\sinfty}^n:\{K_n(A/I^s)\}_s\To\{K_n(A/I^s;\bb Z/p^r)\}_s\xTo{\dlog^n_{A/I^\infty,\bb Z/p^r}}\{W_r\Omega_{A/I^s,\sub{log},\sub{Zar}}^n\}_s.\]
\item (Compatibility as $r\to\infty$) The following diagram commutes:
\[\xymatrix@C=2cm{
\{K_n(A/I^s;\bb Z/p^{r+1})\}_s\ar[d]\ar[r]^-{\dlog_{A/I^\sinfty,\bb Z/p^{r+1}}^n} & \{W_{r+1}\Omega_{A/I^s,\sub{log},\sub{Zar}}^n\}_s\ar[d]^R\\
\{K_n(A/I^s;\bb Z/p^r)\}_s \ar[r]_{\dlog_{A/I^\sinfty,\bb Z/p^r}^n}& \{W_r\Omega_{A/I^s,\sub{log},\sub{Zar}}^n\}_s
}\]
\item (Multiplicativity) $\dlog_{A/I^\infty,\bb Z/p^r}^*:\{K_*(A/I^s;\bb Z/p^r)\}_s\to \{W_r\Omega_{A/I^s,\sub{log},\sub{Zar}}^*\}_s$ is a homomorphism of graded pro rings.
\item (Discrete case, i.e., $I=0$) $\dlog^n_{A/0^\infty,\bb Z/p^r}=\dlog^n_{A,\bb Z/p^r}$.
\item (Compatibility with completion) The following diagram commutes:
\[\xymatrix@C=2cm{
K_n(\hat A;\bb Z/p^r)\ar[r]^{\dlog^n_{\hat A,\bb Z/p^r}}\ar[d] & W_r\Omega_{\hat A,\sub{log},\sub{Zar}}^n\ar[d]\\
\{K_n(A/I^s;\bb Z/p^r)\}_s\ar[r]_{\dlog^n_{A/I^\infty,\bb Z/p^r}} & \{W_r\Omega_{A/I^s,\sub{log},\sub{Zar}}^n\}_s
}\]
where $\hat A:=\projlim_sA/I^s$ is the $I$-adic completion of $A$.
\end{enumerate}
\end{theorem}

The central goal of the paper is the following analogue of the theorems of Geisser--Levine, Bloch--Kato--Gabber, et al.~recalled in Theorem \ref{theorem_GL_BKG}:

\begin{theorem}\label{theorem_pro_GL}
Let $A$ be a regular, F-finite $\bb F_p$-algebra, and $I\subseteq A$ an ideal such that $A/I$ is gnc and local. Then the pro abelian group $\{K_n(A/I^s)\}_s$ is $p$-torsion-free, and \[\dlog^n_{A/I^\infty,\bb Z/p^r}:\{K_n(A/I^s;\bb Z/p^r)\}_s\To\{W_r\Omega_{A/I^s,\sub{log},\sub{Zar}}^n\}_s\] is an isomorphism for all $n\ge 0$, $r\ge1$.
\end{theorem}

\begin{corollary}\label{corollary_thm_from_intro}
Let $A,I$ be as in Theorem \ref{theorem_pro_GL}. Then the natural homomorphisms of pro abelian groups \[\{K_n(A/I^s)/p^r\}_s\longleftarrow \{K_n^M(A/I^s)/p^r\}_s\xTo{\dlog[\cdot]} \{W_r\Omega_{A/I^s,\sub{log}}^n\}_s\] are surjective and have the same kernel.
\end{corollary}
\begin{proof}
The following diagram commutes \[\xymatrix{\{K_n(A/I^s)/p^r\}_s\ar[d] & \{K_n^M(A/I^s)/p^r\}_s\ar[l]\ar@{->>}[r]^-{\dlog[\cdot]}&\{W_r\Omega_{A/I^s,\sub{log},\sub{Zar}}^n\}_s=\{W_r\Omega_{A/I^s,\sub{log}}^n\}_s\\ \{K_n(A/I^s;\bb Z/p^r)\}_s\ar@/_5mm/[rru]_{\dlog^n_{A/I^\infty,\bb Z/p^r}}&&&}\]
by Theorem \ref{theorem_existence}(ii); the equality on the right is Corollary \ref{corollary_Zar_vs_etale}(i)). The previous theorem implies that the vertical and diagonal arrows are isomorphisms, which obviously completes the proof.
\end{proof}

Using a pro excision argument, Theorem \ref{theorem_pro_GL} will be reduced to the case that $A/I$ is regular. Then $A/I$ is formally smooth over $\bb F_p$ by Lemma~\ref{lemma_F-finite_formal_smooth} (this does not require $A/I$ to be local), and so the quotient maps $A/I^s\to A/I$ have compatible splittings for all $s\ge 1$; hence we may take kernels in the commutative diagram
\[\xymatrix{
\{K_n(A/I^s;\bb Z/p^r)\}_s\ar[d]_{\dlog^n_{A/I^\infty,\bb Z/p^r}}\ar[r] & K_n(A/I,\bb Z/p^r)\ar[d]^{\dlog^n_{A,\bb Z/p^r}}\\ 
\{W_r\Omega_{A/I^s,\sub{log},\sub{Zar}}^n\}_s \ar[r]& W_r\Omega_{A/I,\sub{log},\sub{Zar}}^n
}\]
to induce a relative morphism \[\dlog^n_{(A/I^\infty,I/I^\infty),\bb Z/p^r}:\{K_n(A/I^s,I/I^s;\bb Z/p^r)\}_s\to \{W_r\Omega_{(A/I^s,I/I^s),\sub{log},\sub{Zar}}^n\}_s.\] To prove Theorem \ref{theorem_pro_GL}, we will first show:

\begin{theorem}\label{theorem_mccarthy}
Let $A$ be a regular, F-finite $\bb F_p$-algebra, and $I\subseteq A$ an ideal such that $A/I$ is also regular. Then $\dlog^n_{(A/I^\infty,I/I^\infty),\bb Z/p^r}$ is an isomorphism for all $n\ge 0$, $r\ge1$.
\end{theorem}

Taking the diagonal over $r=s$, we then obtain the composition \[\dlog^n_{(A/I^\infty,I/I^\infty)}:\{K_n(A/I^s,I/I^s)\}_s\To \{K_n(A/I^s,I/I^s;\bb Z/p^s)\}_s\Isoto\{W_s\Omega_{(A/I^s,I/I^s),\sub{log},\sub{Zar}}^n\}_s,\] where the isomorphism is a consequence of Theorem \ref{theorem_mccarthy}.

\begin{corollary}\label{corollary_mccarthy}
Let $A$ be a regular, F-finite $\bb F_p$-algebra, and $I\subseteq A$ an ideal such that $A/I$ is also regular. Then $\dlog^n_{(A/I^\infty,I/I^\infty)}$ is an isomorphism for all $n\ge 0$.
\end{corollary}
\begin{proof}
Because of Theorem \ref{theorem_mccarthy}, it remains only to show that the canonical map \[\{K_n(A/I^s,I/I^s)\}_s\To \{K_n(A/I^s,I/I^s;\bb Z/p^s)\}_s\] is an isomorphism. But this is a consequence of the fact that relative $K$-groups of infinitesimal thickenings in characteristic $p$ are known to be $p$-torsion of bounded exponent \cite[Thm.~A]{GeisserHesselholt2011}.
\end{proof}

\begin{remark}\label{remark_dlog_formal_schemes}
We stated, and will prove, Theorem \ref{theorem_existence}--\ref{theorem_mccarthy} in the affine case only to simplify notation: there is no difficulty sheafifying the construction in the Zariski, Nisnevich, or \'etale topology. This sheafification does not strictly follow from the given statements (since a map of pro sheaves which is an isomorphism on all opens is not necessarily an isomorphism, because the necessary bounds in the pro systems may be uncontrollable), but rather from the proofs, which can be repeated verbatim in terms of sheaves on a formal scheme (particularly since we were careful to state Theorem \ref{theorem_de_Rham--Witt_log} in terms of pro \'etale sheaves, by controlling the bounds in the proof\footnote{Being careful, we should note that the pro-HKR theorem \cite{Morrow_Dundas} also holds as sheaves, by controlling bounds on the pro systems involved. This follows easily from the compatibility of both sides of the upcoming isomorphism (pro-HKR) with \'etale base change.}). Here we briefly state, for the sake of completeness and our applications, the main results in the generality of formal schemes. Let $\tau$ denote the Zariski, Nisnevich, or \'etale topology, and let $\cal K_{n,Y,\tau}$ denote the $\tau$-sheafification of the $n^\sub{th}$ $K$-group presheaf on $Y$, and similarly for $\bb Z/p^r$-coefficients (we omit $\tau$ when it denotes the Zariski topology).

Given a regular, F-finite formal $\bb F_p$-scheme $\cal Y$, and a subscheme of definition $Y:=Y_1\into\cal Y$, Theorem \ref{theorem_existence} generalises to the existence of a homomorphism of pro $\tau$-sheaves on $Y_1$ \[\dlog^n_{\cal Y,\bb Z/p^r}:\{\cal K_{n,Y_s,\bb Z/p^r,\tau}\}_s\To\{W_r\Omega^n_{Y_s,\sub{log},\tau}\}_s\] which satisfies the following properties:
\begin{enumerate}
\item (Naturality) $\dlog_{-,\bb Z/p^r}^n$ is natural for morphisms of regular, F-finite $\bb F_p$-schemes.
\item (Symbols) The composition \[\{\cal K_{n,Y_s,\tau}^M\}_s\To \{\cal K_{n,Y_s,\bb Z/p^r,\tau}\}_s\xto{\dlog_{\cal Y,\bb Z/p^r}^n} \{W_r\Omega_{Y_s,\sub{log},\tau}^n\}_s\] is induced $\dlog[\cdot]$.
\item (Compatibility as $r\to\infty$) The obvious generalisation of Theorem \ref{theorem_existence}(iii).
\item (Multiplicativity) $\dlog_{\cal Y,\bb Z/p^r}^*:\{\cal K_{*,\cal Y,\bb Z/p^r,\tau}\}_s\to \{W_r\Omega_{\cal Y,\sub{log},\tau}^*\}_s$ is a homomorphism of pro sheaves of graded rings.
\item (Discrete case) If $\cal Y=Y$, then $\dlog_{\cal Y,\bb Z/p^r}^n$ equals the already defined $\dlog_{Y,\bb Z/p^r}^n$. 
\item (Compatibility with affine case) If $\cal Y=\Spf A$, where $A$ is a regular, F-finite $\bb F_p$-algebra, and $I\subseteq A$ is the defining ideal of $Y$, then (the homomorphism induced on global section by) $\dlog^n_{\cal Y,\bb Z/p^r}$ is $\dlog^n_{A/I^\sinfty,\bb Z/p^r}$.
\end{enumerate}

If the subscheme of definition $Y$ is gnc, then Theorem \ref{theorem_pro_GL} generalises to an isomorphism $\dlog_{\cal Y,\bb Z/p^r}^n:\{\cal K_{n,Y_s,\bb Z/p^r,\tau}\}_s\isoto\{W_r\Omega^n_{Y_s,\sub{log},\tau}\}_s$ of pro $\tau$-sheaves on $Y$, and states that $\{\cal K_{n,Y_s,\tau}\}_s$ is $p$-torsion-free.

If $Y$ is actually regular, then Theorem \ref{theorem_mccarthy} and Corollary \ref{corollary_mccarthy} generalise to \[\dlog^n_{(\cal Y,Y),\bb Z/p^r}:\{\cal K_{n,(Y_s,Y),\bb Z/p^r,\tau}\}_s\isoto \{W_r\Omega^n_{(Y_s,Y),\sub{log},\tau}\}_s\] and \[\dlog^n_{(\cal Y,Y)}:\{\cal K_{n,(Y_s,Y),\tau}\}_s\isoto \{W_r\Omega^n_{(Y_s,Y),\sub{log},\tau}\}_s.\]
\end{remark}

\subsection{Proofs of Theorem \ref{theorem_existence}--\ref{theorem_pro_GL}}
The proofs use not only the earlier results concerning logarithmic Hodge--Witt sheaves on formal schemes, but also topological cyclic homology via McCarthy's theorem and the pro Hochschild--Kostant--Rosenberg theorem. Hence some familiarity with topological cyclic homology and its notation is required, for which we refer the reader to, e.g., \cite{Geisser2005} or \cite{Morrow_Dundas}.

\begin{proof}[Proof of Theorem \ref{theorem_existence}: construction of $\dlog_{r,A/I^\infty}^n$]
The construction begins with the pro version of the Hochschild--Kostant--Rosenberg (HKR) theorem in topological cyclic homology for the fixed point spectra $\TR^r$. If $A$ is an $\bb F_p$-algebra, then the pro graded ring $\{\TR^r_*(A;p)\}_r$ is a $p$-typical Witt complex with respect to its operators $F,V,R$; by universality of the de Rham--Witt complex, there are therefore natural maps of graded $W_r(A)$-algebras \cite[Prop.~1.5.8]{Hesselholt1996} $\lambda_{r,A}:W_r\Omega_A^*\to \TR_*^r(A;p)$ for $r\ge0$, which are compatible with the Frobenius,  Verschiebung, and Restriction maps (in other words, a morphism of $p$-typical Witt complexes).

From now on in the proof assume that $A$ is regular and F-finite, and let $I\subseteq A$ be any ideal. Hesselholt's HKR theorem \cite[Thm.~B]{Hesselholt1996} implies that the resulting map of pro abelian groups \[\lambda_A:\{W_r\Omega_A^n\}_r\To \{\TR_n^r(A;p)\}_r\] is an isomorphism for each $n\ge 1$; similarly, the pro HKR theorem of the author and Dundas \cite{Morrow_Dundas} implies that \[\lambda_{A/I^\infty}: \{W_s\Omega_{A/I^s}^n\}_s\To \{\TR_n^s(A/I^s;p)\}_s \tag{pro-HKR}\] is an isomorphism. (We remark that both the HKR and pro HKR theorem give more precise statements about $\TR_n^r$ for fixed $r$, but we do not need them here.)

Since the pro abelian group $\{W_s\Omega_{A/I^s}^n\}_s$ has no $p$-torsion (Proposition \ref{proposition_filtrations_are_equal}), the pro HKR isomorphism also induces an isomorphism \[\lambda_{A/I^\infty,\bb Z/p^r}:\{W_s\Omega_{A/I^s}^n/p^r\}_s\Isoto \{\TR_n^s(A/I^s;p,\bb Z/p^r)\}_s\] which is compatible with the Frobenius and Verschiebung on each side.

We now consider the following diagram, in which the squares commute and the top row is exact:
\[\hspace{-5mm}\xymatrix@C=4mm{
& \{\TC^s_n(A/I^s;p,\bb Z/p^r)\}_s\ar[r] & \{\TR^s_n(A/I^s;p,\bb Z/p^s)\}_s\ar[r]^{R-F}\ar[d]^\cong_{\lambda_{A/I^\infty,\bb Z/p^r}^{-1}}& \{\TR_n^{s-1}(A/I^s;p,\bb Z/p^r)\}_s\ar[d]^\cong_{\lambda_{A/I^\infty,\bb Z/p^r}^{-1}}\\
\{K_n(A/I^s;\bb Z/p^r)\}_s\ar@{-->}[dr]_{\exists\;\dlog^n_{A/I^\infty,\bb Z/p^r}}
\ar[ur]^-{tr} &&\{W_s\Omega_{A/I^s}^n/p^r\}_s\ar[d]_{\{R^{s-r}\}_s}\ar[r]^{R-F}&\{W_{s-1}\Omega_{A/I^s}^n/p^r\}_s\ar[d]_{\{R^{s-1-r}\}_s}\\
&\{W_r\Omega_{A/I^s,\sub{log},\sub{Zar}}^n\}_s\ar@{^(->}[r]&\{W_r\Omega_{A/I^s}^n\}_s\ar[r]^{R-F}&\{W_{r-1}\Omega_{A/I^s}^n\}_s
}\]
The composition of the four maps from $\{K_n(A/I^s;\bb Z/p^r)\}_s$ to $\{W_r\Omega_{A/I^s}^n\}_s$ has image inside \[\{R(\ker(W_{r+1}\Omega_{A/I^s}^n/p^r\xto{R-F}W_r\Omega_{A/I^s}^n))\}_s\] (because the projection $\{R^{s-r}\}_s$ factors through $\{W_{r+1}\Omega_{A/I^s}^n/p^r\}_s$), which is contained inside $\{W_r\Omega_{A/I^s,\sub{log},\sub{Zar}}^n\}_s$ by Corollary \ref{corollary_log_vs_R-F}(i). Therefore there exists a unique dashed arrow making the diagram commute.

Now we must show that $\dlog^n_{A/I^\infty,\bb Z/p^r}$ has all the properties in the statement of Theorem~\ref{theorem_existence}:

(i), (iii), \& (iv): $\dlog_{A/I^\infty,\bb Z/p^r}^n$ is natural in the pair $(A,I)$, compatible as $r\to\infty$, and multiplicative because the same is true of all the maps $K_n\xto{tr}\TC_n^s\to\TR_n^s\xto{R}\TR_n^{s-1}$ and $\lambda:W_r\Omega^n\to\TR_n^r$. The only one of these assertions which is not completely standard is that the trace map is multiplicative, but this was proved by Geisser and Hesselholt \cite[Corol.~6.4.1]{GeisserHesselholt1999}.

(ii): By (iv) it is enough to consider the case $n=1$, which is a consequence of \cite[Corol.~6.4.1]{GeisserHesselholt1999} and its proof (note that their $B$ corresponds to the differential on the de Rham--Witt complex and that their $\underbar{\;\;}$ is the Teichm\"uller map).

(v): It follows from (ii) that the maps $\dlog^n_{A,\bb Z/p^r}, \dlog^n_{A/0^\infty,\bb Z/p^r}:K_n(A)\to W_r\Omega^n_{A,\sub{log},\sub{Zar}}$ agree on symbols; but $K_n(A;\bb Z/p^r)=K_n(A)/p^r$ is generated by symbols, by Geisser--Levine (see Theorem \ref{theorem_GL_BKG}), and so the maps agree in general.

(vi): It follows from naturality that the given square commutes if $\dlog_{\hat A,\bb Z/p^r}^n$ is replaced by $\dlog_{\hat A/0^\infty,\b\ Z/p^r}^n$; but we have just shown in (v) that these maps are equal.
\end{proof}

\begin{proof}[Proof of Theorem \ref{theorem_mccarthy}]
Let $A$ be a regular, F-finite $\bb F_p$-algebra and $I\subseteq A$ an ideal such that $A/I$ is regular. Recall that $A/I$ is formally smooth over $\bb F_p$, by Lemma \ref{lemma_F-finite_formal_smooth}, and so the quotient maps $A/I^s\to A/I$ have compatible splittings for all $s\ge 1$. By comparing the main diagram in the previous proof to the analogous diagram for the regular $\bb F_p$-algebra $A/I$ itself (which is standard; otherwise just apply our construction to $A/I$ with the zero ideal), one obtains an analogous commutative diagram of relative theories:
\[\hspace{-2.3cm}\xymatrix@C=4mm{
& \{\TC^s_n(A/I^s,A/I;p,\bb Z/p^r)\}_s\ar[r]^{(\dag)} & \{\TR^s_n(A/I^s,A/I;p,\bb Z/p^s)\}_s\ar[r]^{R-F}\ar[d]^\cong& \{\TR_n^{s-1}(A/I^s,A/I;p,\bb Z/p^r)\}_s\ar[d]^\cong\\
\{K_n(A/I^s,A/I;\bb Z/p^r)\}_s\ar[dr]_{\dlog^n_{(A/I^\infty,A/I),\bb Z/p^r}}
\ar[ur]^-{tr}_\cong &&\{W_s\Omega_{(A/I^s,A/I)}^n/p^r\}_s\ar[d]_{\{R^{s-r}\}_s}\ar[r]^{R-F}&\{W_{s-1}\Omega_{(A/I^s,A/I)}^n/p^r\}_s\ar[d]_{\{R^{s-1-r}\}_s}\\
&\{W_r\Omega_{(A/I^s,A/I),\sub{log},\sub{Zar}}^n\}_s\ar@{^(->}[r]&\{W_r\Omega_{(A/I^s,A/I)}^n\}_s\ar[r]^{R-F}&\{W_{r-1}\Omega_{(A/I^s,A/I)}^n\}_s
}\]
The trace map is an isomorphism by Geisser--Hesselholt's strengthening of the McCarthy theorem \cite[Thm.~B]{GeisserHesselholt2011}. 

The relative form of Corollary \ref{corollary_log_vs_R-F}(iii), together with the surjectivity of $R-F$ in the relative nilpotent setting (Lemma \ref{lemma_surj_of_1-F_in_nilp_case}) and the coincidence of Zariski and \'etale logarithmic forms (Corollary \ref{corollary_Zar_vs_etale}(i)), implies that the sequence \[0\To \{W_r\Omega_{(A/I^sA/I),\sub{log},\sub{Zar}}^n\}_r\To\{W_r\Omega_{(A/I^s,A/I)}^n\}_r\xTo{R-F}\{W_{r-1}\Omega_{(A/I^s,A/I)}^n\}_r\To 0\] is exact for any $s\ge1$. By taking the diagonal over $r=s$ and noting that the resulting pro abelian groups have no $p$-torsion by Proposition \ref{proposition_filtrations_are_equal}, there is a similar short exact sequence mod $p^r$: \[0\To \{W_s\Omega_{(A/I^sA/I),\sub{log},\sub{Zar}}^n/p^r\}_r\To\{W_s\Omega_{(A/I^s,A/I)}^n/p^r\}_r\xTo{R-F}\{W_{s-1}\Omega_{(A/I^s,A/I)}^n/p^r\}_s\To 0\] 

In particular, the central $R-F$ in the diagram is surjective, which implies the same for the top $R-F$; since this holds for all $n\ge0$, we deduce that the long exact sequence which is implicit in the top row of the diagram breaks into short exact sequences and so arrow (\dag) is injective. Therefore we may add
\[\xymatrix{
\{\TC^s_n(A/I^s,A/I;p,\bb Z/p^r)\}_s\ar[d]^\cong\\
\{W_s\Omega_{(A/I^sA/I),\sub{log},\sub{Zar}}^n/p^r\}_r\ar[d]^{\{R^{s-r}\}}\\
\{W_r\Omega_{(A/I^s,A/I),\sub{log},\sub{Zar}}^n\}_s
}\]
to the diagram in such a way that it still commutes. But the lower vertical arrow $\{R^{s-r}\}$ occurring here is also an isomorphism, by applying Corollary \ref{corollary_de_Rham--Witt_log_diagonal} in the Zariski topology to both $\op{Spf}A$ and $\Spec A/I$.

It follows that $\dlog^n_{(A/I^\infty,A/I),\bb Z/p^r}$ is an isomorphism, as required to complete the proof.
\end{proof}

\begin{proof}[Proof of Theorem \ref{theorem_pro_GL}]
We will first prove the theorem in the case that $A/I$ is regular: so $A$ is still a regular F-finite $\bb F_p$-algebra, but $I\subseteq A$ is now an ideal such that $A/I$ is both local and regular. Then there is a commutative diagram of pro abelian groups in which the rows are short exact
\[\xymatrix{
0\ar[r] & \{K_n(A/I^s,I/I^s;\bb Z/p^r)\}_s \ar[r]\ar[d]^{\dlog^n_{(A/I^\infty,I/I^\infty),\bb Z/p^r}} & \{K_n(A/I^s;\bb Z/p^r)\}_s \ar[r]\ar[d]^{\dlog^n_{A/I^\infty,\bb Z/p^r}} & K_n(A/I;\bb Z/p^r)\ar[r]\ar[d]^{\dlog^n_{A,\bb Z/p^r}}&0\\
0\ar[r] & \{W_r\Omega_{(A/I^s,I/I^s),\sub{log},\sub{Zar}}^n\}_s \ar[r] & \{W_r\Omega_{A/I^s,\sub{log},\sub{Zar}}^n\}_s \ar[r] & \{W_r\Omega_{A/I,\sub{log},\sub{Zar}}^n\}_s\ar[r]&0
}\]
By Theorem \ref{theorem_mccarthy}, which we just proved, the left vertical arrow is an isomorphism. But the right vertical arrow is also an isomorphism, by the results recalled in Theorem \ref{theorem_GL_BKG}. Hence the central arrow $\dlog^n_{A/I^\infty,\bb Z/p^r}$ is an isomorphism.

Now suppose only that $A/I$ is merely gnc (but still local). We proceed by induction on its complexity, using Lemma \ref{lemma_complexity} to find ideals $J,J'\subseteq A$ such that $I$ has the same radical as $J\cap J'$, and such that $A/J$ is regular and that $A/J'$ and $A/J+J'$ have complexity strictly less than that of $A/I$. Then there is a diagram of pro abelian groups in which the two visible squares are commutative by naturality of our pro $\dlog$ map:
\[\hspace{-1cm}\xymatrix@C=3mm{
\cdots\ar[r] & \{K_n(A/I^s;\bb Z/p^r)\}_s\ar[r]\ar[d]^{\dlog^n_{A/I^\infty,\bb Z/p^r}} & \{K_n(A/J^s;\bb Z/p^r)\}_s\ar[d]^{\dlog^n_{A/J^\infty,\bb Z/p^r}\oplus \dlog^n_{A/{J'}^\infty,\bb Z/p^r}}\oplus \{K_n(A/{J'}^s;\bb Z/p^r)\}_s\ar[r] & \{K_n(A/(J+J')^s;\bb Z/p^r)\}_s\ar[r]\ar[d]^{\dlog^n_{A/(J+J')^\infty,\bb Z/p^r}}&\cdots\\
0\ar[r] & \{W_r\Omega_{A/I^s,\sub{log},\sub{Zar}}^n\}_s\ar[r] &\{W_r\Omega_{A/J^s,\sub{log},\sub{Zar}}^n\}_s\oplus \{W_r\Omega_{A/{J'}^s,\sub{log},\sub{Zar}}^n\}_s\ar[r]&\{W_r\Omega_{A/(J+J')^s,\sub{log},\sub{Zar}}^n\}_s\ar[r]&0
}\]
The long exact top row is a consequence of pro excision for algebraic $K$-theory of Noetherian rings \cite[\S2]{Morrow_pro_H_unitality}, while the bottom short exact sequence is Corollary \ref{corollary_pro_excision}.

By induction on the complexity and the already established regular case, the central and right vertical arrows are isomorphisms. Since this holds for all $n$, the long exact top row breaks into short exact sequences, and so the left vertical arrow is also an isomorphism, as desired.

Finally, since the composition \[\{K_n^M(A/I^s)\}_s\to\{K_n(A/I^s)\}_s\to \{K_n(A/I^s;\bb Z/p^r)\}_s\xto{\dlog^n_{A/I^\infty,\bb Z/p^r}}\{W_r\Omega_{A/I^s,\sub{log},\sub{Zar}}^n\}_s\] is given by $\dlog[\cdot]$, by Theorem \ref{theorem_existence}(ii), and hence is surjective, the fact that $\dlog^n_{A/I^\infty,\bb Z/p^r}$ is an isomorphism implies that the middle arrow is surjective. In light of the usual short exact sequence \[0\To K_n(A/I^s)/p^r\To K_n(A/I^s;\bb Z/p^r)\To K_{n-1}(A/I^s)[p^r]\To 0,\] this means that $\{K_{n-1}(A/I^s)\}_s$ has no $p$-torsion.
\end{proof}

\section{Further applications to $K$-theory}\label{section_applications}
In the remainder of the paper we apply the main theorems of Secton \ref{section_K_theory} to study a variety of questions in $K$-theory. The next four sections are largely independent of one another.

\subsection{Milnor vs Quillen $K$-theory of infinitesimal thickenings}\label{subsection_milnor_vs_quillen}
A conjecture of Beilinson predicts that the Milnor and Quillen $K$-theories of a field of characteristic $p$ (hence of any regular $\bb F_p$-algebra by a Gersten argument) agree rationally, i.e., that the indecomposable $K$-groups $K_n^\sub{ind}:=\op{coker}(K_n^M\to K_n)$ are torsion for such rings. In this section we will prove an infinitesimal form of this conjecture, by showing that the difference between Milnor and Quillen $K$-theory does not grow without bound, in the following sense:

\begin{theorem}\label{theorem_indecomposable}
Let $A$ be a regular, F-finite $\bb F_p$-algebra, and $I\subseteq A$ an ideal such that $A/I$ is gnc and local. Then, for all $n\ge0$, the square of pro abelian groups
\[\xymatrix{
\{K_n^M(A/I^s)\}_s\ar[r]\ar[d] & \{K_n(A/I^s)\}_s\ar[d]\\
K_n^M(A/I)\ar[r] & K_n(A/I)
}\]
is bicartesian up to obstructions killed by a power of $p$; in other words, the three pro cohomology groups of the sequence \[0\To\{K_n^M(A/I^s,I/I^s)\}_s\To \{K_n(A/I^s)\}_s\To K_n(A/I)\To 0\] are killed by a power of $p$.

The square is actually bicartesian (equivalently, the sequence is short exact) if $I$ is principal and $A/I$ is regular.
\end{theorem}
\begin{proof}
Note that the assertions really are the same, since they both state that the vertical arrows in the diagram are surjective (up to a power of $p$) and have isomorphic kernels (up to a power of $p$). Write $\res K_n(A/I^s,I/I^s):=\ker(K_n(A/I^s)\to K_n(A/I))$ and $\res K_n(A/I^s):=\Im(K_n(A/I^s)\to K_n(A/I))$.

We begin by proving left exactness of the short sequence, i.e., that the canonical map \[i_{(A/I^\infty,A/I)}:\{K_n^M(A/I^s,I/I^s)\}_s\to \{\res K_n(A/I^s,I/I^s)\}_s\] has kernel and cokernel killed by a power of $p$; in fact, we will show that they are killed by $p^N$, where $N$ is the $p$-adic valuation of $(n-1)!$. To do this we consider the canonical map of short exact sequences
\[\xymatrix{
0\ar[r] & K_n^M(A/I^\infty,I/I^\infty)\ar[d]^{i_{(A/I^\infty,A/I)}}\ar[r] & K_n^M(A/I^\infty)\ar[d]^{i_{A/I^\infty}}\ar[r] & K_n^M(A/I)\ar[d]^{i_{A/I}}\ar[r]&0\\
0\ar[r] & \res K_n(A/I^\infty,I/I^\infty)\ar[r] & K_n(A/I^\infty)\ar[r] & \res K_n(A/I^\infty)\ar[r]&0
}\]
and the associated long exact sequence of pro abelian groups
\[\hspace{-5mm}0\to \ker i_{(A/I^\infty,A/I)}\to \ker i_{A/I^\infty}\to \ker i_{A/I}\to\op{coker} i_{(A/I^\infty,A/I)}\to \op{coker} i_{A/I^\infty}\to \op{coker} i_{A/I}\to0\]

It is now convenient to introduce an endofunctor $-\otimes_{\bb Z}\bb Z/p^\infty$ of the category of pro abelian groups, defined by sending $\{M_s\}$ to $\{M_s/p^s\}_s$. This is exact (since the transition map $\Tor^\bb Z_1(M_{2s},\bb Z/p^{2s})\to\Tor^\bb Z_1(M_{s},\bb Z/p^{s})$ is zero), and has the additional property that if each abelian group $M_s$ is killed by a power of $p$ (possibly depending on $s$) then the natural map $\{M_n\}_s\to\{M_s/p^s\}_s$ is an isomorphism.

Since $i_{A/I^\infty}\otimes_\bb Z\bb Z/p^\infty$ is easily seen to be surjective by Corollary \ref{corollary_mccarthy}, it follows that $\op{coker} i_{A/I^\infty}$, hence also $\op{coker} i_{A/I}$, are zero after applying $-\otimes_\bb Z\bb Z/p^\infty$; therefore the map \begin{equation}(\ker i_{A/I})\otimes_\bb Z\bb Z/p^\infty\to(\op{coker} i_{(A/I^\infty,A/I)})\otimes_\bb Z\bb Z/p^\infty\label{eqn_p_infty}\end{equation} is surjective. But $\ker i_{A/I}$ is killed by $(n-1)!$, thanks to the existence of the usual Chern class from Quillen to Milnor $K$-theory; so both the left and right side of the previous line are killed by $p^N$.

However, $\op{coker} i_{(A/I^\infty,A/I)}$ was unchanged by applying $-\otimes_\bb Z\bb Z/p^\infty$ since each group $\res K_n(A/^s,A/I^s)$, hence also $\op{coker} i_{(A/I^s,A/I)}$, is killed by a power of $p$ \cite[Thm.~A]{GeisserHesselholt2011}. In conclusion, $\op{coker} i_{(A/I^\infty,A/I)}$ is killed by $p^N$.

The same observations show that $\ker i_{(A/I^\infty,A/I)}=(\ker i_{(A/I^\infty,A/I)})\otimes_\bb Z\bb Z/p^\infty$ is also killed by $(n-1)!$, hence by $p^N$. This completes the proof that the kernel and cokernel of $i_{(A/I^\infty,A/I)}$ are killed by $p^N$.

(Aside: In this parenthetical paragraph we treat the case that $I$ is principal and $A/I$ is regular. Then $i_{A/I}\otimes_\bb Z\bb Z/p^\infty$ is an isomorphism by Theorem \ref{theorem_GL_BKG}, and so its kernel is zero; by the surjection in line (\ref{eqn_p_infty}), we deduce that $\op{coker} i_{(A/I^\infty,A/I)}=0$. Next, after $I$-adically completing $A$, we may assume that $A=R[[t]]$ where $R:=A/I$; we will recall a result of R\"ulling--Saito in Theorem \ref{theorem_BDI_RS} which provides a (non-canonical) isomorphism $\{\bb W_{s-1}\Omega_{R}^{n-1}\}_s\isoto \{K^M_n(A/I^s,I/I^s)\}_s$, and this is $p$-torsion-free by the coincidence of the canonical and $p$-filtrations for the Hodge--Witt groups of the regular $\bb F_p$-algebra $R$ (see the paragraph after Definition \ref{definition_filtrations}). It follows that $\ker i_{(A/I^\infty,A/I)}$, which we already know is killed by $p^N$, must be zero. In conclusion, $i_{(A/I^\infty,A/I)}$ is an isomorphism; finally, $K_n(A/I^s)\to K_n(A/I)$ is surjective since $A/I^s\to A/I$ has a section.)

It remains to show that the map $K_n(A/I^\infty)\to K_n(A/I)$ has cokernel killed by a power of $p$. This is clear if $A/I$ is regular, since then the quotient maps $A/I^s\to A/I$ have sections from the formal smoothness of $A/I$ over $\bb F_p$  (Lemma~\ref{lemma_F-finite_formal_smooth}). We now proceed by induction on the complexity of $A/I$. Using Lemma \ref{lemma_complexity}, let $J,J'\subseteq A$ be ideals such that $I$ has the same radical as $J\cap J'$, and such that $A/J$ is regular and that $A/J'$ and $A/J+J'$ have complexity strictly less than that of $A/I$. We consider the following two excision squares:
\[\xymatrix{
K(A/I^\infty)\ar[r]\ar[d]&K(A/J^\infty)\ar[d]\\
K(A/J'^\infty)\ar[r]&K(A/(J+J')^\infty)
}\qquad
\xymatrix{
K(A/I)\ar[r]\ar[d]&K(A/J)\ar[d]\\
K(A/J)\ar[r]&K(A/(J+J'))
}
\]
The left square of pro spectra is homotopy cartesian by pro excision for algebraic $K$-theory, while the right square is homotopy cartesian up to power of $p$ (i.e., the birelative $K$-groups describing the obstruction to being homotopy cartesian are killed by a power of $p$) by \cite[Thm.~C]{GeisserHesselholt2011}. Considering the Mayer--Vietoris sequences associated to these diagrams, as well to as to the analogous square of relative groups, and henceforth working in the category of pro abelian groups modulo those killed by a power of $p$, we arrive at a commutative diagram of pro abelian groups
\[\hspace{-15mm}\xymatrix@C=2mm@R=5mm{
&\vdots\ar[d]&\vdots\ar[d]&\vdots\ar[d]&\\
\cdots\ar[r] &K_n(A/I^\infty,I/I^\infty)\ar[r]\ar[d] & K_n(A/J^\infty,J/J^\infty)\oplus K_n(A/J'^\infty,J'/J'^\infty)\ar[r]\ar[d] & K_n(A/(J+J')^\infty,(J+J')/(J+J')^\infty)\ar[d]\ar[r]&\cdots\\
\cdots\ar[r]& K_n(A/I^\infty)\ar[r]\ar[d] & K_n(A/J^\infty)\oplus K_n(A/J'^\infty)\ar[r]\ar[d] & K_n(A/(J+J')^\infty)\ar[d]\ar[r]&\cdots\\
\cdots\ar[r] &K_n(A/I)\ar[r] \ar[d]& K_n(A/J)\oplus K_n(A/J')\ar[r]\ar[d] & K_n(A/(J+J'))\ar[d]\ar[r]&\cdots\\
&\vdots&\vdots&\vdots\\
}\]
with exact rows and columns. The central and right columns break into short exact sequences by the inductive hypothesis. We claim that also the top row breaks into short exact sequences, whence a diagram chase will show that the left column also breaks into short exact sequences, thereby completing the inductive step and the proof.

It remains to prove the claim about the top row. Using the main result of the first half of the proof, it is necessary and sufficient to show that the map \[K_n^M(A/J^\infty,J/J^\infty)\oplus K_n^M(A/J'^\infty,J'/J'^\infty)\To K_n^M(A/(J+J')^\infty,A/(J+J')) \] is surjective for all $n\ge 0$. We prove this separately in the next remark since it holds in greater generality.
\end{proof}

\begin{remark}
If $A$ is a local ring, and $J,J'\subseteq A$ are ideals, then we show here that the canonical map \[K_n^M(A/J^s,J/J^s)\oplus K_n^M(A/J'^s,J'/J'^s)\To K_n^M(A/(J^s+J'^s),(J+J')/(J^s+J'^s))\] is surjective for all $n,s\ge0$.

Indeed, it is well-known that the relative Milnor $K$-group on the right is generated by symbols where at least one term is a unit of $A/(J'^s+J^s)$ which is congruent to $1$ modulo $(J+J')/(J^s+J'^s)$. Hence it is sufficient to prove the surjectivity assertion when $n=1$. Direct verification shows that the sequence \[0\To A/(J\cap J')^\times\To A/J^\times\oplus A/J'^\times\To A/(J+J')^\times\To 0\] is exact; replacing $J,J'$ by $J^s,J'^s$ gives a second short exact sequence, and the kernel  of the canonical surjection between the two sequences is therefore also short exact:
\begin{align*}
\hspace{-2.3cm}0\To K_1^M(A/(J^s\cap J'^s),(J\cap J')/(J^s\cap J'^s))\To &K_1^M(A/J^s,J/J^s)\\&\oplus K_1^M(A/J'^s,J'/J'^s)\To K_1^M(A/(J^s+J'^s),(J+J')/(J^s+J'^s))\To 0\end{align*}
This completes the proof.
\end{remark}

\begin{corollary}\label{corollary_principal_case}
Let $A$ be a regular, F-finite $\bb F_p$-algebra, and $I\subseteq A$ a principal ideal such that $A/I$ is regular and local. Then the canonical map \[\{K_n^M(A/I^s)/p^r\}_s\To \{K_n(A/I^s)/p^r\}_s\] is an isomorphism for all $n,r\ge0$.
\end{corollary}
\begin{proof}
Applying $-\otimes_\bb Z\bb Z/p^r$ to the bicartesian square of the previous theorem obtains a bicartesian diagram in which the bottom horizontal arrow is an isomorphism, by Theorem \ref{theorem_GL_BKG}; it follows that the top row is also an isomorphism.
\end{proof}

\begin{corollary}
Let $A$ be a regular, F-finite $\bb F_p$-algebra, and $I\subseteq A$ an ideal such that $A/I$ is gnc and local. Then the map of pro abelian groups $\{K_n^\sub{ind}(A/I^s)\}_s\to K_n^\sub{ind}(A/I)$ has kernel and cokernel killed by a power of $p$.
\end{corollary}
\begin{proof}
This follows by taking cokernels of the horizontal maps of the bicartesian (up to a power of $p$) square of the previous theorem.
\end{proof}

\begin{remark}
It is plausible that the previous corollary holds without the assumption that $I$ is principal. By imitating the proof in the principal case, it is sufficient to prove either of the following equivalent (by the proof of Theorem \ref{theorem_indecomposable}) statements, in which $R$ is a regular, local, F-finite $\bb F_p$-algebra, $A:=R[t_1,\dots,t_c]$, and $I:=(t_1,\dots,t_c)$:
\begin{enumerate}
\item the pro abelian group $\{K_n^M(A/I^s,I/I^s)\}_s$ is $p$-torsion free;
\item the map of pro abelian groups $\{K_n^M(A/I^s,I/I^s)\}_s\to \{K_n(A/I^s,I/I^s)\}_s$ is injective (hence an isomorphism by Corollary \ref{corollary_mccarthy})
\end{enumerate}
It is likely that condition (i) can be directly verified, but we have not seriously considered the problem.
\end{remark}

\subsection{Curves on $K$-theory}\label{subsection_Milnor}
In this section we consider Bloch's curves on $K$-theory and his original description of the de Rham--Witt complex in terms of $K$-groups.

For a moment, let $R$ be any (commutative) ring. We must recall the sense in which multiplication by the symbol $\{t\}$ is interpreted on the $K$-groups of $R[[t]]$ modulo powers of $t$; a useful reference may be \cite{Morrow_K2}. The Dennis--Stein--Suslin--Yarosh map $\rho_t:1+tR[[t]]\to K_2(R[[t]])$ is defined by $\rho_t(1+ft):=\langle f,t\rangle$, where the latter element is a Dennis--Stein symbol; standard properties of Dennis--Stein symbols show that $\rho_t$ is a homomorphism. If $R$ is local (as it will be in our cases of interest), so that either $f$ or $1+f$ is a unit, then $\rho_t$ is described in terms of Steinberg symbols as follows:
\[\rho_t(1+ft)=\begin{cases}\{-f,1+ft\}&f\in R[[t]]^\times\\ \left\{-\frac{1+f}{1-t},\frac{1+ft}{1-t}\right\}& 1+f\in R[[t]]^\times\end{cases}\]
Moreover $\rho_t$ fits into a commutative diagram
\[\xymatrix{
(1+tR[[t]])\ar@{^(->}[r]\ar[d]_{\rho_t}&R[[t]]^\times\ar[d]^{\{\cdot,t\}}\\
K_2(R[[t]],(t))\ar[r] &K_2(R((t)))
}\]
(the bottom horizontal arrow is injective if $R$ is local, regular and contains a field by the Gersten conjecture; or if $R$ is local and its residue field has $>5$ elements by \cite[Eg.~4.4(2)]{Morrow_K2}; probably it is always injective), and hence it may be thought of as ``right multiplication by $\{t\}$''. Finally, it is clear from the definition in terms of Dennis--Stein symbols that $\rho_t(1+t^sR[[t]])$ vanishes in $K_2(R[[t]]/t^s)$, and hence $\rho_t$ induces \[\res\rho_t:1+tR[t]/1+t^rR[t]\To K_2(R[t]/t^s,(t))\subseteq K_2(R[t]/t^s).\] Denoting by \[\gamma:\bb W_{s-1}(R)\isoto 1+tR[t]/1+t^sR[t],\quad\sum_{i=1}^{s-1}V_i[r_i]\mapsto\prod_{i=1}^{s-1}(1+r_it^i) \] the usual (up to normalisation) isomorphism of groups, where the left side denote a big Witt group, we now recall the role played by $\rho_t$ in relating curves on $K$-theory and the de Rham--Witt complex:

\begin{theorem}[Bloch--Deligne--Illusie, R\"ulling--Saito]\label{theorem_BDI_RS}
Let $R$ be a regular, local ring containing a field, and fix $n\ge 1$. Then there is an isomorphism of abelian groups $\gamma_n:\bb W_{s-1}\Omega_R^{n-1}\isoto K^M_n(R[t]/t^s,(t))$ for each $s\ge 1$ satisfying \[a\dlog[b_1]\cdots\dlog[b_{n-1}]\mapsto \{\gamma(a),b_1,\dots,b_{n-1}\}\] and \[da\dlog[b_1]\cdots\dlog[b_{n-2}]\mapsto -\res\rho_t(\gamma(a))\{b_1,\dots,b_{n-2}\}.\] If $R$ has characteristic $p$ then the resulting morphism of pro abelian groups \[\{\bb W_{s-1}\Omega_R^{n-1}\}_s\isoto \{K^M_n(R[t]/t^s,(t))\}_s\to \{K^\sub{sym}_n(R[t]/t^s,(t))\}_s\] is an isomorphism. 
\end{theorem}
\begin{proof}
The first assertion is a recent result of R\"ulling and Saito \cite[Thm.~4.13]{RullingSaito2015}. In characteristic $p$ the resulting composition $\{\bb W_{s-1}\Omega_R^{n-1}\}_s\to \{K^\sub{sym}_n(R[t]/t^s,(t))\}_s$ is the original comparison map of Bloch--Deligne--Illusie from the de Rham--Witt complex to the curves on $K$-theory, which was shown to be an isomorphism of pro abelian groups at the time \cite[II.\S5]{Illusie1979} (for some further discussion, including references for the case $p=2$, see the proof of \cite[Prop.~2.1]{Morrow_Birelative_dim1}).
\end{proof}

As well as needing the result of R\"ulling--Saito to treat the principal, regular case of Theorem \ref{theorem_indecomposable}, we have recalled the comparison map of Bloch--Deligne--Illusie to state the following curious consequence our main results; it is an inverse log/exp isomorphism between big Hodge--Witt groups and $p$-typical log Hodge--Witt groups:

\begin{corollary}\label{corollary_log_exp}
Let $R$ be a regular, local, F-finite $\bb F_p$-algebra, and fix $n,r\ge 1$. Then the composition \[\{\bb W_s\Omega_R^{n-1}\}_s\xto{\gamma_n}\{K_n^\sub{sym}(R[t]/t^s,(t))\}_s\xto{\dlog^n_{(R[[t]]/t^\infty,(t))}}\{W_r\Omega^n_{(R[t]/t^s,(t)),\sub{log}}\}_s\] is an isomorphism of pro abelian groups, and induces \[\{\bb W_s\Omega_R^{n-1}/p^r\}_s\isoto \{W_r\Omega^n_{(R[t]/t^s,(t)),\sub{log},\sub{Zar}}\}_s.\]
\end{corollary}
\begin{proof}
The first isomorphism is an immediate consequence of Corollary \ref{corollary_mccarthy} together with the isomorphism of Bloch--Deligne--Illusie recalled in Theorem \ref{theorem_BDI_RS}. The isomorphism modulo $p^r$ then follows from Theorem \ref{theorem_de_Rham--Witt_log}.
\end{proof}

\subsection{Continuity in characteristic $p$}\label{subsection_continuity_K}
The continuity problem in algebraic $K$-theory asks if the map $K(A)\to \op{holim_s}K(A/I^s)$ is a weak-equivalence, at least with finite coefficients, when $A$ is an $I$-adically complete ring. Omitting the early history (i.e., discrete valuation rings, and Gabber's rigidity theorem away from the residue characteristic) and the mixed characteristic results, this is know to be true if $A$ is a regular, local, F-finite $\bb F_p$-algebra and $I\subseteq A$ is an ideal such that $A/I$ is also regular by Geisser--Hesselholt \cite{GeisserHesselholt2006b} (note that these hypotheses imply that $A$ is a power series algebra over $A/I$, which is how Geisser--Hesselholt state their result); we improve this by allowing $A/I$ to be gnc:

\begin{theorem}\label{theorem_continuity}
Let $A$ be a regular, local, F-finite $\bb F_p$-algebra, and $I\subseteq A$ an ideal such that $A$ is $I$-adically complete and $A/I$ is gnc. Then the canonical maps \[K_n(A;\bb Z/p^r)\To \pi_n\op{holim}_sK_n(A/I^s;\bb Z/p^r)\To \projlim_sK_n(A/I^s;\bb Z/p^r)\] are isomorphisms for all $n\ge0$, $r\ge1$. 
\end{theorem}
\begin{proof}
Thanks to Theorem \ref{theorem_pro_GL} we know that $\projlim_s^1K_n(A/I^s;\bb Z/p^r)\cong\projlim_s^1W_r\Omega_{A/I^s,\sub{log},\sub{Zar}}^n$, which vanishes for all $n$ since the transition maps in the latter system are surjective; therefore $\pi_n\op{holim}_sK_n(A/I^s;\bb Z/p^r)\isoto \projlim_sK_n(A/I^s;\bb Z/p^r)$.

Combining Theorems \ref{theorem_GL_BKG} and \ref{theorem_pro_GL}, the remaining assertion to show is that the map \[W_r\Omega_{A,\sub{log},\sub{Zar}}^n\To \projlim_sW_r\Omega_{A/I^s,\sub{log},\sub{Zar}}^n\] is an isomorphism. We already proved this in Corollary \ref{lemma_continuity_of_log_HW}.
\end{proof}

Since we have never previously seen a continuity result for Milnor $K$-theory, we explicitly state the following consequence:

\begin{corollary}\label{corollary_continuity_Milnor}
Let $A$ be a regular, local, F-finite $\bb F_p$-algebra, and $I\subseteq A$ an ideal such that $A$ is $I$-adically complete and $A/I$ is gnc. Then, for each $n\ge 0$, $r\ge 1$, the canonical map \[K_n^M(A)/p^r\To\projlim_s(K_n^M(A/I^s)/p^r)\] is split injective with cokernel isomorphic to $\projlim_s\ker(K_n^M(A/I^s)/p^r\to K_n(A/I^s)/p^r)$. This cokernel vanishes either
\begin{enumerate}\itemsep0pt
\item if $n\le p$,
\item or if $I$ is principal and $A/I$ is regular.
\end{enumerate}
\end{corollary}
\begin{proof}
There is a commutative diagram with exact top row:
\[\xymatrix@C=4mm{
0\ar[r]& \ar[r]\projlim_s\ker(K_n^M(A/I^s)/p^r\to K_n(A/I^s)/p^r)\ar[r] & \projlim_s(K_n^M(A/I^s)/p^r)\ar[r] & \projlim_s(K_n(A/I^s)/p^r)\\
&&K_n^M(A)/p^r\ar[r]\ar[u]&K_n(A)/p^r\ar[u]
}\]
The bottom horizontal arrow is an isomorphism by Theorem \ref{theorem_GL_BKG} and the right vertical arrow is isomorphism by Theorem \ref{theorem_continuity} (recall from Theorem \ref{theorem_GL_BKG} that we know $K_n(A;\bb Z/p^r)=K_n(A)/p^r$ and $\{K_n(A/I^s;\bb Z/p^r)\}_s=\{K_n(A/I^s)/p^r\}_s$). A trivial diagram chase shows that the central vertical arrow is therefore split injective with the claimed cokernel. This cokernel vanishes if $n\le p$ by the existence of the Chern class from Quillen to Milnor $K$-theory, and if $I$ is principal and $A/I$ is regular by Corollary~\ref{corollary_principal_case}.
\end{proof}

\subsection{Class groups in Zariski and Nisnevich toplogies}\label{subsection_class_group}
If $X$ is a smooth, $d$-dimensional variety over a perfect field of characteristic $p$, then a standard consequence of Gersten's conjecture (or of the structure of $\cal K_n$ as a homotopy invariant presheaf with transfer) is that the canonical maps $H^i_\sub{Zar}(X,\cal K_n)\to H^i_\sub{Nis}(X,\cal K_{n,\sub{Nis}})$ are isomorphisms for all $i,n\ge 0$, and similarly for Milnor $K$-theory.

If now $Y\into X$ is a normal crossing divisor, then it was conjectured by Kato and Saito \cite[pg.~256]{KatoSaito1986}, as part of their higher dimensional class field theory, that the analogous maps \[\projlim_sH^i_\sub{Zar}(X,\cal K^M_{n,(X,Y_s)})\To \projlim_sH^i_\sub{Nis}(X,\cal K^M_{n,(X,Y_s),\sub{Nis}})\] would also be isomorphisms if the base field were finite and $i=n=d$, in which case the left and right side play the role of certain Zariski/Nisnevich class groups in their theory. A similar conjecture over general base fields was then raised in \cite[Qu.~IV]{KerzSaito2013}. The new theory of reciprocity sheaves \cite{KahnSaitoYamazaki2014} even predicts that $H^i_\sub{Zar}(X,\cal K^M_{n,(X,Y_s)})\isoto H^i_\sub{Nis}(X,\cal K^M_{n,(X,Y_s),\sub{Nis}})$ for each fixed $s\ge1$; this was established in the case that $Y$ is a smooth divisor by R\"ulling and Saito \cite[Corol.~2.29]{RullingSaito2015} (it follows easily from their result recalled in Theorem \ref{theorem_BDI_RS}, since the big de Rham--Witt sheaf has no higher Nisnevich cohomology on affines).

The goal of this section is to show that the following mod $p$-power version of Kato--Saito's conjecture is true:

\begin{theorem}\label{theorem_class_groups}
Let $X$ be a regular, F-finite $\bb F_p$-scheme, and $Y\into X$ a gnc closed subscheme; let $i,n,r\ge0$. Then the maps of pro abelian groups
\[\{H^i_\sub{Zar}(Y_s,\cal K_{n,Y_s}/p^r)\}_s\To \{H^i_\sub{Nis}(Y_s,\cal K_{n,Y_s,\sub{Nis}}/p^r)\}_s\]
\[\{H^i_\sub{Zar}(X,\cal K_{n,(X,Y_s)}/p^r)\}_s\To \{H^i_\sub{Nis}(X,\cal K_{n,(X,Y_s),\sub{Nis}}/p^r)\}_s\]
are isomorphisms.
\end{theorem}
\begin{proof}
It follows from the sheaf versions of Theorems \ref{theorem_GL_BKG} and \ref{theorem_pro_GL} that $\{\cal K_{n,Y_s}/p^r\}_s\isoto\{W_r\Omega_{(X,Y_s),\sub{log},\sub{Zar}}^n\}_s$ and $\{\cal K_{n,(X,Y_s)}/p^r\}_s\isoto\{W_r\Omega_{(X,Y_s),\sub{log},\sub{Zar}}^n\}_s$, and similarly in the Nisnevich topology. So the desired isomorphisms are exactly Corollary \ref{corollary_relative_nilpotent_hw}.\qedhere
\end{proof}

\begin{remark}[Pro Gersten vanishing]\label{remark_pro_Gersten}
Under the hypotheses of the previous theorem, it seems possible that $\{H^i_\sub{Zar}(Y_s,\cal K_{n,Y_s})\}_s$ vanishes whenever $i>n$, as an analogue of usual Gersten vanishing in the regular case. See Remark \ref{remark_pro_Gersten_1} concerning the $n=1$ case.
\end{remark}

\begin{remark}[Variations]
If $Y$ is actually regular in the previous theorem, then a modification of the proof shows that the maps of pro abelian groups
\begin{align}
\{H^i_\sub{Zar}(X,\cal K_{n,Y_s})\}_s&\To \{H^i_\sub{Nis}(X,\cal K_{n,Y_s,\sub{Nis}})\}_s \label{eqn_c1}\\
\{H^i_\sub{Zar}(X,\cal K_{n,(Y_s,Y)})\}_s&\To \{H^i_\sub{Nis}(X,\cal K_{n,(Y_s,Y),\sub{Nis}})\}_s \label{eqn_c2}
\end{align}
are isomorphisms for all $i,n\ge0$, and (using the proof of Theorem \ref{theorem_indecomposable}) that therefore the kernel and cokernel of
\begin{align}
\{H^i_\sub{Zar}(X,\cal K^M_{n,(Y_s,Y)})\}_s\To \{H^i_\sub{Nis}(X,\cal K^M_{n,(Y_s,Y),\sub{Nis}})\}_s \label{eqn_c3}
\end{align}
are killed by $p^N$, where $N$ is the $p$-adic valuation of $(n-1)!$.

If $Y$ is once again gnc, similar excision arguments to those used in the proof of Theorem \ref{theorem_indecomposable} then show that the kernel and cokernel of (\ref{eqn_c1})--(\ref{eqn_c3}) are still killed by a power of $p$. It also easily follows from the sheaf versions of Theorems \ref{theorem_GL_BKG} and \ref{theorem_pro_GL} that the canonical map $\{K^M_{n,(X,Y_s)}/p^r\}_s\to \{K_{n,(X,Y_s)}/p^r\}_s$ is injective and has kernel killed by $p^N$, and similarly in the Nisnevich topology; hence we deduce from the previous theorem that the kernel and cokernel of \[\{H^i_\sub{Zar}(X,\cal K^M_{n,(X,Y_s)}/p^r)\}_s\To \{H^i_\sub{Nis}(X,\cal K^M_{n,(X,Y_s),\sub{Nis}}/p^r)\}_s\] are killed by $p^N$.
\end{remark}

\subsection{The deformation of higher codimension cycles}\label{subsection_higher_codim}
Now we extend the results of Section \ref{subsection_deform_line} concerning line bundles to higher codimension cycles, as well as proving our infinitesimal Lefschetz theorem for Chow groups. We consider, for any $\bb F_p$-scheme $Y$, its {\em cohomological Chow group} \[\CH^n(Y):=H^n_\sub{Zar}(Y,\cal K_{n,Y}),\] where $\cal K_{n,Y}$ denotes as usual the Zariski sheafification of the $n^\sub{th}$ $K$-group presheaf on $Y$. If $Y$ is regular then Gersten's conjecture implies that $\CH^n(Y)$ is the usual Chow group of codimension-$n$ cycles modulo rational equivalence, but this is not true for general $Y$: in particular, $\CH^n(Y)$ is sensitive to infinitesimally thickening $Y$.

Assuming that $Y$ is regular, the composition \[\dlog^n_{r,Y}:\cal K_{n,Y}\To\cal K_{n,Y}/p^r\cong\cal K_{n,Y}^M/p^r\xTo{\dlog[\cdot]}W_r\Omega_{Y,\sub{log},\sub{Zar}}^n\] induces $c_{n,\sub{Zar}}:\CH^n(Y)\to H^n_\sub{Zar}(X,W_r\Omega_{Y,\sub{log},\sub{Zar}}^n)$, and composing with a change of topology map $H^n_\sub{Zar}(X,W_r\Omega_{Y,\sub{log},\sub{Zar}}^n)\to H^n_\sub{\'et}(X,W_r\Omega_{Y,\sub{log}}^n)$ defines the {\em \'etale-motivic cycle class map} \[c_n:\CH^n(Y)\To H^n_\sub{\'et}(X,W_r\Omega_{Y,\sub{log}}^n)=H^{2n}_\sub{\'et}(Y,\bb Z/p^r\bb Z(n))\] Letting $r\to\infty$ similarly defines \[c_n:\CH^n(Y)\To H^{2n}_\sub{\'et}(Y,\bb Z_p(n))\] (to be precise, this is the map on continuous Zariski hypercohomology induced by $\cal K_{n,Y}\xto{\dlog_{r,Y}^n} \{W_r\Omega^n_{Y,\sub{log},\sub{Zar}}\}_r\to \{R\ep_*W_r\Omega^n_{Y,\sub{log}}\}_r$).

Our main result in this section generalises Theorem \ref{theorem_deforming_line_bundles} by characterising whether a cycle deforms in terms of its \'etale-motivic cycle class:

\begin{theorem}\label{theorem_higher_codim}
Let $\cal Y$ be a regular, F-finite, formal $\bb F_p$-scheme whose reduced subscheme of definition $Y=Y_1$ is regular. Let $z\in \CH^n(Y)$. Then:
\begin{enumerate}
\item Given $r\ge 1$ there exist $t\ge p^r$ (depending only on $\cal Y$, not $z$) such that, if $c_n(z)\in H^{2n}_\sub{\'et}(Y,\bb Z/p^r\bb Z(n))$ lifts to $H^{2n}_\sub{\'et}(Y_t,\bb Z/p^r\bb Z(n))$ then $L$ lifts to $\CH^n(Y_{p^r})$.
\item $z$ lifts to $\projlim_s\CH^n(Y_s)$ if and only if $c_n(z)\in H^{2n}_\sub{\'et}(Y,\bb Z_p(n))$ lifts to $\projlim_sH^{2n}_\sub{\'et}(Y_s,\bb Z_p(n))$.
\end{enumerate}
\end{theorem}
\begin{proof}
The argument is similar to the proof of Theorem \ref{theorem_deforming_line_bundles}, using the results on $K$-theory from the previous section instead of Theorem \ref{theorem_pro_higher_Cartier}, and using Corollary \ref{corollary_Zar_to_etale} to overcome a new problem in passing between the Zariski and \'etale topologies. 

Thanks to the existence of the $\dlog$ map for formal schemes from Remark \ref{remark_dlog_formal_schemes}, the map $\dlog_{r,Y}^n:\cal K_{n,Y}^n\to W_r\Omega_{Y,\sub{log},\sub{Zar}}^n$ fits into a commutative diagram of pro Zariski sheaves on $Y$:

\[\xymatrix{
&0\ar[d]&0\ar[d]&0\ar[d]\\
0\ar[r] &\{\cal K_{n,(Y_s,Y)}\}_s\ar[r]^{p^r}\ar[d] & \{\cal K_{n,(Y_s,Y)}\}_s\ar[d] \ar[r] & \{W_r\Omega^n_{(Y_s,Y),\sub{log},\sub{Zar}}\}_s\ar[r]\ar[d] & 0\\
0\ar[r] &\{\cal K_{n,Y_s}\}_s\ar[r]^{p^r}\ar[d] & \{\cal K_{n,(Y_s,Y)}\}_s\ar[d] \ar[r]^{\dlog^n_{r,\cal Y}} & \{W_r\Omega^n_{\cal Y_s,\sub{log},\sub{Zar}}\}_s\ar[r]\ar[d] & 0\\
0\ar[r] &\cal K_{n,Y}\ar[r]^{p^r}\ar[d] & \cal K_{n,Y}\ar[d] \ar[r]^{\dlog^n_{r,Y}} & W_i\Omega^n_{Y,\sub{log},\sub{Zar}}\ar[r]\ar[d] & 0\\
&0&0&0
}\]

Each vertical sequence is exact; the bottom row is exact by the results recalled in Theorem \ref{theorem_GL_BKG}; the middle row, hence the top row, is exact by the formal scheme version of Theorem \ref{theorem_pro_GL} explained in Remark \ref{remark_dlog_formal_schemes}.

Taking Zariski cohomology and repeating the proof of Theorem \ref{theorem_deforming_line_bundles}(i) proves the following: given $r\ge 1$ there exist $t\ge p^r$ such that the boundary map $\delta:\CH^n(Y)=H^n_\sub{Zar}(Y,\cal K_{n,Y})\to H^{n+1}_\sub{Zar}(Y_{p^r},\cal K_{n,(Y_{p^r},Y)})$ kills all $z\in\CH^n(Y)$ with the property that $c_{n,\sub{Zar}}(z)$ lifts to $H^n_\sub{Zar}(Y_t,W_r\Omega_{Y_t,\sub{log},\sub{Zar}}^n)$. (In fact, to repeat the argument in the proof of Theorem \ref{theorem_deforming_line_bundles}(i), one further observation is required: possibly after increasing $t$, we can arrange that the map $H^{n+1}_\sub{Zar}(Y_t,\cal K_{n,(Y_t,Y)})\to H^{n+1}_\sub{Zar}(Y_{p^r},\cal K_{n,(Y_{p^r},Y)})$ vanishes on multiples of $p^r$; this is true since we can pick $t\gg p^r$ such that the map $\cal K_{n,(Y_t,Y)}\to \cal K_{n,(Y_{p^r},Y)}$ has image in $\cal K_{n,(Y_{p^r},Y)}^\sub{sym}$, by the formal scheme version of Corollary \ref{corollary_mccarthy}, and the latter sheaf is killed by $p^r$.)

This implies (by the same diagram chase as in the proof of Theorem \ref{theorem_deforming_line_bundles}(i)), that if $c_{n,\sub{Zar}}(z)$ lifts to $H^n_\sub{Zar}(Y_t,W_r\Omega_{Y_t,\sub{log},\sub{Zar}}^n)$ then $z$ lifts to $\CH^n(Y_{p^r})$. To pass to the \'etale topology we consider the diagram with exact rows
\[\xymatrix{
H^n_\sub{Zar}(Y_t,W_r\Omega_{Y_t,\sub{log},\sub{Zar}}^n)\ar[r]\ar[d] & H^n_\sub{Zar}(Y,W_r\Omega_{Y,\sub{log},\sub{Zar}}^n) \ar[r]\ar[d] & H^{n+1}_\sub{Zar}(Y_t,W_r\Omega_{(Y_t,Y),\sub{log},\sub{Zar}}^n)\ar[d]^\cong\\
H^n_\sub{\'et}(Y_t,W_r\Omega_{Y_t,\sub{log}}^n)\ar[r] & H^n_\sub{\'et}(Y,W_r\Omega_{Y,\sub{log}}^n) \ar[r] & H^{n+1}_\sub{\'et}(Y_t,W_r\Omega_{(Y_t,Y),\sub{log}}^n)
}\]
in which the indicated isomorphism is Corollary~\ref{corollary_Zar_to_etale}. It follows at once that an element in the top middle of the diagram (e.g., $c_{n,\sub{Zar}}(z)$) lifts to the top left if and only if its image in the bottom middle (i.e., $c_n(z)$) lifts to the bottom left. This completes the proof of part (i).

To prove (ii), we again proceed as in Theorem \ref{theorem_deforming_line_bundles} by assembling the first diagram into one of pro sheaves indexed over the diagonal $r=s$ and taking continuous cohomology, to obtain a bicartesian diagram of pro Zariski sheaves on $Y$
\[\xymatrix@C=1cm{
\{\cal K_{n,Y_s}\}\ar[d]\ar[r]^{\dlog_{r,\cal Y}^n} & \{W_s\Omega_{Y_s,\sub{log},\sub{Zar}}^n\}_s\ar[d]\\
\cal K_{n,Y}\ar[r]_{\dlog_{r,Y}^n}&\{W_s\Omega_{Y,\sub{log},\sub{Zar}}^n\}_s
}
\]
But Corollary \ref{corollary_Zar_to_etale} implies that the change of topology square 
\[\xymatrix@C=1cm{
\{W_s\Omega_{Y_s,\sub{log},\sub{Zar}}^n\}_s\ar[d]\ar[r] & \{R\ep_*W_s\Omega_{Y_s,\sub{log}}^n\}_s\ar[d]\\
\{W_s\Omega_{Y,\sub{log},\sub{Zar}}^n\}_s\ar[r]&\{R\ep_*W_s\Omega_{Y,\sub{log}}^n\}_s
}
\]
is a homotopy cartesian square of pro Zariski sheaves on $Y$, and so concatenation shows that
\begin{equation}\xymatrix@C=1cm{
\{\cal K_{n,Y_s}\}\ar[d]\ar[r] & \{R\ep_*W_s\Omega_{Y_s,\sub{log}}^n\}_s\ar[d]\\
\cal K_{n,Y}\ar[r]&\{R\ep_*W_s\Omega_{Y,\sub{log}}^n\}_s
}
\label{eqn_heart}\end{equation}
is also homotopy cartesian (this final square is the heart of the proof and future applications).

We now continue just as in the proof of Theorem \ref{theorem_deforming_line_bundles}(ii): taking continuous cohomology yields the following diagram of abelian groups with exact columns:

\[\xymatrix{
\vdots\ar[d]&\vdots\ar[d]\\
H^n_\sub{Zar}(Y,\{\cal K_{n,(Y_s,Y)}\}_s)\ar[d]\ar[r]^-\cong& H^n_\sub{\'et}(Y,\{W_r\Omega_{(Y_s,Y),\sub{log}}^n\}_s)\ar[d]\\
H^n_\sub{Zar}(Y,\{\cal K_{n,Y_s}\}_s)\ar[r]\ar[d] & H^n_\sub{\'et}(Y,\{W_s\Omega_{Y_s,\sub{log}}^n\}_s)\ar[d]\\
\CH^n(Y)\ar[r]^-{c_{n}}\ar[d] & H^n_\sub{\'et}(Y,\{W_s\Omega_{Y,\sub{log}}^n\}_s)\ar[d]\\
H^{n+1}_\sub{Zar}(Y,\cal K_{n,(Y_s,Y)})\ar[d]\ar[r]_-\cong& H^{n+1}_\sub{\'et}(Y,\{W_r\Omega_{(Y_s,Y),\sub{log}}^n\}_s)\ar[d]\\
\vdots&\vdots
}\]

As at the end of the proof of Theorem \ref{theorem_deforming_line_bundles}(ii), the middle vertical arrows on the left and right factor surjectively through $\projlim_s\CH^n(Y_s)$ and $\projlim_s H^n_\sub{\'et}(Y_s,\{W_r\Omega_{Y_s,\sub{log}}^n\}_r)$ (see also Remark \ref{remark_ses_in_continuous_cohomology}); then a diagram chase completes the proof.
\end{proof}

\begin{corollary}\label{corollary_variational_in_families}
Let $A$ be a Noetherian, F-finite $\bb F_p$-algebra which is complete with respect to an ideal $I\subseteq X$, and let $X$ be a proper scheme over $A$; assume that $X$ and the special fibre $Y:=X\times_AA/I$ are regular. For any $z\in\CH^n(Y)$, the following are equivalent:
\begin{enumerate}
\item $z$ lifts to $\projlim_s\CH^n(Y_s)$;
\item $c_n(z)\in H^{2n}_\sub{\'et}(Y,\bb Z_p(n))$ lifts to $\projlim_sH^{2n}_\sub{\'et}(Y_s,\bb Z_p(n))$;
\item $c_n(z)$ lifts to $H^{2n}_\sub{\'et}(X,\bb Z_p(n))$.
\end{enumerate}
\end{corollary}
\begin{proof}
The previous theorem immediately implies the equivalence of (i) and (ii), while the implication (iii)$\Rightarrow$(ii) is obvious. To prove (ii)$\Rightarrow$(iii), we will show that the map $H^{2n}_\sub{\'et}(X,\bb Z_p(n))\to \projlim_sH^{2n}_\sub{\'et}(Y_s,\bb Z_p(n))$ is surjective. This ``algebrisation'' result is the content of the next lemma, where we consider the problem in greater generality.
\end{proof}

\begin{lemma}\label{lemma_algebrisation_of_WOmega_log}
Let $A$ be a Noetherian, F-finite $\bb F_p$-algebra which is complete with respect to an ideal $I\subseteq A$ and let $X$ be a proper scheme over $A$. Then the canonical map \[H^i_\sub{\'et}(X,\bb Z_p(n))\To \projlim_sH^i_\sub{\'et}(X\times_AA/I^s,\bb Z_p(n))\] is surjective for all $i\ge0$.
\end{lemma}
\begin{proof}
Let $Y=X\times_AA/I$ be the special fibre of $X$, so that $Y_s=X\times_AA/I^s$. We claim that the canonical map of continuous cohomologies \[H^i_\sub{\'et}(X,\bb Z_p(n))=H^{i-n}_\sub{\'et}(X,\{W_r\Omega_{X,\sub{log}}^n\}_r)\To H^{i-n}_\sub{\'et}(Y,\{W_r\Omega_{Y_r,\sub{log}}^n\}_r)\] is an isomorphism for all $i\ge 0$; this is sufficient to complete the proof since the right group fits into a short exact sequence by Remark \ref{remark_ses_in_continuous_cohomology}:
\begin{align*}0\to {\projlim_s}^1\projlim_r H^{i-n-1}_\sub{\'et}(Y,W_r\Omega_{Y_s,\sub{log}}^n)\to H^{i-n}_\sub{\'et}(Y,\{W_r\Omega_{Y_r,\sub{log}}^n\}_r)\to &\projlim_sH^{i-n}_\sub{\'et}(Y,\{W_r\Omega_{Y_s,\sub{log}}^n\}_r)\to 0\\
&=\projlim_sH^{i}_\sub{\'et}(Y_s,\bb Z_p(n))
\end{align*}

To prove the claim in the shortest space, we note that it can be rewritten (using the definition of continuous cohomology) as a quasi-isomorphism \[\op{Rlim}_rR\Gamma_\sub{\'et}(X,W_r\Omega_{X,\sub{log}}^n)\quis \op{Rlim}_rR\Gamma_\sub{\'et}(Y_r,W_r\Omega_{Y_r,\sub{log}}^n)\] Since there is a fibre sequence \[\op{Rlim}_rR\Gamma_\sub{\'et}(X,W_r\Omega_{X,\sub{log}}^n)\To \op{Rlim}_rR\Gamma_\sub{Zar}(X,W_r\Omega_{X}^n)\xTo{1-F}\op{Rlim}_rR\Gamma_\sub{Zar}(X,W_r\Omega_{X}^n)\] and compatibly for $W_r\Omega_{Y_r,\sub{log}}^n$, by Corollary \ref{corollary_log_vs_R-F}, it is therefore sufficient to prove that \begin{equation}\op{Rlim}_rR\Gamma_\sub{Zar}(X,W_r\Omega_{X}^n)\To \op{Rlim}_rR\Gamma_\sub{Zar}(Y_r,W_r\Omega_{Y_r}^n)\label{equ_alg}\end{equation} is a quasi-isomorphism.

For each fixed $r\ge 1$, the scheme $W_r(X)$ (i.e., the topological space $X$ with structure sheaf $W_r(\roi_X)$) is proper over $W_r(A)$ \cite[App.]{LangerZink2004}, and $W_r\Omega_X^n$ is a coherent sheaf on it (finite generation was treated in Lemma \ref{lemma_finite_generation_conditions}, while behaviour under localisation is well-known); moreover, $W_r(A)$ is a Noetherian ring which is $W_r(I)$-adicially complete by \cite[Lem.~2.3]{Morrow_Dundas}. So Grothendieck's formal function implies that $R\Gamma_\sub{Zar}(X,W_r\Omega_X^n)\quis\op{Rlim}_sR\Gamma_\sub{Zar}(X,W_r\Omega_X^n\otimes_{W_r(\roi_X)}W_r(\roi_X)/W_r(I)^s)$. But the target of this map may be rewritten as $\op{Rlim}_sR\Gamma_\sub{Zar}(Y_s,W_r\Omega_{Y_s}^n)$ by Lemma \ref{lemma_Witt_dg_ideal_gen_by_I}(ii).

In conclusion, the left side of (\ref{equ_alg}) is quasi-isomorphic to $\op{Rlim}_r\op{Rlim}_sR\Gamma_\sub{Zar}(Y_s,W_r\Omega_{Y_s}^n)$, which is (quasi-isomorphic to) the right side of (\ref{equ_alg}). This completes the proof.
\end{proof}

\begin{remark}
If $Y$ is assumed to have codimension $1$ in $\cal Y$ (resp.~in $X$), then Theorem \ref{theorem_higher_codim} and Corollary \ref{corollary_variational_in_families} remain true if we replace Quillen by Milnor $K$-theory, by the bicartesian square of Theorem \ref{theorem_indecomposable}. In general the theorem and corollary remain true for Milnor $K$-theory up to an obstruction killed by a power of $p$, again by Theorem \ref{theorem_indecomposable}.
\end{remark}

\subsection{Infinitesimal weak Lefschetz for Chow groups}
We finish the paper by establishing the infinitesimal part of the weak Lefschetz theorem for Chow groups in characteristic $p$. Recall that if $Y$ is a smooth ample divisor on a smooth variety $X$ over a field, the weak Lefschetz conjecture for Chow groups predicts that the canonical map $\CH^n(X)_\bb Q\to\CH^n(Y)_\bb Q$ is an isomorphism if $2n<\dim X-1$. We will prove the following infinitesimal form of this conjecture, which over an algebraically closed field of characteristic zero is due to Patel--Ravindra \cite{PatelRavindra2014}:

\begin{theorem}\label{theorem_lefschetz}
Let $X$ be a smooth, projective, $d$-dimensional variety over a perfect field $k$ of characteristic $p$, and $Y\into X$ a smooth ample divisor. Then the canonical map \[\projlim_sH^i_\sub{Zar}(Y_s,\cal K_{n,Y_s})\To H^i_\sub{Zar}(Y,\cal K_{n,Y})\] has kernel and cokernel killed by a power of $p$ if $i+n<d-1$. In particular, if $2n<d-1$ then \[(\projlim_s\CH^n(Y_s))\otimes_{\bb Z}\bb Z[\tfrac1p]\Isoto\CH^n(Y)\otimes_{\bb Z}\bb Z[\tfrac1p]\]
\end{theorem}
\begin{proof}
We begin by claiming that, for each $r\ge 1$, the canonical map of pro hypercohomology groups \[\bb H^i_\sub{Zar}(X,W_r\Omega_X^\blob)\To\{\bb H^i_\sub{Zar}(Y_s,W_r\Omega_{Y_s}^\blob)\}_s\] is an isomorphism if $i<d-1$ and is injective if $i=d-1$. By filtering $W_r\Omega_X^\blob$ and $W_r\Omega_{Y_s}^\blob$ with the canonical filtration from Section \ref{subsection_filtrations}, and arguing inductively using the $5$-lemma, it is sufficient to prove, for each $j\ge0$, that \[\bb H^i_\sub{Zar}(X,\op{Fil}^jW_r\Omega_X^\blob/\op{Fil}^{j+1}W_r\Omega_X^\blob)\To\{\bb H^i_\sub{Zar}(Y_s,\op{Fil}^jW_r\Omega_{Y_s}^\blob/\op{Fil}^{j+1}W_r\Omega_{Y_s}^\blob)\}_s\] is an isomorphism if $i<d-1$ and is injective if $i=d-1$. By Corollary \ref{corollary_graded_pieces_are_de_rham} this previous map can be rewritten simply as \[\bb H^i_\sub{Zar}(X,\Omega_X^\blob)\To\{\bb H^i_\sub{Zar}(Y_s,\Omega_{Y_s}^\blob)\}_s.\] By na\"ively filtering the de Rham complex and making the same  $5$-lemma argument as before, this reduces our claim to proving that $\{H^i_\sub{Zar}(X,\Omega_{(X,Y_s)}^q)\}_s=0$ for $i\le d-1$ and all $q\ge0$.

Let $I\subseteq\roi_X$ be the ideal sheaf defining $Y$. The usual Leibnitz rule argument shows that $\{\Omega_{(X,Y_s)}^q\}_s\cong\{\Omega_X^q\otimes_{\roi_X}\cal I^s\}$, while coherent duality states that $H^i_\sub{Zar}(X,\Omega_X^q\otimes_{\roi_X}\cal I^s)$ is isomorphic to the dual of $H^{d-i}_\sub{Zar}(X,\ul{Hom}(\Omega_X^q,\roi_X)\otimes_{\roi_X}\Omega_X^d\otimes_{\roi_X}\cal I^{-s})$, which vanishes for $s\gg0$ since $\cal I^{-1}$ is ample by assumption. This completes the proof of the claim.

Passing to continuous cohomology over the diagonal $r=s$, the claim shows that the canonical map \[H^i_\sub{crys}(X/W(k))=\bb H^i_\sub{Zar}(X,\{W_s\Omega_X^\blob\}_s)\To \bb H^i_\sub{Zar}(Y,\{W_s\Omega_{Y_s}^\blob\}_s)\] is an isomorphism for $i<d-1$ and an injection for $i=d-1$. But weak Lefschetz for crystalline cohomology (as well as its finite generation) implies that the kernel and cokernel of $H^i_\sub{crys}(X/W(k))\to H^i_\sub{crys}(Y/W(k))=\bb H^i_\sub{Zar}(Y,\{W_s\Omega_Y^\blob\}_s)$ are killed by a power of $p$ for $i<d-1$ (and the kernel is killed by a power of $p$ if $i=d-1$), and so we finally deduce that the kernel and cokernel of \begin{equation}\bb H^i_\sub{Zar}(Y,\{W_s\Omega_{Y_s}^\blob\}_s)\To\bb H^i_\sub{Zar}(Y,\{W_s\Omega_Y^\blob\}_s)\label{eqn_lef}\end{equation} are killed by a power of $p$ for $i<d-1$ (and we can not conclude anything if $i=d-1$).

Corollary \ref{corollary_p_power_eigenspace} provides us with compatible long sequences, in which the failure of exactness is killed by a power of $p$,
\[\cdots \To H^{i-n}_\sub{\'et}(Y,\{W_s\Omega_{Y_s,\sub{log}}^n\}_s)\To \bb H^i_\sub{Zar}(Y,\{W_s\Omega_{Y_s}^\blob\}_s)\xTo{p^n-\phi} \bb H^i_\sub{Zar}(Y,\{W_s\Omega_{Y_s}^\blob\}_s)\To\cdots,\]
\[\cdots \To H^{i-n}_\sub{\'et}(Y,\{W_s\Omega_{Y,\sub{log}}^n\}_s)\To \bb H^i_\sub{Zar}(Y,\{W_s\Omega_{Y}^\blob\}_s)\xTo{p^n-\phi} \bb H^i_\sub{Zar}(Y,\{W_s\Omega_{Y}^\blob\}_s)\To\cdots \]
and comparing these via the $5$-lemma with (\ref{eqn_lef}) in mind shows that the map \begin{equation}H^i_\sub{\'et}(Y,\{W_s\Omega_{Y_s,\sub{log}}^n\}_s)\To H^i_\sub{\'et}(Y,\{W_s\Omega_{Y,\sub{log}}^n\}_s)\label{eqn_for_flat}\end{equation} has kernel and cokernel killed by a power of $p$ if $i+n<d-1$.

Finally, take continuous cohomology of the homotopy cartesian square (\ref{eqn_heart}) in the proof of Theorem \ref{theorem_higher_codim} (applied to the formal completion of $X$ along $Y$ of course) to obtain a homotopy cartesian square
\[\xymatrix{
\op{Rlim}_sR\Gamma_\sub{Zar}(Y_s,\cal K_{n,Y_s})\ar[d]\ar[r]& \op{Rlim}_sR\Gamma_\sub{\'et}(Y_s,W_s\Omega_{Y_s,\sub{log}}^n)\ar[d]\\
R\Gamma_\sub{Zar}(Y,\cal K_{n,Y})\ar[r]&R\Gamma_\sub{\'et}(Y,W_s\Omega_{Y,\sub{log}}^n)
}\]
Since the right vertical arrow has just been shown to induce an isomorphism (up to $p$-power torsion) on cohomology in degrees $<d-n-1$, the same is true of the left vertical arrow; i.e., the kernel and cokernel of $H^i_\sub{Zar}(Y,\{\cal K_{n,Y_s}\})\to H^i_\sub{Zar}(Y,\cal K_{n,Y})$ are killed by a power of $p$ if $i+n<d-1$. Since this maps always kills its subgroup $\projlim^1H^{i-1}_\sub{Zar}(Y,\cal K_{n,Y})$, and thus factors through the quotient $\projlim_sH^i_\sub{Zar}(Y_s,\cal K_{n,Y_s})$ (simply because the target is the continuous cohomology of a constant pro system), it follows that $\projlim^1H^{i-1}_\sub{Zar}(Y,\cal K_{n,Y})$ is killed by a power of $p$ and that the proof is complete.
\end{proof}

\begin{remark}
The previous theorem remains true if we replace Quillen by Milnor $K$-theory, by taking continuous cohomology of the sheaf version of the bicartesian square of Theorem~\ref{theorem_indecomposable}.
\end{remark}

\small
\bibliographystyle{acm}
\bibliography{../Bibliography}

\noindent Matthew Morrow\\
Mathematisches Institut\\
Universit\"at Bonn\\
Endenicher Allee 60\\
53115 Bonn, Germany\\
{\tt morrow@math.uni-bonn.de}

\end{document}